\documentclass[12pt,a4paper]{article}

\usepackage{fullpage}
\usepackage{amsthm,amssymb,bm}
\usepackage{graphicx,tikz,pgfplots,pgfplotstable,color}
\usepackage{color}
\usepackage{hyperref}
\hypersetup{linktocpage}
\usepackage{frcursive}
\usepackage{mathrsfs}
\usepackage{mathtools,stmaryrd}
\usepackage{algorithm}
\usepackage{caption,subcaption}
\usepackage{booktabs}

\newtheorem{theorem}{Theorem}[section] 
\newtheorem{proposition}[theorem]{Proposition}
\newtheorem{example}[theorem]{Example}
\newtheorem{definition}[theorem]{Definition}
\newtheorem{corollary}[theorem]{Corollary}
\newtheorem{lemma}[theorem]{Lemma}
\newtheorem{remark}[theorem]{Remark}

\DeclareMathOperator*{\argmin}{arg\, min}

\DeclareMathOperator{\dist}{dist}

\DeclareMathOperator{\vspan}{span}

\DeclareMathOperator{\Lip}{\Lip}
\let\div\relax
\DeclareMathOperator{\div}{div}

\DeclareMathOperator{\cddot}{\mathrm{:}} 

\newcommand{\cC}{{\cal C}}

\newcommand{\cE}{{\cal E}}

\newcommand{\cL}{\mathscr{L}}
\newcommand{\cM}{{\cal M}}

\newcommand{\cV}{{\cal V}}

\newcommand{\A}{{\mathbb A}}
\newcommand{\B}{{\mathbb B}}

\let\H\relax
\newcommand{\H}{\mathbb{H}}
\newcommand{\N}{\mathbb{N}}
\newcommand{\R}{\mathbb{R}}
\let\S\relax
\newcommand{\S}{\mathbb{S}}
\renewcommand{\P}{{\mathbb P}}

\newcommand{\td}{{\text d}}

\newcommand{\tin}{\text{ in }}
\newcommand{\ton}{\text{ on }}
\newcommand{\tex}{\text{ex}}

\newcommand{\tand}{\text{and}}

\newcommand{\imp}{\Rightarrow}

\newcommand{\bs}{\backslash}
\newcommand{\lra}{\longrightarrow}
\newcommand{\eq}{\Leftrightarrow}
\newcommand{\wt}{\widetilde}

\newcommand{\e}{{\varepsilon}}
\newcommand{\ph}{{\varphi}}
\newcommand{\g}[1]{\bm{#1}}
\newcommand{\ol}[1]{\overline{#1}}
\newcommand{\norm}[2]{\left\|{#1}\right\|_{#2}}
\newcommand{\norms}[2]{\|{#1}\|_{#2}}
\newcommand{\inner}[2]{\left<{#1}\right>_{#2}}

\newcommand{\red}[1]{\textcolor{red}{#1}}

\newcommand{\set}[2]{\left\{\left.{#1}\ \right|\ {#2}\right\}}

\pgfplotsset{compat=1.18}
\newcommand{\textscale}{5}
\newcommand{\imagescale}{5}

\def\graphscale{\imagescale*0.15}
\def\scale{\textscale*0.115}
\def\mainwidth{\imagescale/\textscale*4 cm}
\def\mainheight{\imagescale/\textscale*4 cm}
\def\xmin{-1}\def\xmax{1}\def\ymin{-1}\def\ymax{1}\def\pointsize{0.25}
\def\width{\mainwidth}\def\height{\mainheight}

\makeindex
\makeindex
\title{Multi-parameter identification in systems of PDEs from internal data}
\author{Elie BRETIN, Eliott KACEDAN, Laurent SEPPECHER} 

\begin{document}
\maketitle
\begin{abstract}
This article aims to present a general analysis of a class of inverse problems that consists in recovering the elliptic parameter maps in systems of PDEs, such as the linear elastic system, from the knowledge of some of their solutions. This identification problem is reformulated as a first-order linear system of the form $\nabla\g\mu + \g B \cdot \g\mu = F$, where $F$ and $\g B$ are tensor fields constructed from the data. A closed range property is proved, which induces $L^2$-stability estimates. We then characterize the null space by introducing the concept of conservative third-order tensor field.

Finally, a discretization based on the finite element method is proposed and some numerical examples show the efficiency of this approach to recover anisotropic elastic parameters from both static and dynamic solutions of the PDE system.
\end{abstract}






\tableofcontents

\section{Introduction}

In this work, we propose a general study of a class of inverse problems that consist in finding the elliptic parameters of a system of PDEs from the knowledge of some of its solutions inside a smooth connected domain $\Omega\subset \R^d$ with $d\geq 2$.

More precisely, we tackle the identification of the parameters of a fourth-order tensor field $\g C_{ijk\ell}(x)$ from the knowledge of some sources $\g f$ and corresponding solutions $\g u$ of the linear elastic system of equations 
\begin{equation}\label{eq:vectorial}
-\div(\g C\cddot \cE(\g u)) = \g f\quad\tin\Omega,
\end{equation}
where $\cE(\g u):=(\nabla \g u+(\nabla \g u)^T)/2$ is the strain matrix. We assume no knowledge on the unknown parameters at the boundary $\partial\Omega$.

The right-hand-side $\g f$ can take different forms, depending on the modeling. It can be either $\g f:=-\rho\partial_{tt}\g u$ or $\g f:=-\partial_{t}\g u$ in the case of hyperbolic or parabolic problems, $\g f:=\omega^2\g u$ for a time-harmonic wave equation, or even $\g f=\g 0$ for static or quasi-static problems. In any case, this term is assumed to be known in $\Omega$. 

The proposed analysis also applies to scalar PDE inverse parameter problems where one wants to recover a matrix field $C_{ij}(x)$ from the knowledge of some $f$ and $u$ solution of the scalar elliptic equation
\begin{align*}
	- \mathrm{div} (C \cdot \nabla u) = f \quad \textrm{ in }\Omega.
\end{align*}
Note that \eqref{eq:vectorial} can be generalized to the more general elliptic system $-\div(\g C\cddot \nabla\g u) = \g f$. In this paper we chose to present the analysis in the context of the linear elastic system  \eqref{eq:vectorial}, as it corresponds to the main application that we have in mind.

\subsection{Scientific context}

One of the main motivations for this work concerns elastography, an imaging modality that aims to reconstruct the mechanical properties of biological tissues. Local values of elastic parameters can be used as discriminating criteria to distinguish between healthy and diseased tissue \cite{sarvazyan1995biophysical,krouskop1998elastic}. This technique was developed at the end of the 90's with 1D ultrasound probes \cite{catheline1998}, where some shear waves generated by a transient impulse on the medium is recorded. It has been extended in the early 2000's to 2D images with the introduction of new devices and has been tested on both phantom data \cite{tanterFink2002} and \emph{in vivo data} \cite{tanterFink2003}, allowing the precise location of tumors from their mechanical properties. While several techniques of elastography exist (see, for example, \cite{gennisson2013ultrasound,parker2010imaging,doyley2012model,brusseau20082}), the most common approach is to use an auxiliary imaging method (such as ultrasound imaging \cite{deprez2011potential,brusseau2000axial}, magnetic resonance imaging, optical coherence tomography \cite{nahas2013supersonic}, \ldots) to measure the displacement field $\g u$ in a medium when a mechanical perturbation is applied. See \cite{sherina2021challenges} and inside references for recent advances on this topic. The inverse problem can be formulated as recovering the Lamé parameters $\mu$ and $\lambda$ in the isotropic linear elastic equation
\begin{equation*}
-\nabla\cdot(2\mu\cE(\g u))-\nabla(\lambda\div\g u) =  \g f\quad\tin\Omega.
\end{equation*}
This problem reads as \eqref{eq:vectorial} if one defines the tensor field $\g C:= 2\mu\g I+\lambda I\otimes I$.

Inverse elliptic multi-parameter problems also arise from other practical situations, such as the problem of identifying the electrical conductivity (see, for instance, \cite{uhlmannCalderon} for a survey of advances in this topic and \cite{ghosh2020calderon, munoz2020calderon} for more recent works) or thermal conductivity \cite{huang1995thermal} from the knowledge of some corresponding potentials or temperature fields in possibly anisotropic media \cite{ferreira2016calderon}.

Recent advances have been made on theoretical aspects of inverse problems posed by elastography. Notably, mathematical results related to uniqueness and stability have emerged under various assumptions related to the number of reconstructed parameters, the knowledge of boundary conditions, and the number of given displacement data. 

Some of the first results concern the reconstruction of a unique elasticity parameter in two dimensions. These have been reported in \cite{barbone2004} with some \emph{a priori} knowledge and two or more measurements and with the knowledge of either the shear or the first Lamé coefficient \cite{barbone2007elastic}.	

The authors in \cite{widlak2015,ammari2015stability} provide some stability estimates for the identification of the Lamé coefficients in the case of the isotropic elastic medium. More specifically, in \cite{widlak2015} the authors obtain stability results under various hypotheses in the context of  quasi-static, transient or time-harmonic elastography. Note also the results of \cite{bal2014,mclaughlin2004unique,mclaughlin2010calculating}, where the stability analysis is based on the properties of transport equations solved by the parameters, under specific assumptions on the displacement fields. However, most of these approaches require the knowledge of the elasticity maps at the boundary, as in \cite{bal2014} where the adjoint state method is used. Numerical simulations for the isotropic elastic problem using iterative methods \cite{ammari2015mathematical,babaniyi2017direct} or hybrid methods \cite{goksel2018} have been carried out and appear to be sufficiently stable to handle the case of high noise data.

The general reconstruction of the anisotropic elastic tensor has also been analyzed in \cite{bal2015reconstruction}, where the authors obtain a Lipschitz stability result under a minimal condition on the number of known displacement fields. Some numerical simulations in the anisotropic case have also been provided \cite{pierron2012virtual,barbone2010adjoint}, but require rather restrictive assumptions on the data.

More recently, in the case of a single elastic map, the work \cite{ammari2021direct} focuses on a stability result under weak assumptions of data regularity and without the need for additional information at the boundary. A discrete version of these results has also been adapted in \cite{bretin2023stability}, where a discrete \textit{inf-sup} condition is given, to guarantee the stability of the discrete inverse problem. An original finite element approach based on a hexagonal tilling is also derived to ensure stable reconstructions in practice. It should also be noted that this numerical method has also been used in the case of experimental data \cite{Brusseau_Seppecher1,Brusseau_Seppecher2}, allowing quantitative reconstructions of the elasticity maps.

The main objective of this paper is to propose a stability analysis in the multi-parameter case and
under minimal assumptions of regularity of the data fields. The proposed approach idea is to adapt the analysis from \cite{ammari2021direct}, to general anisotropic elastic tensors.

\subsection{The Reverse Weak Formulation}

In most applications, some \emph{a priori} knowledge on the tensor $\g C$ is available. Depending on the model, one can assume, for instance, classical symmetries, isotropy, specific anisotropy directions, etc. In the case of an isotropic linear elastic medium, the tensor $\g C$ admits the decomposition $\g C= 2\mu\g I+\lambda I\otimes I$ where $\g I$ is the fourth-order identity tensor, $I$ is the identity matrix, and $\lambda,\mu$ are the Lamé parameters. In more general cases, the unknown tensor $\g C$ admits a decomposition of the form
\begin{align}\label{eq:decomp}
\g C(x) :=\sum_{k=1}^{n}\mu_{k}(x)\g C_{k},\qquad x\in\Omega,
\end{align}
where the $\mu_{k}$ are some unknown scalar maps and the $\g C_{k}$ are known constant tensors. 

Without specific smoothness assumptions and boundary knowledge, the linear elastic system of equations \eqref{eq:vectorial} makes sense in $H^{-1}(\Omega,\R^{d})$.The inverse problem that we aim to solve is posed in the weak sense:
\begin{equation}\label{eq:RWF}\begin{aligned}
&\text{Find }\mu_1,\dots,\mu_n\in L^\infty(\Omega) \text{ such that} \\
&\int_\Omega \sum_{k=1}^n\mu_k\g (\g C_k\cddot\cE(\g u_\ell))\cddot (\nabla\g v)^T = \inner{\g f_\ell,\g v}{H^{-1},H^1_0},\quad \forall \g v\in H^{1}_0(\Omega,\R^d),\quad \forall \ell=1,\dots,m, 
\end{aligned}\end{equation}
from the knowledge of some pairs $(\g u_1,\g f_1),\dots, (\g u_m,\g f_m)\in H^1(\Omega,\R^d)\times H^{-1}(\Omega,\R^d)$. This is called the Reverse Weak Formulation (RWF) of the inverse parameters problem. For a more compact formulation, we denote 
\begin{align*}
& \g u:= \left[\begin{matrix}\g u_1 \\ \vdots \\ \g u_m \end{matrix}\right] \in H^1\left(\Omega,\R^{N}\right), 
\quad \g f:=\left[\begin{matrix}\g f_1 \\ \vdots \\ \g f_m \end{matrix}\right]\in H^{-1}\left(\Omega,\R^{N}\right),\quad\text{with } N=m\, d, \\
	&\bm{\mu} := \left[\begin{matrix}\mu_{1} \\ \vdots \\ \mu_{n} \end{matrix}\right]\in L^\infty(\Omega,\R^n),
	\quad \mathrm{and} \quad
	(\bm{T}_{\g u})_{ijk}:= \left[\begin{matrix}\g C_{k}\cddot\cE(\g u_1) \\ \vdots \\ \g C_{k}\cddot\cE(\g u_m) \end{matrix}\right]_{ij}. 
\end{align*}
Then \eqref{eq:RWF} is rewritten as
\begin{equation}\label{eq:RWF2} \begin{aligned}
	\text{Find }\g\mu\in L^\infty(\Omega,\R^n)\quad \text{ s.t. } \quad
	\int_\Omega (\g T_{\g u}\cdot \g \mu)\cddot (\nabla \g v)^T = \inner{\g f,\g v}{H^{-1},H^1_0}\quad \forall \g v\in H^{1}_0(\Omega,\R^{N}). 
\end{aligned} \end{equation}
The tensor field $\g T_{\g u}$ usually belongs to $L^2(\Omega,\R^{N\times d\times n})$ and is known from the data. This weak formulation defines a data-dependent first-order differential operator 
\begin{equation}\label{eq:AT}\begin{aligned}
A_{\g u} : L^\infty(\Omega,\R^n) &\to H^{-1}\left(\Omega,\R^N\right)\\
\g \mu &\mapsto-\div(\g T_{\g u}\cdot \g \mu)
\end{aligned}\end{equation}
and the inverse parameters problem simply reads:
\begin{equation}\label{eq:RWFcompact}
A_{\g u}\g\mu=\g f.
\end{equation}

\subsection{Paper outline}

The main objective of this article is to study the invertibility of \eqref{eq:RWFcompact} and to propose an efficient method to approach its solution. 

In Section \ref{sec:expanded}, we provide an expanded version of the equation $A_{\g u}\g\mu=\g f$ under the canonical form  $\nabla\bm{\mu} + \g B\cdot\bm{\mu} = F$, where $\g B$ is a third-order tensor field and $F$ is a matrix field, both depending on the measurements. This transformation requires a minimal smoothness hypothesis on the data and a ``left-invertibility'' hypothesis on tensor $\g T_{\g u}$ which is satisfied if the data is ``rich enough''.

From this expanded form of the problem, we are able to prove, in Section \ref{sec:stability}, that $A_{\g u}$ extends to a closed range operator from $L^2$ to $H^{-1}$. We naturally deduce some $L^2$-Lipschitz stability estimates that occur in the null space orthogonal $N(A_{\g u})^\perp$. 

In Section \ref{sec:kernel}, we study the null space $N(A_{\g u})$, which is also the null space of the first order differential operator 
\begin{equation}\label{eq:nablaB}
(\nabla+\g B\cdot):\g\mu\mapsto \nabla\g\mu+\g B\cdot\g\mu.
\end{equation}
under the hypotheses from Section \ref{sec:expanded}. We first show that the dimension is finite and lower than $n$ (the number of parameter maps to be recovered). Then, we propose a finer characterization of the null space with respect to the tensor field $\g B$. For this purpose, we introduce the notion of $k$-conservative tensor fields, which is, in a sense, a generalization of the notion of conservativity for vector fields to third-order tensor fields. 

Lastly, we provide in Section \ref{sec:num} a suitable discretization of \eqref{eq:RWFcompact} and illustrate the stability of the recovery of the tensor $\g C$ in various situations.

\section{The expanded formulation of the inverse problem}\label{sec:expanded}

We provide here some sufficient conditions in order to expand the equation $A_{\g u}\g \mu=\g f$ to
\begin{equation}\label{eq:expanded}
\nabla\g \mu + \g B\cdot\g \mu = F,
\end{equation}
where $\g B$ is a third-order tensor field and $F$ is a matrix field.

\begin{definition} For some positive integers, $n$, $d$, and $N$, we say that a third-order tensor $\g T\in\R^{N\times d\times n}$ admits a left-inverse $\g T^{-1}\in \R^{n\times d\times N}$ if $\g T^{-1}\cdot \g T=\g I$, \emph{i.e.}
\begin{equation}\nonumber
(\g T^{-1}\cdot \g T)_{ijk\ell} = \delta_{i\ell}\delta_{jk}. 
\end{equation}
\end{definition}

\begin{proposition}\label{prop:expanded}\ 
\begin{enumerate}
\item   If $\g T \in W^{1,p}\left(\Omega,\R^{N\times d\times n}\right)$ with $p>d$ then 
\begin{equation*}
-\div(\g T\cdot \g \mu)=\g f\quad\eq\quad \g T\cddot \nabla \g \mu + D \cdot  \g \mu   = -\bm{f},
\end{equation*}
where $D_{ik}:=\sum_{j}\partial_{j}\g T_{ijk}$.
\item Moreover, if for all $x\in\Omega$, $\g T(x)$ admits a left-inverse $\g T^{-1}(x)$ such that $\g T^{-1}\in W^{1,p}\left(\Omega,\R^{n\times d\times n}\right)$, then 
\begin{equation*}
-\div(\g T\cdot \g \mu)=\g f\quad\eq\quad \nabla \g \mu  + \g B \cdot  \g \mu  = F,
\end{equation*}
where $\g B:=\g T^{-1}\cdot D\in L^{p}\left(\Omega,\R^{n\times d\times n}\right)$ and $F:=-\g T^{-1}\cdot \g f\in  H^{-1}\left(\Omega,\R^{n\times d}\right)$. 
\end{enumerate}
\end{proposition}


\begin{proof} 1. The equation $-\div(\g T\cdot \g \mu)=\g f$ reads in the weak sense
\begin{equation}\nonumber
\int_\Omega (\g T\cdot\g \mu) \cddot  (\nabla \g v)^T = \inner{\g f,\g v}{H^{-1},H^1_0},\qquad \forall \g v\in H^1_0\left(\Omega,\R^N\right).
\end{equation}
Remark that  $\g T \in W^{1,p}$ with $p> d$ implies that $\g T\in L^\infty$ and $\g v\cdot \g T\in H^1_0$. We can write for all $\g v\in H^1_0(\Omega,\R^N)$,
\begin{equation}\nonumber\begin{aligned}
\inner{\g f,\g v}{H^{-1},H^1_0} &= \int_\Omega(\g T\cdot\g \mu) \cddot  (\nabla \g v)^T = \int_\Omega\sum_j\partial_j(\g v\cdot \g T)_{jk}\g \mu_k - \int_\Omega\g v\cdot D\cdot \g \mu,\\
&= - \inner{\nabla \g \mu,\g v\cdot \g T}{H^{-1},H^1_0} - \int_\Omega\g v\cdot D\cdot \g \mu,\\
&= - \inner{\g T\cddot\nabla \g \mu,\g v}{H^{-1},H^1_0} - \int_\Omega\g v\cdot D\cdot \g \mu.\\
\end{aligned}\end{equation}
Hence we have $\g T\cddot \nabla \g \mu + D \cdot  \g \mu   = - \g f$ in $H^{-1}(\Omega,\R^{N})$. \medskip

2. Let $V\in H^1_0(\Omega,\R^{n\times d})$. We use $\g v:=V\cddot \g T^{-1}\in H^1_0\left(\Omega,\R^N\right)$ as a test function in the previous computation. We get that
\begin{equation}\nonumber\begin{aligned}
\inner{\g f,\g v}{H^{-1},H^1_0} &= - \inner{\nabla \g \mu,\g v\cdot \g T}{H^{-1},H^1_0} - \int_\Omega\g v\cdot D\cdot \g \mu,\\
\inner{\g f,V\cddot \g T^{-1}}{H^{-1},H^1_0}  &= - \inner{\nabla \g \mu,V}{H^{-1},H^1_0} - \int_\Omega V\cddot \g T^{-1}\cdot D\cdot \g \mu,\\
\inner{\g T^{-1}\cdot \g f,V}{H^{-1},H^1_0}  &= - \inner{\nabla \g \mu,V}{H^{-1},H^1_0} - \int_\Omega V\cddot \g B\cdot \g \mu.\\
\end{aligned}\end{equation}
Hence we have $\nabla \g \mu  + \g B \cdot  \g \mu  = F$ in $H^{-1}\left(\Omega,\R^{n\times d}\right)$.
\end{proof}

\begin{example} To find the scalar conductivity $\sigma$ in dimension $2$ from two data sets  $(u_1,f_1)$ and $(u_2,f_2)$ satisfying of $-\div(\sigma \nabla u_i)=f_i$, $i=1,2$, we consider the tensor $(\g T_{\g u})_{ij1}:=(\partial_j u_i)$. If the solutions $u_1,u_2$ belong to $W^{2,p}$, then the inverse problem reads 
\begin{equation}\nonumber
\nabla\g u\cdot \nabla\sigma + \sigma \div \nabla\g u = -\g f\quad \tin H^{-1}(\Omega,\R^2).
\end{equation}
Moreover, if $(\partial_j u_i)$ is invertible everywhere, the problem reads:
\begin{equation}\nonumber
\nabla\sigma + \sigma (\nabla\g u)^{-1}\cdot \div \nabla\g u = - (\nabla\g u)^{-1}\g f\quad \tin H^{-1}(\Omega,\R^2),
\end{equation}
which is of the form \eqref{eq:expanded}. 
\end{example}

\begin{example} To find a single parameter $\mu$ from a single elastic displacement field $\g u$ solution of $-\div(\mu \cE(\g u))=\g f$, we consider $(\g T_{\g u})_{ij1}:=\cE(\g u)_{ij}$. Hence, if $\g u$ belongs to $W^{2,p}(\Omega)$ and if  $\cE(\g u)$ is invertible everywhere, we end up with the same equation as in the previous example, with $\cE(\g u)$ replacing $\nabla \g u$. 
\end{example}

\begin{example} To find the two Lamé parameters $\lambda,\; \mu$ from two data sets $(\g u_1,\g f_1)$ and $(\g u_2,\g f_2)$ satisfying the linear elastic system 
\begin{align*}
 -\div(2\mu\cE(\g u_i))-\nabla(\lambda\div \g u_i) = \g f_i,
\end{align*} 
we consider the tensor
\begin{equation}\nonumber
(\g T_{\g u})_{ij1}:=
\left[\begin{matrix}2 \cE(\g u_1) \\ 2\cE(\g u_2) \end{matrix}\right]_{ij},
\qquad
(\g T_{\g u})_{ij2}:=
\left[\begin{matrix}\div(\g u_1)I \\ \div(\g u_2)I \end{matrix}\right]_{ij}.
\end{equation}
It admits a left inverse everywhere if and only if the $4\times 4$ matrix 
\begin{equation}\nonumber
\left[\begin{matrix}2 \cE(\g u_1) & \div(\g u_1)I\\ 2\cE(\g u_2) & \div(\g u_2)I \end{matrix}\right]
\end{equation}
is invertible everywhere. This condition is identical to the one in \cite{bal2014} (Subsection 3.1, Lemma 1). In this case, the problem can be written under the form
\begin{equation}\nonumber
\nabla \g \mu  + \g B \cdot  \g \mu  = -\g T_{\g u}^{-1}\g f,
\end{equation}
where $\g\mu:=(\mu,\lambda)^T$. 
\end{example}

\begin{remark} \label{rq:number_parameters} We observe that if one has $n$ parameter maps to recover, a necessary condition to write the problem under the expanded form $\nabla \g \mu  + \g B \cdot  \g \mu  = F$ is to have at least $n$ solutions for vectorial elliptic problems and $n\times d$ solutions for scalar elliptic problems. This is necessary (but not sufficient) for the third-order tensor $\g T_{\g u}$ to be left-invertible everywhere. 
\end{remark}

\section{Closed range properties and $L^2$-stability estimates}\label{sec:stability}

In this section, we establish the closed range property of the operators $(\nabla + \g B\cdot)$ and $A_{\g u}$. From this, we immediately deduce some stability estimates in the orthogonal of null space of $A_{\g u}$.

\subsection{Closed range properties}

\begin{theorem}\label{theo:closedrange}
If $\g B \in L^p(\Omega, \R^{n \times d \times n})$ with $p>d$, then the operator $(\nabla +\g B\cdot ): L^2(\Omega,\R^n)\to H^{-1}(\Omega,\R^{n\times d})$ has closed range \emph{i.e.} there exists $c >0$ such that
  \begin{align}\label{eqtheo:closedrange}
      \forall \g \mu \in N(\nabla +\g B\cdot)^\perp, \qquad \| \g \mu \|_{L^2(\Omega)} \le c \|\nabla \g \mu + \g B \cdot \g \mu \|_{H^{-1}(\Omega)}. 
  \end{align}
\end{theorem}
The proof of this theorem is given in the subsection \ref{sec:proof}. From this result, we can establish the closed range property of the main operator $A_{\g u}$ as soon as the hypothesis of Proposition \ref{prop:expanded} are satisfied. 

\begin{corollary}  If $\g T_{\g u}$ satisfies the two hypotheses of Proposition \ref{prop:expanded}, then the operator $A_{\g u}$ defined in \eqref{eq:AT} extends continuously to $A_{\g u}:L^2(\Omega,\R^n)\to H^{-1}(\Omega,\R^N)$ and has closed range. 
\end{corollary}

\begin{proof}
Under the hypotheses of Proposition \ref{prop:expanded}, we have enough regularity so that $\g T_{\g u}$ belongs to $L^\infty\left(\Omega,\R^{N\times d\times n}\right)$ from Sobolev embedding rules. Hence the operator $A_{\g u}$ extends continuously from $L^2(\Omega,\R^n)$ to $H^{-1}(\Omega,\R^N)$ using the following estimation:
\begin{equation}\nonumber
\begin{aligned}
	\| A_{\g u} \g \mu \|_{H^{-1}(\Omega)}  &= \norm{\div(\g T_{\g u} \cdot \g \mu)}{H^{-1}(\Omega)} \\
	&\leq  \| \g T_{\g u} \cdot \g \mu  \|_{L^2(\Omega)} \leq \| \g T_{\g u} \|_{L^{\infty}(\Omega)}  \| \g \mu \|_{L^2(\Omega)},\quad \forall  \g \mu \in L^{\infty}(\Omega,\R^n).
\end{aligned}
\end{equation}
From Proposition \ref{prop:expanded}, the equation $A_{\g u}\g \mu=\g f$ is equivalent to $\nabla \g \mu + \g B \cdot \g \mu=F$ where $F=-\g T_{\g u}^{-1}\cdot \g f$. We first remark that $N(A_{\g u}) = N(\nabla + \g B\cdot)$ and from Lemma \ref{prop:products}, we can bound $F$ by $\g f$ as 
\begin{equation}\nonumber
\norm{F}{H^{-1}(\Omega)} = \| \g T_{\g u}^{-1}\cdot \g f \|_{H^{-1}(\Omega)}\leq c_1 \norm{\g T_{\g u}^{-1}}{W^{1,p}(\Omega)}\norm{\g f}{H^{-1}(\Omega)},
\end{equation}
where $c_1>0$ depends only on $p,d$ and $\Omega$. Applying the previous theorem, we get that for all $\g\mu\in N(A_{\g u})^\perp$, 
\begin{align*}
	\| \g \mu \|_{L^2(\Omega)} \le  c_1 \norm{ F}{H^{-1}(\Omega)}\leq c\, c_1 \norm{\g T_{\g u}^{-1}}{W^{1,p}(\Omega)}\norm{\g f}{H^{-1}(\Omega)},
\end{align*}
which concludes. 
\end{proof}

\subsection{Stability estimates}

The closed range property of the operator $A_{\g u}$ naturally induces its stable inversion in $N(A_{\g u})^\perp$. In the next result, we propose a $L^2$-Lipschitz stability estimate for recovering $\g\mu$ with respect to some uncertainty on both $\g u$ and $\g f$.

\begin{theorem} \label{theo:stability-result}  
Assume that $\g T_{\g u}$ satisfies the hypotheses of Proposition \ref{prop:expanded}. There exists a constant $c>0$ such that  for any $\g\mu,\tilde{\g \mu}\in L^2(\Omega,\R^n)$ respectively solutions of
\begin{equation}\nonumber
A_{\g u}\g\mu = \g f
\quad\tand\quad
A_{\tilde{\g u}}\tilde{\g \mu} = \tilde{\g f},
\end{equation}
there exists $\g \mu_0\in N(A_{\g u})$ such that
\begin{equation}\nonumber
\norm{\tilde{\g \mu}-{\g \mu}-\g \mu_0}{L^2(\Omega)} \leq c\left( \norm{\tilde{\g\mu}}{L^2}\max_{\ell\in \{1\dots m\}} \norm{\cE(\tilde{\g u}_\ell) - \cE({\g u}_\ell)}{L^\infty(\Omega)}+\norms{\tilde{\g f}-{\g f} }{H^{-1}(\Omega)} \right).
\end{equation}
\end{theorem}

\begin{proof} Select $\g \mu_0$ such that $\tilde{\g \mu}-{\g \mu}-\g \mu_0\in N(A_{\g u})^\perp$ and  use the closed range property of $A_{\g u}$to write
\begin{align*}
	\norm{\tilde{\g \mu}-{\g \mu}-\g \mu_0}{L^2} &\leq c_1\norm{A_{\g u}(\tilde{\g\mu} -\g\mu) }{H^{-1}}\leq c_1\norm{A_{\tilde{\g u}}\tilde{\g\mu} -A_{\g u}\g\mu }{H^{-1}} +  c_1\norm{(A_{\tilde{\g u}}-A_{\g u})\tilde{\g\mu} }{H^{-1}},\\
	&\leq c_1\norms{\tilde{\g f}-{\g f} }{H^{-1}} +  c_1\norm{A_{\tilde{\g u}}-A_{\g u} }{L^2,H^{-1}}\norm{\tilde{\g\mu}}{L^2},\\
	&\leq c_1\norms{\tilde{\g f}-{\g f} }{H^{-1}} +  c_1 \norms{\g T_{\tilde{\g u}}-\g T_{\g u}}{L^\infty} \norm{\tilde{\g\mu}}{L^2}.
\end{align*}
We then compute
\begin{equation}\nonumber
\begin{aligned}
 \norms{\g T_{\tilde{\g u}}-\g T_{\g u}}{L^\infty} &= \max_{\ell\in \{1,\dots, m\}}  \max_{k \in \{1,\dots, n\}}  \norms{\g C_{k}\cddot\cE(\tilde{\g u}_\ell-{\g u}_\ell)}{L^\infty},\\
 &\leq c_2 \max_{\ell\in \{1,\dots, m\}}   \norms{\cE(\tilde{\g u}_\ell-{\g u}_\ell)}{L^\infty},
\end{aligned}
\end{equation}
where $c_2$ depends only on the fixed $\g C_k$ tensors. This concludes the proof. 
\end{proof}

This result is well-suited to non-homogeneous problems. In the next corollary, we specifically consider the homogeneous case $\g f=\g 0$. In this situation, an additional constraint must be used to seek a non-trivial solution. We draw the reader's attention to the fact that the noisy operator $A_{\tilde{\g u}}$ may, in general, have a trivial null space.

\begin{corollary} \label{cor:stability-static}  Assume that $\g T_{\g u}$ satisfies the hypotheses of Proposition \ref{prop:expanded}. Then, there exists a positive constant $c$ such that, for any $\e >0$ and $\tilde{\g \mu}\in L^2(\Omega,\R^n)$ satisfying
\begin{equation}\nonumber
\norm{A_{\tilde{\g u}}\tilde{\g \mu}}{H^{-1}} \leq \e,\quad \norm{\tilde{\g\mu}}{L^2(\Omega)} =1,
\end{equation}
it exists $\g \mu_0\in N(A_{\g u})$ such that
\begin{equation}\nonumber
\norm{\tilde{\g \mu}-\g \mu_0}{L^2(\Omega)} \leq c\left( \max_{\ell\in \{1,\dots, m\}}\norm{\cE(\tilde{\g u}_\ell) - \cE({\g u}_\ell)}{L^\infty(\Omega)}+\e\right).
\end{equation}
In the case where $\dim N(A_{\g u})$ = 1, if $\g\mu \in L^2(\Omega,\R^n)$ is the normalized solution of 
\begin{equation}\nonumber
A_{{\g u}}{\g \mu}= \g 0,\quad \norm{{\g\mu}}{L^2(\Omega)} =1,\quad\text{with}\quad\inner{\g\mu,\tilde{\g\mu}}{}\geq 0,
\end{equation}
then,
\begin{equation}\nonumber
\norm{\tilde{\g \mu}-\g \mu}{L^2(\Omega)} \leq c\sqrt{2}\left( \max_{\ell\in \{1,\dots, m\}}\norm{\cE(\tilde{\g u}_\ell) - \cE({\g u}_\ell)}{L^\infty(\Omega)}+\e\right).
\end{equation}
\end{corollary}

\begin{proof}
This follows directly from \ref{theo:stability-result}. For the first estimate, we simply apply the theorem to the case $\g\mu = \g 0$. For the second estimate, write $\tilde{\g\mu} = \alpha\g\mu+\g\mu^\perp$ choosing $\g\mu$ such that $\alpha\geq 0$  and $\g\mu^\perp\perp\g\mu$. As $\norm{\tilde{\g\mu}}{L^2(\Omega)} =1$, we have $\alpha^2+\norms{\g\mu^\perp}{L^2(\Omega)}^2=1$ and then:
\begin{equation}\nonumber
\begin{aligned}
\norm{\tilde{\g \mu}-\g \mu}{L^2(\Omega)}^2  = (\alpha-1)^2 + \norms{\g\mu^\perp}{L^2(\Omega)}^2  = 2(1-\alpha)\leq 2(1-\alpha^2)=2\norms{\g\mu^\perp}{L^2(\Omega)}^2.
\end{aligned}
\end{equation}
Consider $\g\mu_0$ from the previous statement. As $\dist(\g{\tilde \mu},N(A_{\g u}))\leq \norm{\tilde{\g \mu}-\g \mu_0}{L^2(\Omega)}$, this gives $\norm{\g\mu^\perp}{L^2(\Omega)}\leq \norm{\tilde{\g \mu}-\g \mu_0}{L^2(\Omega)}$ which leads to $\norm{\tilde{\g \mu}-\g \mu}{L^2(\Omega)} \leq \sqrt{2}  \norm{\tilde{\g \mu}-\g \mu_0}{L^2(\Omega)}$, which completes the proof. 
\end{proof}

\subsection{Proof of Theorem \ref{theo:closedrange}} \label{sec:proof}

\begin{proof} Call $\nabla_{\g B} :=\nabla+ \g B\cdot$, and suppose that \eqref{eqtheo:closedrange} does not hold. Then, there exists a sequence $\left( \g \mu_k\right)_{k \in \mathbb{N}}$ of $ N(\nabla_{\g B})^\perp  $ so that for all $k \in \mathbb{N}$, $\| \g \mu_k \|_{L^2}=1$ and $\|\nabla_{\g B} \g \mu_k \|_{H^{-1}}\to 0$ when $k\to +\infty$. Using classical compactness arguments, the sequence $(\g\mu_k)$ weakly converges (up to a subsequence) to $\g \mu^*\in L^2(\Omega,\R^n)$ and converges (strongly) in $H^{-1}$ to $\g \mu^*$. As $N(\nabla_{\g B})^\perp$ is closed in $L^2$, we also have that $\g \mu^*\in N(\nabla_{\g B})^\perp$. 

For any $V \in H^1_0(\Omega, \R^{d \times n})$,
\begin{align*}
    \left< \nabla_{\g B} \g \mu^*, V \right>_{H^{-1},H_0^1} = \underset{k\to \infty}{\mathrm{lim}} \left< \nabla_{\g B} \g \mu_k, V \right>_{H^{-1},H_0^1} = 0,
\end{align*}
as $\nabla_{\g B} \g \mu_k\to 0$ in $H^{-1}$. We deduce that $\nabla_{\g B} \g \mu^* =0 $, \emph{i.e.} $\g \mu^* \in N(\nabla_{\g B})$ and therefore $\g \mu^*=\bm{0}$. Moreover,
\begin{align}\label{eq:shouldbenum1}
\norm{\g\mu_k}{H^{-1}}\overset{k\to \infty}\lra 0. 
\end{align}

We now use the identity $\| \g \mu_k \|_{L^2}^2 = \| \g \mu_k \|_{H^{-1}}^2 + \|\nabla \g \mu_k \|_{H^{-1}}^2$ and write 
\begin{align} \label{eq:shouldbenum2}
1=\| \g \mu_k \|_{L^2}^2 = \| \g \mu_k \|_{H^{-1}}^2 + \|\nabla \g \mu_k \|_{H^{-1}}^2. 
\end{align}
In order to build a contradiction, we show that $\|\nabla \g \mu_k \|_{H^{-1}}$ converges also to zero.

As $\nabla\g \mu_k=\nabla_{\g B}\g\mu_k - \g B \cdot \g \mu_k$ we write first 
\begin{align}\label{eq:eq1bis}
    \|\nabla \g \mu_k \|_{H^{-1}} \le \| \nabla_{\g B} \g \mu_k \|_{H^{-1}}+ \|\g B \cdot \g \mu_k  \|_{H^{-1}},
\end{align}
where $ \|  \nabla_{\g B} \g \mu_k \|_{H^{-1}}$ converges to zero by hypothesis. Let $\e>0$, and consider a smooth approximation $\g B_\e\in\cC^\infty_c$ of $\g B$ such that $\norm{\g B-\g B_\e}{L^p}<\e$. Then,
\begin{align}\label{eq:lhs_rhs}
\forall k \in \mathbb{N},\quad  & \|\g B \cdot \g \mu_k\|_{H^{-1}} \le \|\g B_\e \cdot \g \mu_k\|_{H^{-1}} + \|(\g B-\g B_\e) \cdot \g \mu_k\|_{H^{-1}}.
\end{align}
We bound the first right-hand side term using Appendix \ref{prop:poincare1}:
\begin{equation}\label{eq:c1}
\norm{\g B_\e\cdot\bm{\mu}_k}{H^{-1}}\leq c \norm{\g B_\e}{W^{1,p}} \norm{\g\mu_k}{H^{-1}},
\end{equation}
where $c>0$ on $n,p,d$ and $\Omega$. As $\norm{\g\mu_k}{H^{-1}}\to 0$ for $k$ sufficiently large, we have $\norm{\g B_\e \cdot \g \mu_k}{H^{-1}}<\e$.

Let us now bound the second term on the right-hand side of \eqref{eq:lhs_rhs}.  Let $V\in H^1_0(\Omega,\R^{d\times n})$, we write 
\begin{align*}
\langle(\g B-\g B_\e) \cdot \g \mu_k, V \rangle_{H^{-1},H^1_0} &= \int_\Omega V \cddot (\g B-\g B_\e) \cdot \g \mu_k\leq \|\g B-\g B_\e \|_{L^p} \| \g \mu_k \|_{L^2} \|V \|_{L^q}\leq \e \|V \|_{L^q},
\end{align*}
with $1/q=1/2-1/p$. As $p>d$, we have, from the Sobolev embedding rules that $H^1_0\hookrightarrow L^q$ and so $\norm{V}{L^q}\leq C_E\norm{V}{H^1_0}$ where $C_E>0$ depends only on $n,d,q$ and $\Omega$. We then deduce that
\begin{equation}\label{eq:c2}
\norm{(\g B-\g B_\e) \cdot \g \mu_k}{H^{-1}}\leq C_E\e. 
\end{equation}

We now combine \eqref{eq:lhs_rhs}, \eqref{eq:c1} and \eqref{eq:c2} to claim that $\norm{\g B \cdot \g \mu_k}{H^{-1}} < (1+C_E)\e $ for $k$ sufficiently large which proves that $\norm{\g B \cdot \g \mu_k}{H^{-1}}$ converges to zero. From \eqref{eq:eq1bis}, we deduce that $\norm{\nabla \g \mu_k}{H^{-1}}$ converges to zero. This combined with \eqref{eq:shouldbenum1} leads to $\norm{\g \mu_k}{L^2}\to 0$, which contradicts the initial hypothesis.
\end{proof}
\section{Null space dimension and $k$-conservative tensor fields}\label{sec:kernel}

This section is devoted to the study of the null space of the operator
\begin{equation}\nonumber
(\nabla+\g B \cdot):L^2(\Omega,\R^n)\to H^{-1}(\Omega,\R^{n\times d}),
\end{equation}
where  $\g B\in L^p(\Omega,\R^{n\times d\times n})$.  In the specific case $\g B=\g 0$, the null space is identified to $\R^n$. We prove that in general, the dimension of $N(\nabla +\g B \cdot)$ is lower or equal to $n$. We then provide characterization of the null space dimension with respect to $\g B$. For that we introduce the original notion of $k$-conservative third order tensor fields.

\subsection{General properties of the null space}

\begin{proposition}\label{prop:prop} Let $\g B\in L^p\big(\Omega,\R^{n\times d\times n}\big)$ for $p>d\geq 2$. Then, for all $\g\mu\in N(\nabla + \g B\cdot)$, we have

\begin{itemize}

\item[1.]  $\g\mu\in W^{1,p}\big(\Omega,\R^{n}\big)$,

\item[2.]  $\g\mu\in \cC^0\big(\ol\Omega,\R^{n}\big)$,

\item[3.] If $\g B\in \cC^0\big(\ol\Omega,\R^{n\times d\times n}\big)$, then $\g\mu\in \cC^1\big(\ol\Omega,\R^{n}\big)$. 

\end{itemize}
\end{proposition}

\begin{proof} 1. Consider $\g\mu\in L^2(\Omega,\R^n)$ such that $\nabla\g\mu + \g B\cdot\g\mu =\g 0$. Using Hölder's inequality, we have that $\g B\cdot \g \mu\in L^q(\Omega,\R^{n\times d})$ with $1/q=1/2+1/p<1$, and then $\nabla \g \mu\in L^q(\Omega,\R^{n\times d})$. Hence $\g\mu\in W^{1,q}(\Omega,\R^n)$ with $q>1$. Call now
\begin{equation}\nonumber
q^*:=\sup\set{q>1}{\g\mu\in W^{1,q}(\Omega,\R^n)}.
\end{equation}
We prove that $q^*>d$. If $q^*\leq d$, then for any $q\in (1,q^*)$, $\g \mu\in W^{1,q}(\Omega,\R^n)$. By Sobolev embedding, $\g \mu\in L^r(\Omega,\R^n)$ where $1/r= 1/q-1/d\in (1,+\infty)$ and then $\g B\cdot \g \mu\in L^s(\Omega,\R^{n\times d})$ with
\begin{align*}
\frac 1s &= \frac 1r + \frac 1p=  \frac 1q + \frac 1p -  \frac 1d = \frac{1}{q} - \tau,
\end{align*}
where $\tau := (p-d)/(pd)>0$. Then $\g \mu\in W^{1,s}(\Omega,\R^n)$. If we choose $q$ close enough to $q^*$, then $s>q^*$. This contradicts the definition of $q^*$ and hence proves that $q^*>d$.

Now, again by Sobolev embeddings, we get that $\g\mu\in L^\infty(\Omega,\R^n)$ and then $\g B\cdot \g \mu\in L^p(\Omega,\R^{n\times d})$, which implies that $\nabla \g\mu \in L^p(\Omega,\R^{n\times d})$ and so $\g\mu\in W^{1,p}(\Omega,\R^{n})$. 

2. It is a consequence of the first statement and the Sobolev embedding $W^{1,p}(\Omega,\R^n)\hookrightarrow  \cC^0\big(\ol\Omega,\R^{n}\big)$, that holds as $p>d$. 

3. From 2, if $\g B\in \cC^0\big(\ol\Omega,\R^{n\times d\times n}\big)$, we get that $\g B\cdot \g\mu\in \cC^0\big(\ol\Omega,\R^{n\times d}\big)$. Therefore,  $\nabla \g\mu\in \cC^0\big(\ol\Omega,\R^{n\times d}\big)$, which implies $\g\mu\in \cC^1\big(\ol\Omega,\R^{n\times d}\big)$.
\end{proof}

We now prove a crucial property which will be useful in most of the subsequent proofs. We state that any non trivial solution of $\nabla\g\mu + \g B\cdot\g\mu =\g 0$ cannot cancel anywhere in $\ol\Omega$.

\begin{proposition}\label{prop:nonzero} Let $\g B\in L^p\big(\Omega,\R^{n\times d\times n}\big)$ with $p>d\geq 2$, and let $\g\mu \in N(\nabla + \g B\cdot)$. If there exists $x_0 \in \overline{\Omega}$ so that $\g\mu(x_0) = \bm{0}$, then $\g\mu = \bm{0} $ in $\overline{\Omega}$.
\end{proposition}
\begin{proof}
Suppose first that $x_0 \in \Omega$, and call $Z = \{x \in \Omega \,| \, \g\mu(x) = \bm{0}\}$. From Proposition \ref{prop:prop}, $\g\mu$ is continuous, and therefore $Z$ is a closed subset of $\Omega$. We will prove that $Z$ is also an open part of $\Omega$, which will imply that $Z = \Omega$ as $\Omega$ is connected. Take $x_1 \in Z$. By applying Corollary \ref{cor:poincare1}, we know that there exists a constants $C>0$ such that for all $\e>0$ with $B(x_1,\e) \subset \Omega$,
	\begin{align*}
\|\g\mu\|_{L^\infty(B(x_1,\e))} &\leq C\e^{1-d/p}\|\nabla \g\mu\|_{L^p(B(x_1,\e))} =  C\e^{1-d/p}\|\g B\cdot \g\mu \|_{L^p(B(x_1,\e))},\\
\|\g\mu\|_{L^\infty(B(x_1,\e))}  &\leq C\e^{1-d/p}\|\g B\|_{L^p(B(x_1,\e))} \|\g\mu \|_{L^\infty(B(x_1,\e))},\\
0 &\leq  \left(C\e^{1-d/p}\|\g B\|_{L^p(B(x_1,\e))}-1 \right)\|\g\mu\|_{L^\infty(B(x_1,\e))}  .\\
\end{align*}
 For $\e$ small enough, we have $C\e^{1-d/p}\|\g B\|_{L^p(B(x_1,\e))}< 1$. The last inequality necessarily implies $\|\g\mu \|_{L^\infty(B(x_1,\e))}=0$, \emph{i.e.} $\g\mu=\bm{0}$ in $B(x_1,\e)$. Hence, $B(x_1,\e) \subset Z$ and $Z$ is open. As mentionned above, the fact that $Z$ is both closed and open in $\Omega$ necessarily connected implies $Z=\Omega$. As $\g\mu$ is continuous in $\ol\Omega$, we conclude that $\g\mu=\g 0$ in $\ol\Omega$.

Suppose now that $\g\mu$ cancels at $x_0 \in \partial\Omega$. We will show that $\g\mu$ cancels also somewhere inside $\Omega$ and conclude using the previous case. As $\Omega$ is Lipschitz, it satisfies the interior cone condition (see \cite{hofmann2007geometric}), meaning there exists an open cone of vertex $x_0$, $C(x_0,v)$, so that for all $\e>0$ sufficiently small, $C(x_0,v) \cap B(x_0, \e) \subset \Omega$. From Corollary \ref{cor:poincare2}, we have the existence of a constant $C>0$ such that for all sufficiently small $\e>0$ with $C(x_0,v)\cap B(x_0,\e) \subset \Omega$, we have 
	\begin{align*}
		\| \bm{\mu} \|_{L^{\infty}(C(x_0,v) \cap B(x_0,\e))} &\le C \e^{1-d/p} \| \bm{B} \|_{L^{p}(C(x_0,v) \cap B(x_0,\e))}\| \bm{\mu} \|_{L^{\infty}(C(x_0,v) \cap B(x_0,\e))},\\
		0 &\leq \left(C \e^{1-d/p} \| \bm{B} \|_{L^{p}(C(x_0,v) \cap B(x_0,\e))}-1\right)\| \bm{\mu} \|_{L^{\infty}(C(x_0,v) \cap B(x_0,\e))}.
	\end{align*}
For $\e$ small enough, the last inequality implies $\|\g\mu \|_{L^\infty(C(x_0,v) \cap B(x_0,\e))}=0$. Hence, $\g\mu$ cancels somewhere inside $\Omega$. We then apply the previous case and conclude that $\g\mu=\g 0$ in $\ol\Omega$.
\end{proof}

Using this result, we then provide a characterization of the linear independence of the solutions of $\nabla\g\mu + \g B\cdot\g\mu =\g 0$ using this result.

\begin{corollary}\label{lem:gram_positive} Let $\g B\in L^p\big(\Omega,\R^{n\times d\times n}\big)$ with $p>d\geq 2$. Let $\g\mu_1,\dots,\g\mu_k \in N(\nabla + \g B\cdot)$. The following assertions are equivalent:
\begin{itemize}
	\item[1.] The family $(\g\mu_1,\dots,\g\mu_k)$ is linearly independent.
	\item[2.] At every point $x\in \overline{\Omega}$, the family $(\g\mu_1(x),\dots,\g\mu_k(x))$ of $\R^n$ is linearly independent. 
	\item[3.] The Gram matrix $G_{ij}(x) := \g\mu_i(x) \cdot\g \mu_j(x)$ is uniformly positive definite, \emph{i.e.}
\begin{align*}
	\exists \alpha > 0,\quad\forall x\in \overline{\Omega},\quad \forall \bm{v} \in \mathbb{R}^k, \quad \bm{v}^T G(x) \bm{v} \ge \alpha |\bm{v}|^2.
\end{align*}
\end{itemize}
\end{corollary}

\begin{proof} We prove this equivalence by the scheme: (not 1) $\imp$ (not 2) $\imp$ (not 3) $\imp$ (not 1). 

 (not 1) $\imp$ (not 2): If 1 is false, then there exists then $(\alpha_1,\dots,\alpha_k)\in\R^k\bs\{\g 0\}$ such that 
 \begin{equation}\nonumber
\sum_{i=1}^k \alpha_i\g\mu_i = \g 0. 
 \end{equation}
Let  $x \in \ol\Omega$ we have $\sum_{i=1}^k \alpha_i \g\mu_i(x) = \g 0$ so  $(\g\mu_1(x),\dots,\g\mu_k(x))$ is not independent and then 2 is false. 

 (not 2) $\imp$ (not 3): If $(\g\mu_1(x),\dots,\g\mu_k(x))$ is linearly dependent for some $x$, then $G(x)$ is not invertible and 3 is false. 
 
  (not 3) $\imp$ (not 1): If 3 is false, we can construct two sequences $(x_n)$ of $\ol\Omega$ and $(\g v_n)$ of $\R^k$ such that $|\g v_n|=1$ such that $\g v_n^TG(x_n)\g v_ n$ converges to zero. By compactness these sequences converges, up to an extraction, respectively to $x^*\in\ol\Omega$ and $\g v^*$ with $|\g v^*|=1$. By continuity we have $(\bm{v}^*)^T G(x^*) \bm{v}^* = 0$ which means that $G(x^*)$ is not injective and therefore  $(\g\mu_1(x^*),...,\g\mu_k(x^*))$ is linearly dependent. 
  
One may  find $(a_1,...,a_k)\in\R^k\bs\{\g 0\}$ such that $\g\mu := \sum_{j=1}^k a_j \g\mu_j$ satisfies $ \g\mu(x^*) = \bm{0}$. This fonction belongs to $N(\nabla + \g B\cdot)$ so thanks to Proposition \ref{prop:nonzero}, we get that $\g\mu=\g 0$ in $\ol\Omega$ and so the statement 1 is false. 
\end{proof}

We can now deduce a general upper bound of the null space dimension of the operator $(\nabla + \g B\cdot)$.

%
\begin{corollary}\label{prop:upperbound-n}  For any $\g B\in L^p(\Omega,\R^{n\times d\times n})$, we have $\dim N(\nabla + \g B\cdot)\leq n$.
\end{corollary}
\begin{proof} Consider a family $(\g \mu_1,\dots,\g\mu_{n+1})$ of $n+1$ elements of $N(\nabla + \g B\cdot)$. Thanks to Proposition \ref{prop:prop}, $\g \mu_1,\dots,\g\mu_{n+1}$ are continuous in $\ol\Omega$. Let $x_0\in \Omega$, the family of vectors $(\g \mu_1(x_0),\dots,\g\mu_{n+1}(x_0))$ is linearly dependent in $\R^n$. Hence, there exists a non zero vector of coefficients $(\alpha_1,\dots,\alpha_{n+1})$ such that
\begin{equation}\nonumber
\sum_{k=1}^{n+1}\alpha_{k}\g\mu_{k}(x_0)=0.
\end{equation}
We then denote $\g\mu := \sum_{k=1}^{n+1}\alpha_{k}\g\mu_{k}$, which belongs to $N(\nabla + \g B\cdot)$ by linearity. It satisfies $\g\mu(x_0)=\g 0$, and thanks to Proposition \ref{prop:nonzero}, $\g\mu=\g 0$ everywhere in $\ol\Omega$. This means that the family $(\g \mu_1,\dots,\g\mu_{n+1})$ is linearly dependent. This proves \emph{de facto} that $\dim N(\nabla + \g B\cdot)\leq n$.
\end{proof}

\subsection{Null space dimension in the case $n=1$}

In the case $n=1$, the tensor field $\g B$ reduces to a simple vector field denoted $\g b\in L^p(\Omega,\R^d)$ and we know from Proposition \ref{prop:upperbound-n} that the null space of $(\nabla + \g B\cdot)$ is of dimension zero or one. The details of this case have already been studied in \cite{ammari2021direct}. We will recall here the corresponding result:

\begin{proposition} \label{prop:scalarhomo} Let $\g b\in L^p(\Omega,\R^d)$, with $p>d$. The problem 
\begin{align}\label{eq:scalarhomo}
	\textrm{Find } \quad \mu\in L^2(\Omega),\quad \textrm{s.t. } \quad \nabla \mu+\mu\, \g b = \g 0,
\end{align}
admits a non-trivial solution if and only if there exists $\nu\in W^{1,p}(\Omega)$ such that $\g b=\nabla \nu$. In this case, the non trivial solutions are $\mu=\alpha\, \mathrm{exp}(-\nu)$, $\alpha\in\R\bs \{0\}$. 
\end{proposition}

Hence, the null space is of dimension one if and only if $\g b$ is a gradient, \emph{i.e.}, is a conservative vector field. Let us formally explain why this property of conservativity allows for a non trivial solution to the problem \eqref{eq:scalarhomo}. 

If we suppose that $\g b\in\cC^0\big(\ol\Omega,\R^d\big)$, the property of conservativity is equivalent to the fact that the circulation of $\g b$ along a smooth path in $\Omega$ does not depend on the path itself but only on the starting and the ending points: for any $\g\gamma\in\cC^1([0,1],\Omega)$, the circulation 
\begin{equation}\nonumber
	\int_0^1 \g b (\bm{\gamma})\cdot \bm{\gamma}'
\end{equation}
depends only on $\bm{\gamma}(0)$ and $\bm{\gamma}(1)$. Using the characteristics method, we can construct a non-trivial solution to the equation \eqref{eq:scalarhomo}. Indeed, fix $x_0\in\Omega$ and an arbitrary value $\mu(x_0)\neq 0$. We build the solution $\mu$ in the whole domain $\Omega$ as follows. Let $x\in\Omega$ and $\bm{\gamma} \in \cC^1([0,1],\Omega)$ such that $\bm{\gamma}(0)=x_0$ and $\bm{\gamma}(1)=x$, we define $\varphi$ as the unique solution of the Cauchy problem 
\begin{align}\label{eq:systemODEs}\left\{\begin{aligned}
\varphi' &= - \left( (\g b\circ\g\gamma) \cdot \bm{\gamma}'\right)\varphi \quad \mathrm{in} \; [0,1], \\
\varphi(0) &= \mu(x_0),
\end{aligned}\right.\end{align}
and define $\mu(x):=\varphi(1)$. This is well defined as $\varphi$ has an explicit expression and so
\begin{align*}
	\mu(x)=\mu(x_0)\exp\left(-\int_0^1 (\g b\circ\g\gamma)\cdot \bm{\gamma}'\right),
\end{align*}
where the above integral depends only on $x$ (and $x_0$ which is fixed). We then check that $\mu$, as defined above, is indeed a solution of \eqref{eq:scalarhomo}. Let $v \in\R^d\bs\{0\}$ and choose $\bm{\gamma}$ such that $\bm{\gamma}'(1)= v$. The solution $\mu$ satisfies $\varphi(t)=\mu(\bm{\gamma}(t))$ for all $t\in[0,1]$ and then, from the definition of $\varphi$,
\begin{align*}
	\nabla\mu(x)\cdot v = \varphi'(1) =  - \mu(x)\g b(x)\cdot v. 
\end{align*}
Hence $\nabla\mu(x)\cdot v + \mu(x) \g b(x)\cdot v= \bm{0}$. As this is true for all $v\in\R^d$, $\nabla\mu + \mu\g b= \bm{0}$ is satisfied. 

To sum up, when the vector field $\g b$ is continuous and conservative, the method of characteristics yields a well-defined, non-trivial solution. This approach is generalized in the next subsection to the case $n>1$. To do so, we need to introduce the notion of \emph{conservative third-order tensor fields}.

\subsection{Continuous $k$-conservative third-order tensor fields}

We consider in this subsection a continuous third-order tensor field $\g B\in \cC^0\big(\ol\Omega,\R^{n\times d\times n}\big)$.  For any $x,y\in\Omega$, we define the set $\Gamma(x,y)$ of all smooth curves from $x$ to $y$ in $\Omega$ as 
\begin{equation}\nonumber
\Gamma(x,y):=\set{\g\gamma\in\cC^1([0,1],\Omega)}{\g\gamma(0)=x,\ \g\gamma(1)=y}. 
\end{equation}
Some basic operations on these curves will be useful:
\begin{itemize}
	\item Reversion: for $\g\gamma\in\Gamma(x,y)$, we denote $\ol{\g\gamma}\in \Gamma(y,x)$ defined by $\ol{\g\gamma}(t) = {\g\gamma}(1-t)$ for all $t\in [0,1]$,
	\item Concatenation: for $\g\gamma_1\in\Gamma(x,y)$ and $\g\gamma_2\in\Gamma(y,z)$, we define $[\g\gamma_1 \, \g\gamma_2](t):=\gamma_1(2t)\chi_{[0,1/2]}(t)+\gamma_2(2t-1)\chi_{(1/2,1]}(t)$ which belongs to $\Gamma(x,z)$ if the connexion at point $y$ is of class $\cC^1$. 
\end{itemize}  
Consider now a solution $\bm{\mu}$ of the equation $\nabla \bm{\mu} + \g B \cdot \bm{\mu}=\bm{0}$. For any smooth curve $\bm{\gamma}\in \Gamma(x,y)$, the flow $\bm{\varphi}:= \bm{\mu} \circ \bm{\gamma}$ of $\bm{\mu}$ on the curve $\bm{\gamma}$ satisfies the following system of ODEs:
\begin{align}\label{eq:ODE}
\bm{\varphi}'=(\nabla \bm{\mu} \circ \bm{\gamma})\cdot \bm{\gamma}' = - \left(( \g B\circ \bm{\gamma})\cdot(\bm{\mu} \circ \bm{\gamma})\right)\cdot \bm{\gamma}' = -  \left((\g B\circ \bm{\gamma})\cdot\bm{\varphi}\right)\cdot \bm{\gamma}'. 
\end{align}
This function also satisfies the boundary values $\bm{\varphi}(0)=\bm{\mu}(x)$ and $\bm{\varphi}(1)=\bm{\mu}(y)$. Hence, we see that the final value of the solution of \eqref{eq:ODE} should not depend on the path $\bm{\gamma}$ but only on $\bm{\gamma}(0)=x$ and $\bm{\gamma}(1)=y$. This will allow us to characterize the solutions of $\nabla \bm{\mu} +\bm{B} \cdot \bm{\mu} = \bm{0}$. For this purpose, we introduce the following linear operator:  

 \begin{definition} For any $(x,y)\in\Omega^2$ and any $\g\gamma\in \Gamma(x,y)$, we define the linear operator $R^{\g\gamma}_{\g B}:\R^n\to \R^n$ by
 \begin{equation}\nonumber
R^{\bm{\gamma}}_{\g B} (\bm{v}):=\bm{\varphi}_{\bm{v}}(1), \quad \forall \bm{v} \in \mathbb{R}^n,
 \end{equation} 
where $\bm{\varphi}_{\bm{v}}$ is the unique solution of the Cauchy problem 
 \begin{equation}\label{eq:phi}\left\{\begin{aligned}
 \bm{\varphi}' &= -  \left((\g B\circ\g\gamma)\cdot\bm{\varphi}\right)\cdot \g\gamma'\quad \ton [0,1],\\
 \bm{\varphi}(0) &= \bm{v}.
 \end{aligned}\right.\end{equation}
 \end{definition}

It satisfies some straightforward properties:
 \begin{proposition}\label{prop:resolvant_properties} Let $\g\gamma\in \Gamma(x,y)$ and $\g\delta \in\Gamma(y,z)$ such that $[\g\gamma \, \g\delta ]\in \Gamma(x,z)$, then 
 \begin{enumerate}
 \item $R_{\g B}^{\g\gamma}:\R^n\to\R^n$ is a bijection and $R_{\g B}^{\ol{\g\gamma}} = (R_{\g B}^{\g\gamma})^{-1}$,
 \item $R_{\g B}^{[\g\gamma \, \g\delta]} = R_{\g B}^{\g\delta} \circ R_{\g B}^{\g\gamma}$.
 \end{enumerate} 
 \end{proposition}
 
 \begin{proof} These results come from simple changes of variable in the ODE \eqref{eq:phi} and the application of the Cauchy-Lipschitz uniqueness theorem. 
 \end{proof}

 This operator $R^{\g\gamma}_{\g B}$ allows to give a rather simple criterium for a vector field $\g\mu$ to belong to $N(\nabla + \g B\cdot)$.
  
 \begin{proposition}\label{prop:resolvant_kernel} A vector field $\g\mu$ belongs to $N(\nabla + \g B\cdot)$ if and only if $\g\mu\in \cC^1\big(\ol\Omega,\R^{n}\big)$ and 
 \begin{equation}\nonumber
 \forall x,y\in\Omega,\quad \forall \gamma\in\Gamma(x,y),\quad R_{\g B}^{\g \gamma}(\g\mu(x)) = \g \mu(y).
 \end{equation}
 \end{proposition}

 \begin{proof} For the direct implication, if $\g\mu\in N(\nabla + \g B\cdot)$, then $\bm{\mu} \in \cC^1\big(\ol\Omega,\R^{n}\big)$ from Proposition \ref{prop:prop}. We then take $\gamma\in\Gamma(x,y)$ and we define $\bm{\varphi}:=\g\mu\circ\g\gamma$, which is a solution of the ODE system \eqref{eq:phi} with the initial condition $\bm{\varphi}(0)=\g\mu(x)$. Then $\bm{\varphi}(1)=\g\mu(y)$ and, by definition of $R_{\g B}^{\g \gamma}$, we have $R_{\g B}^{\g \gamma}(\g\mu(x))=\g\mu(y)$.

Conversely,  we fix $y\in\Omega$ and  $\bm{v}\in\R^n\bs\{\bm{0}\}$. We then chose $\g\gamma\in\Gamma(x,y)$ such that $\g\gamma'(1)=\bm{v}$ which is always possible. If one calls $\bm{\varphi}$ the unique solution of the characteristic ODE system \eqref{eq:phi} with $\bm{v} = \bm{\mu}(x)$, we see that for all $s\in [0,1]$, $\bm{\varphi}(s)=R_{\g B}^{\g\gamma_s}(\g v)$, where $\bm{\gamma}_s(t):=\bm{\gamma}(st)$ for $t\in [0,1]$. Hence, for all $s\in [0,1]$, $\bm{\varphi}(s)=\g\mu(\g\gamma(s))$. Then, using \eqref{eq:phi}, we write:
 \begin{equation}\nonumber
 \begin{aligned}
  \nabla \g\mu(\g\gamma(1))\cdot\g\gamma' (1) &=  - \left(\g B(\g\gamma(1))\cdot \bm{\varphi}(1)\right)\cdot \g\gamma'(1),\\
    \nabla \g\mu(y)\cdot \bm{v} &=  - (\g B(y)\cdot \g\mu(y))\cdot \bm{v},\\
        \left[\nabla \g\mu(y) + \g B(y)\cdot \g\mu(y)\right]\cdot \bm{v} &= 0.\\
 \end{aligned}
 \end{equation}
 As the last equality holds for any $\bm{v}\in\R^n$ and any $y\in\Omega$, we conclude that $\nabla \g\mu + \g B\cdot \g\mu =\g 0$.
\end{proof}

For a chosen point $x\in\Omega$, we can define the set of vectors $\bm{v}\in\R^n$ such that for all $y\in\Omega$, the image $R^{\g\gamma}_{\g B} (\g v)$ does not depend on the path $\g\gamma\in\Gamma(x,y)$. 

\begin{definition} For any $x\in\Omega$, define
\begin{equation}\nonumber
E_{\g B}^{x}:=\set{\bm{v}\in\R^n}{\forall y\in \Omega,\ \forall \bm{\gamma}_1,\bm{\gamma}_2\in\Gamma(x,y),\ R^{\g\gamma_1}_{\g B}(\g v) = R^{\g\gamma_2}_{\g B}(\g v) }. 
\end{equation}
\noindent In other terms,
\begin{equation}\nonumber
E_{\g B}^{x}=\bigcap_{y\in\Omega}\  \bigcap_{\g\gamma_1,\g\gamma_2 \in \Gamma(x,y)} (R^{\g\gamma_1}_{\g B}-R^{\g\gamma_2}_{\g B})^{-1}(\{0\}). 
\end{equation}
\end{definition}

From this definition, it is clear that $E_{\g B}^{x}$ is a vector subspace of $\R^n$. Its dimension is independent on the point $x$. Thus, $\dim E_{\g B}^x$ is a characteristic number of the tensor field $\g B$. We clarify this statement in the following proposition.

\begin{proposition}\label{prop:dimconst} The integer $\dim E_{\g B}^{x}\in\{0,\dots,n\}$ does not depend on $x\in\Omega$, \emph{i.e.}
\begin{equation}\nonumber
\forall x_1,x_2\in\Omega,\quad \dim E_{\g B}^{x_1} = \dim E_{\g B}^{x_2}.
\end{equation}
\end{proposition}

\begin{proof} Let $\g \gamma\in \Gamma(x_1,x_2)$. We prove that $R_{\g B}^{\g\gamma}(E_{\g B}^{x_1})\subset E_{\g B}^{x_2}$. Take $\bm{v}\in E_{\g B}^{x_1}$ and call $\bm{w}:=R_{\g B}^{\g\gamma}(\g v)$. Let $y\in\Omega$ and $\g\gamma_1, \g\gamma_2\in \Gamma(x_2,y)$. We can find $\wt{\g\gamma}_1, \wt{\g\gamma}_2\in \Gamma(x_1,x_2)$ such that $[\wt{\g\gamma}_1 \, {\g\gamma_1}]\in \Gamma(x_1,y)$  and $[\wt{\g\gamma}_2 \, {\g\gamma_2}]\in \Gamma(x_1,y)$, hence by the definition of $E_{\g B}^{x_1}$, we have $R_{\g B}^{[\wt{\g\gamma}_1 \, \g\gamma_1]}(\g v)=R_{\g B}^{[\wt{\g\gamma}_2 \, \g\gamma_2]}(\g v)$. From Proposition \ref{prop:resolvant_properties}, 
\begin{equation}\nonumber\begin{aligned}
R_{\g B}^{[\wt{\g\gamma}_1 \g\gamma_1]}(\g v) =R_{\g B}^{[\wt{\g\gamma}_2 \g\gamma_2]}(\g v) &\quad\eq\quad R_{\g B}^{\g\gamma_1}\circ R_{\g B}^{\wt{\g\gamma}_1}(\g v) =R_{\g B}^{\g\gamma_2}\circ R_{\g B}^{\wt{\g\gamma}_2}(\g v)  \\
& \quad\eq\quad R_{\g B}^{\g\gamma_1}\circ R_{\g B}^{\g\gamma}(\g v) =R_{\g B}^{\g\gamma_2}\circ R_{\g B}^{\g\gamma}(\g v) \\ 
&\quad\eq\quad R_{\g B}^{\g\gamma_1}(\g w) =R_{\g B}^{\g\gamma_2}(\g w).\\
\end{aligned}\end{equation}
See Figure \ref{fig:scheme} for a graphic interpretation of this computation. This last line shows that $\bm{w}\in E^{x_2}_{\g B}$. We then deduce that  $R_{\g B}^{\g\gamma}(E_{\g B}^{x_1})\subset E_{\g B}^{x_2}$ and, as $R_{\g B}^{\g\gamma}$ is bijective, $\dim E_{\g B}^{x_1} \leq  \dim E_{\g B}^{x_2}$. By symmetry of roles of $x_1$ and $x_2$, we also have $\dim E_{\g B}^{x_2} \leq  \dim E_{\g B}^{x_1}$. 
\end{proof}

\begin{figure}
\begin{center}
\begin{tikzpicture}[scale=3]
\node[label=below:$x_1$] (1) at (0,0) {};
\node[label=below:$x_2$] (2) at (1,1){};
\node[label=below:$y$] (3) at (3,1){};
\node[label=left:$\red{\bm{v}}$]  at (0,0) {};
\node[label=above:$\red{R_{\g B}^{{\g\gamma}}(\g v)=R_{\g B}^{\wt{\g\gamma}_1}(\g v)=R_{\g B}^{\wt{\g\gamma}_2}(\g v)}$]  at (0.7,1.3){};
\draw [red] (0.75,1.3) -- (1,1);
\node[label=right:$\red{R_{\g B}^{[\wt{\g\gamma}_1 \g\gamma_1]}(\g v)=R_{\g B}^{[\wt{\g\gamma}_2 \g\gamma_2]}(\g v) }$]  at (2.5,1.3){};
\draw [black] plot [smooth,tension=1] coordinates {(1) (0.5,0.3)  (2)  (1.8,1.3) (3)};
\draw [black] plot [smooth,tension=1] coordinates {(1) (0.1,0.5)  (2) (1.8,0.5) (3)};
\draw [black] plot [smooth,tension=1] coordinates {(1) (0.2,0.4)  (2)};
\draw [fill=black] (1) circle (1pt);
\draw [fill=black] (2) circle (1pt);
\draw [fill=black] (3) circle (1pt);
\node[label=below:$\g\gamma_1$] at (1.7,1.6) {};
\node[label=below:$\g\gamma$] at (0.4,0.6) {};
\node[label=below:$\g\gamma_2$] at (1.7,0.5) {};
\node[label=below:$\tilde{\g\gamma_1}$] at (0.7,0.5) {};
\node[label=below:$\tilde{\g\gamma_2}$] at (0.2,1.1) {};
\end{tikzpicture}
\end{center}
\caption{\label{fig:scheme} Illustration of the proof of \ref{prop:dimconst}. As $\bm{v}\in E_{\g B}^{x_1}$, we get all the identities in red. From that, we deduce $R_{\g B}^{\g\gamma_1}(\g w) =R_{\g B}^{\g\gamma_2}(\g w)$ where $\bm{w}=R_{\g B}^{\g\gamma}(\g v)$. }
\end{figure}
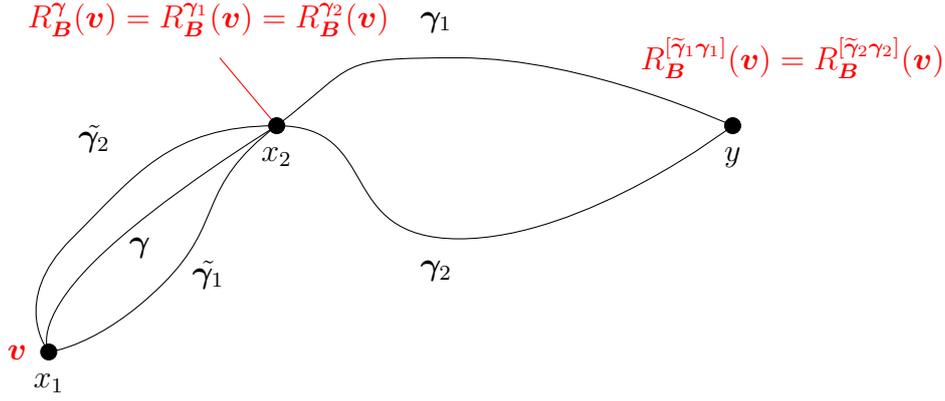

We now propose the following definition. 

\begin{definition} We say that a third-order tensor field $\g B\in\cC^0\big(\ol\Omega,\R^{n\times d\times n}\big)$ is $k$-conservative if 
\begin{equation}\nonumber
\dim E_{\g B}^{x} = k,\quad\forall x\in\Omega.
\end{equation}
We also say that $k$ is the degree of conservativity of the field $\g B$. 
\end{definition}

\begin{remark} From the previous proposition, any tensor field $\g B\in\cC^0\big(\ol\Omega,\R^{n\times d\times n}\big)$ is $k$-conservative for some $k\in\{0,\dots,n\}$. 
\end{remark}

\subsection{Null space dimension for continuous tensor fields}

This subsection focuses on the link between the degree of conservativity and the dimension of $N(\nabla + \g B\cdot)$. In what follows, we prove that they are identical.

\begin{theorem}\label{theo:C0kcons} Let $\g B\in\cC^0(\ol\Omega,\R^{n\times d\times n})$ and $k\in\{0,\dots,n\}$, then
\begin{equation}\nonumber
\dim N(\nabla + \g B\cdot) = k\quad\eq \quad \g B\text{ is $k$-conservative. }
\end{equation}
\end{theorem}

\begin{proof}First remark that the statement of the theorem is equivalent to the following:  
\begin{align*}
\forall k\in \{1,\dots,n\}, \quad \dim N(\nabla + \g B\cdot) \geq k\quad\eq \quad \g B\text{ is $k'$-conservative for some } k'\geq k. 
\end{align*}
We then prove the above equivalence, starting with the direct implication. If $\dim N(\nabla + \g B\cdot) \geq k\geq 1$, then there exists an independent family $(\g\mu_1,\dots,\g\mu_k)$ satisfying $\nabla\g\mu_\ell + \g B\cdot\g \mu_\ell = \g 0$ for all $\ell\in \{1,\dots,k\}$. For any $x_0\in \Omega$, the family of vectors $(\g\mu_1(x_0),\dots,\g\mu_k(x_0))$ is also independent, as stated in Corollary \ref{lem:gram_positive}. 

For any $\ell\in \{1,\dots,k\}$, $y\in\Omega$ and $\bm{\gamma}\in\Gamma(x,y)$, we have from Proposition \ref{prop:resolvant_kernel} that  $R_{\g B}^{\g \gamma}(\g\mu_\ell(x)) = \g\mu_\ell(y)$. Hence, the image of $\g\mu_\ell(x)$ does not depend on the path of $\g\gamma$, implying that $\g\mu_\ell(x)\in E_{\g B}^x$. The family $(\g\mu_1(x),\dots,\g\mu_k(x))$ being independent in $E_{\g B}^x$, this induces $\dim E_{\g B}^x\geq k$. Hence $\g B$ is $k'$-conservative for $k'\geq k$. 

Now, for the reciprocal implication, suppose that $\g B$ is $k'$-conservative for some $k'\geq k$. We fix $x\in\Omega$ and denote $(\bm{v}_1,\dots,\g v_k)$ an independent family of $E^{x}_{\g B}$. For any $y \in \Omega$, the vector $R^{\g\gamma}_{\g B}(\bm{v}_\ell)$ is independent on $\g\gamma\in \Gamma(x,y)$. Hence, we define
\begin{equation}\nonumber
\g\mu_\ell(y) := R^{\g\gamma}_{\g B}(\bm{v}_\ell),
\end{equation}
for $\ell\in\{1,\dots,k\}$. These fields are of class $\cC^1$ as $\g B$ is continuous. Proposition \ref{prop:resolvant_kernel} proves that these functions all belong to $N(\nabla + \g B\cdot)$. As the family $(\g\mu_1(x):=\g v_1,\dots,\g\mu_k(x):=\g v_k)$  is independent in $\R^n$, Corollary \ref{lem:gram_positive} implies that the family $(\g\mu_1,\dots,\g\mu_k)$ is also independent. Then, $\dim N(\nabla + \g B\cdot) \geq k$. 
\end{proof}

\subsection{Examples of $k$-conservative tensor fields}

We propose some examples to illustrate the notion of third-order conservative tensor fields.

\begin{example}
	Let $\g\mu \in \mathcal{C}^1(\ol\Omega,\mathbb{R}^n)$ so that for all $x \in \overline{\Omega}$, $\g\mu(x) \neq \g 0$. Then	\begin{align*}
		\g B := -\frac{\nabla \g\mu \otimes \g\mu }{|\g\mu|^2}.
	\end{align*}
is at least $1$-conservative.
\end{example}

\begin{proof} Simply write
\begin{equation}\nonumber
\g B\cdot \g\mu:= - \frac{\nabla \g\mu (\g\mu\cdot\g\mu)}{|\g\mu|^2} = -\nabla\g \mu.
\end{equation}
Then $\g\mu\in N(\nabla + \g B\cdot)$ and, from Proposition \ref{theo:C0kcons}, the field $\g B$ is at least $1$-conservative. 
\end{proof}

\begin{example} Let $M\in\cC^1(\ol\Omega,\R^{n\times n})$ be a smooth matrix field everywhere invertible. Call $(DM)_{ijk}:=\partial_{j}M_{ik}$. Then the tensor field 
\begin{equation}\nonumber
\g B:=M^{-1}\cdot DM
\end{equation}
is $n$-conservative. 
\end{example}

\begin{proof} Call $(\g e_1,\dots ,\g e_n)$ the canonical basis of $\R^n$ and define $\g \mu_i(x):=M^{-1}(x)\cdot \g e_i$ for all $i$. The family $(\g\mu_1,\dots,\g\mu_n)$ of $N(\nabla + \g B\cdot)$ is linearly independent. Indeed,
\begin{equation}\nonumber
\begin{aligned}
\sum_{i=1}^n\alpha_i\g\mu_i=\g 0\quad &\eq\quad M^{-1}(x)\cdot \sum_{i=1}^n\alpha_i \g e_i=\g 0,\quad \forall x\in \ol\Omega,\\
&\eq\quad\sum_{i=1}^n\alpha_i \g e_i=\g 0\quad \eq\quad \alpha_1=\dots=\alpha_n=0. 
\end{aligned}
\end{equation}
For each $i$, we see that $M(x)\cdot \g\mu_i(x)$ is constant in $\ol\Omega$, hence,
\begin{equation}\nonumber
\begin{aligned}
\nabla (M\cdot \g\mu_i) &= \g 0,\\
DM\cdot \g\mu_i +M\cdot \nabla \g\mu_i &= \g 0,\\
M^{-1}\cdot DM\cdot \g\mu_i +\nabla \g\mu_i &= \g 0,\\
\g B\cdot \g\mu_i +\nabla \g\mu_i &= \g 0. \\
\end{aligned}
\end{equation}
Applying Proposition \ref{theo:C0kcons}, we conclude that $\g B$ is $n$-conservative. 
\end{proof}

\begin{example} Let $\g b_1,\dots \g b_n\in \cC^0\big(\ol\Omega,\R^d\big)$ such that $\g b_1,\dots,\g b_k$ are conservative vector fields and $\g b_{k+1},\dots,\g b_n$ are not. Then the tensor field $\g B\in \cC^0\big(\ol\Omega,\R^{n\times d\times n}\big)$ defined by
\begin{equation}\nonumber
(\g B)_{iji} = (\g b_i)_j\quad\tand\quad (\g B)_{ijk} = 0\quad\text{if } i\neq k,
\end{equation}
is $k$-conservative. 
\end{example}

\begin{proof} We prove that for all $x\in\Omega$, $E_{\g B}^x = \vspan\{\g e_1,\dots,\g e_k\}$, which by definition implies that $\g B$ is $k$-conservative. Take $x,y\in\Omega$ and $\bm{\gamma}\in\Gamma(x,y)$. We first remark that the operator $R_{\g B}^{\bm{\gamma}}$ is diagonal in the canonical basis. Indeed, for any $\g w\in\R^n$, we have $R_{\g B}^{\bm{\gamma}}(\g w):=\bm{\varphi}(1)$, where $\bm{\varphi}$ is the unique solution of 
 \begin{equation}\nonumber
 \left\{\begin{aligned}
 \bm{\varphi}' &= -  \left((\g B\circ\g\gamma)\cdot\bm{\varphi}\right)\cdot \g\gamma'\quad \ton [0,1],\\
 \bm{\varphi}(0) &= \bm{w}.
 \end{aligned}\right.
 \end{equation}
 Here we can write that for all $i$,
 \begin{equation}\nonumber
 \begin{aligned}
 \bm{\varphi}_i' &= - \sum_{jk}(\g B\circ\g\gamma)_{ijk}\bm{\varphi}_k\g\gamma_j' = - \sum_{j}(\g B\circ\g\gamma)_{iji}\bm{\varphi}_i\g\gamma_j' = - \sum_{j}(\g b_i\circ\g\gamma)_{j}\bm{\varphi}_i\g\gamma_j',\\
 &= - \bm{\varphi}_i (\g b_i\circ\g\gamma)\cdot \g\gamma'.
 \end{aligned}
 \end{equation}
 Hence this ODE can be solved explicitly: for all $i$,
 \begin{equation}\nonumber
 \bm{\varphi}_i(t)=\g w_i\exp\left(-\int_0^t\g b_i(\g\gamma)\cdot\g\gamma'\right),\quad\forall t\in[0,1].
 \end{equation}
Then,
\begin{equation}\label{eq:Rdiago}
R_{\g B}^{\g\gamma}(\g w)_i =\g w_i\exp\left(-\int_0^1\g b_i(\g\gamma)\cdot\g\gamma'\right). 
\end{equation}
In particular, for all $\ell\in\{1,\dots,k\}$, 
\begin{equation}\nonumber
R_{\g B}^{\g\gamma}(\g e_\ell) =\exp\left(-\int_0^1\g b_\ell(\g\gamma)\cdot\g\gamma'\right)\g e_\ell,
\end{equation}
and, as $\g b_1,\dots,\g b_k$ are conservative, $R_{\g B}^{\g\gamma}(\g e_\ell)$ does not depend on $\g\gamma\in\Gamma(x,y)$. Consequently, $\g e_\ell\in E_{\g B}^x$ and so $\vspan\{\g e_1,\dots,\g e_k\}\subset E_{\g B}^x$. This proves that $\g B$ is at least $k$-conservative. 

Take now $\g w\notin\vspan\{\g e_1,\dots,\g e_k\}$, there exists $\ell>k$ such that $\g w_\ell\neq 0$. As $\g b_\ell$ is not conservative, there exists $y\in\Omega$ and $\g\gamma_1,\g\gamma_2\in\Gamma(x,y)$ such that 
\begin{equation}\nonumber
\int_0^1\g b_\ell(\g\gamma_1)\cdot\g\gamma'_1\neq \int_0^1\g b_\ell(\g\gamma_2)\cdot\g\gamma'_2.
\end{equation}
From \eqref{eq:Rdiago} we deduce that $R_{\g B}^{\g\gamma_1}(\g w) \neq R_{\g B}^{\g\gamma_2}(\g w)$. This means that $\g w\notin E_{\g B}^x$ and finally $E_{\g B}^x=\vspan\{\g e_1,\dots,\g e_k\}$. 
\end{proof}

\subsection{Null space dimension for tensor fields in $L^p$}

In this subsection, we extend Theorem \ref{theo:C0kcons} to the case of a non-continuous tensor field $\g B\in L^p(\Omega, \mathbb{R}^{ n\times d \times n})$, according to the hypothesis of Proposition \ref{prop:expanded}. We first extend the notion of $k$-conservative tensor fields by density.

\begin{definition}
 For any $k\in\{0,\dots,n\}$, we say that $\g B \in L^p(\Omega, \mathbb{R}^{ n\times d \times n})$ is at least $k$-conservative if it belongs to the closure in $L^p$ of the set of all the at least continuous $k$-conservative tensor fields:
 \begin{equation}\nonumber
 \forall \e>0,\quad \exists \g B_\e \text{ continuous and at least $k$-conservative s.t.}  \quad \norm{\g B-\g B_\e}{L^p}<\e. 
 \end{equation} 
 Then, we say that $\g B \in L^p(\Omega, \mathbb{R}^{n \times d \times n})$ is $k$-conservative if it is at least $k$-conservative and not 
 at least $(k+1)$-conservative. 
\end{definition}

\begin{theorem} \label{theo:kernel-LP} Let $\g B\in L^p(\Omega,\R^{n\times d\times n})$, $p> d$ and $k\in\{0,\dots,n\}$, then
\begin{align*}\nonumber
\dim N(\nabla + \g B\cdot) = k\quad\eq \quad \g B\text{ is $k$-conservative. }
\end{align*}
\end{theorem}

\begin{proof} As in Theorem \ref{theo:C0kcons}, the statement of the present theorem is equivalent to the following:
\begin{equation}\nonumber
\forall k\in \{1,\dots,n\}, \quad \dim N(\nabla + \g B\cdot) \geq k\quad\eq \quad \g B\text{ is at least $k$-conservative}.
\end{equation}
We start here with reciprocal sense.   

Let $\g B\in L^p(\Omega,\R^{n\times d\times n})$ being at least $k$-conservative for $k\geq 1$, then there exists by definition a sequence $\g B_\ell$ of continuous $k$-conservative tensors such that $\g B_\ell\to \g B$ in $L^p$ when $\ell\to +\infty$. The sequence of operators $(\nabla + \g B_\ell\cdot)$ converges to $(\nabla + \g B\cdot)$ in $\cL(L^2(\Omega,\R^n),H^{-1}(\Omega,\R^{n\times d}))$. Indeed, for $\g\mu\in L^2(\Omega,\R^n)$ and $V\in H^{1}_0(\Omega,\R^{n\times d})$, 
\begin{align*}
\inner{((\nabla + \g B\cdot)-(\nabla + \g B_\ell\cdot))\g\mu,V}{H^{-1},H^1_0} &= \inner{(\g B-\g B_\ell)\cdot \g\mu,V}{H^{-1},H^1_0} = \int_\Omega ((\g B-\g B_\ell)\cdot \g\mu) : V^T\\ 
&\leq \norms{\g B-\g B_\ell}{L^p}\norm{\g\mu}{L^2}\norm{V}{L^q}\\
&\leq C \norms{\g B-\g B_\ell}{L^p}\norm{\g\mu}{L^2}\norm{V}{H^1_0},
\end{align*}
where $1/q=1/2-1/p$ and $C>0$ is the Sobolev embedding constant of $H^1_0\hookrightarrow L^q$. Then we get that 
\begin{align*}
\norm{(\nabla + \g B\cdot )-(\nabla + \g B_\ell\cdot)}{L^2,H^{-1}}\leq C \norms{\g B-\g B_\ell}{L^p} \to 0. 
\end{align*}
From Theorem \ref{theo:C0kcons}, we know that for each $\ell\in\N$, $\dim N(\nabla + \g B_\ell\cdot)=k$ as $\g B_\ell$ is continuous and $k$-conservative. Recalling that $(\nabla + \g B\cdot)$ has closed ranged from Theorem \ref{theo:closedrange}, we can apply the following Lemma.

\begin{lemma}\label{prop:lemmaLp} Let $H$ be an Hilbert space and $F$ a normed vector space. Let $(A_\ell)_{\ell\in\N}$ be a sequence of $\cL(H,F)$ that converges to $A\in\cL(H,F)$ for the operator norm. If
\begin{itemize}

\item[-] $A$ has closed range,

\item[-] $\exists k\in\N$ such that $\forall\ell\in\N$, $\dim N(A_\ell) \geq k$,

\end{itemize}
then  $\dim N(A)\geq k$. 
\end{lemma}

The proof of this lemma if given in \ref{app:approx}. We apply this result to the sequence of operators $(\nabla + \g B_\ell\cdot)$  which gives immediately $\dim N(\nabla + \g B\cdot)\geq k$. 
\medskip

Conversely, suppose that $\dim N(\nabla + \g B\cdot) \geq k$. We must show that $\g B$ can be approached by some continuous $k$-conservative tensor fields. Fix $\e>0$ and call $\g B_1^\e\in \cC^\infty(\ol\Omega,\R^{n\times d\times n})$ any smooth approximation of  $\g B$ such that 
\begin{align*} 
\norm{\g B-\g B^\e_1}{L^p}<\e.
\end{align*}
There is no reason for $\g B^\e_1$ to be $k$-conservative. Therefore, we shall construct from $\g B^\e_1$ another smooth approximation  that is $k$-conservative. 

Let $(\g\mu_1,\dots,\g\mu_k)$ be an orthonormal family of $N(\nabla + \g B\cdot)$. From  Proposition \ref{prop:prop}, we know that these functions belongs to $W^{1,p}(\Omega,\R^n)$. We approximate this family by a smooth one, $(\g\mu_1^\e,\dots,\g\mu_k^\e)\in \cC^\infty(\ol\Omega,\R^n)$, such that
\begin{equation}\nonumber
\norm{\g\mu_\ell^\e-\g\mu_\ell}{W^{1,p}}< \e,\quad \forall i\in\{1,\dots,k\}.
\end{equation}
For each $\ell$, we define the residual:
\begin{equation}\label{eq:res}
F_i^\e:= \nabla \g\mu_i^\e + \g B_{1}^\e \cdot \g\mu_i^\e.
\end{equation}
We claim that this residual is small in $L^p$. Indeed, subtracting the equation $\g 0=\nabla \g\mu_\ell + \g B \cdot \g\mu_\ell$ to the previous one, we get 
\begin{equation}\nonumber
\begin{aligned}
F_\ell^\e &:= \nabla \g\mu_\ell^\e-\nabla \g \mu_\ell + \g B_{1}^\e \cdot \g\mu_\ell^\e-\g B \cdot \bm{\mu}_\ell = \nabla(\g\mu_\ell^\e-\g \mu_\ell) + \g B\cdot (\g\mu_\ell^\e-\g \mu_\ell) + (\g B_{1}^\e-\g B)\cdot \g\mu_\ell^\e,\\
\norm{F_\ell^\e}{L^p} &\leq \norm{\g\mu_\ell^\e-\g \mu_\ell}{W^{1,p}}+ \norm{\g B}{L^p}\norm{\g\mu_\ell^\e-\g \mu_\ell}{L^\infty} + \norm{\g B-\g B^\e_1}{L^p}\norm{\g\mu_\ell^\e}{L^\infty}\\
&\leq \norm{\g\mu_\ell^\e-\g \mu_\ell}{W^{1,p}}+ \norm{\g B}{L^p} C \norm{\g\mu_\ell^\e-\g \mu_\ell}{W^{1,p}} + \norm{\g B-\g B^\e_1}{L^p} C \norm{\g\mu_\ell^\e}{W^{1,p}}\\
&\leq \e(1+C(\norm{\g B}{L^p}+ \norm{\g \mu_\ell}{W^{1,p}}+\e)),
\end{aligned}
\end{equation}
where $C>0$ is the embedding constant of $W^{1,p}\hookrightarrow L^\infty$. Hence, for all $\ell$, $\norm{F_\ell^\e}{L^p}\to 0$ when $\e\to 0$. We now look for another small smooth tensor $\g B_2^\e \in \mathcal{C}^\infty(\overline{\Omega}, \R^{n \times d \times n})$ satisfying the system of equations
\begin{align}
	\label{eq:scalar_prod} &\g B_2^\e \cdot \g\mu_\ell^\e = F_\ell^\e, \qquad \forall \ell \in \{1,...,k\}.
\end{align}
Specifically, we search $\g B_2^\e$ under the form
\begin{align*}
	\g B_2^\e :=  \sum_{j=1}^k X_j^\e \otimes \g\mu_j^\e,
\end{align*}
where $(X_1^\e,\dots X_k^\e) $ are smooth matrix fields of $\mathcal{C}^\infty(\overline{\Omega}, \R^{n \times d})$. With this decomposition, \eqref{eq:scalar_prod} reads:
\begin{align*}
	 \sum_{j=1}^k  (\g\mu_j^\e \cdot \g\mu_\ell^\e) X_j^\e = \sum_{j=1}^k  (G^\e)_{\ell j} X_j^\e = F_\ell^\e, \qquad \forall \ell \in \{1,\dots,k\},
\end{align*}
where $G^\e_{\ell j}(x) := \g\mu_\ell^\e(x) \cdot \g\mu_j^\e(x)$ is the Gram matrix of $(\g\mu_\ell^\e(x),..,\g\mu_k^\e(x))$ at any point $x \in \overline{\Omega}$. Denoting $G_{\ell j}(x) := \g\mu_1(x) \cdot \g\mu_j(x)$ the Gram matrix of $(\g\mu_\ell,..,\g\mu_k)$, the uniform convergence of each $\g\mu_\ell^\e$ to $\g\mu_i$ leads to the uniform convergence of $G^\e$ to $G$. The family $(\g\mu_1,\dots,\g\mu_k)$ being linearly independent, from Lemma \ref{lem:gram_positive}, we know that there exists $\alpha >0$ so that $G$ satisfies  
\begin{align*}
	 \bm{v}^T G(x) \bm{v} \ge \alpha |\bm{v}|^2, \qquad  \forall x \in \Omega, \; \forall \bm{v} \in \mathbb{R}^k.
\end{align*}
From the uniform convergence of $G^\e$, there exists $\e_0>0$ so that for all $0 < \e < \e_0$, $G^\e$ satisfies the same property with $\frac{\alpha}{2}$. This means that, for $\e$ sufficiently small, $G^\e(x)$ can be uniformly inverted, \emph{i.e.} there exists $C >0$ such that
\begin{align*} 
	\norm{G^\e(x)^{-1}}{p,p} \le C, \quad \forall \e \le \e_0,\quad\forall x\in \overline{\Omega}.
\end{align*}
We can therefore retrieve the matrices $X_i^\e$ \emph{via} the formula
\begin{align*}
	X_i^\e = \sum_j^k (G^\e)_{ij}^{-1} F_j^\e. 
\end{align*}
These matrix fields are small in $L^p(\Omega,\R^{n\times d})$ as for all $i_0\in\{1,\dots,k\}$,
\begin{align*}
 |X_{i_0}^\e(x)|^p \leq \sum_{i}|X_i^\e(x)|^p &\le C^p \sum_{j} |F_j^\e(x)|^p, \quad \forall \e \le \e_0, \; \forall x \in \overline{\Omega},\\
\|X_{i_0}^\e\|_{L^p}^p &\le C^p \, \sum_j \, \|F_j^\e \|_{L^p}^p, \quad \forall \e \le \e_0,\\
\end{align*}
and then
\begin{equation}\nonumber
\norm{X_{i_0}^\e}{L^p} \leq Cn^{1/p}\e.
\end{equation}
As a consequence, $\|\g B_2^\e\|_{L^p}\le C\e$ for $\e\le \e_0$. Finally, defining $\g B^\e = \g B_1^\e - \g B_2^\e \in \cC^\infty(\ol\Omega,\R^{n\times d\times n}) $,  we immediately remark that from \eqref{eq:res} and \eqref{eq:scalar_prod} we have
\begin{align*}
	\nabla \g\mu_i^\e + \g B^\e \cdot \g\mu_i^\e =\bm{0}, \qquad \forall i\in \{1,\dots,k\}.
\end{align*}
Therefore, the family of linearly independent vectors $(\g\mu_1^\e,\dots, \g\mu_k^\e)$ belongs to $N(\nabla + \g B^\e \cdot)$ and hence $\mathrm{dim} \, N(\nabla + \g B^\e \cdot) \ge k$, \emph{i.e.}, $\g B^\e$ is at least $k$-conservative. Furthermore, as for every $\e >0$, 
\begin{align*}
	\|\g B - \g B^\e \|_{L^p} \le (C +1 )\e,
\end{align*} 
the tensor $\g B\in L^p(\Omega,\R^{n\times d\times n})$ is by definition at least $k$-conservative.
\end{proof}

\section{Numerical applications for inverse elasticity problems in dimension two} 
\label{sec:num}

We illustrate that it is possible in practice to reconstruct a numerical approximation of the solution $\bm{\mu}$ using a discrete resolution of the Reverse Weak Formulation \eqref{eq:RWFcompact}. For that, we extend the finite element methodology proposed in \cite{bretin2023stability} for a single parameter recovery to the multi-parameter elastography problem. We observe some stable reconstructions from one to six parameter maps from both static and dynamic numerical experiments in dimension two. 

Both forward and inverse numerical computation have been made using the FIFEMlab environnement\footnote{https://github.com/seppecher/FIFEMLab}.

\subsection{Discretization of the forward problem}\label{sec:direct}

We set $\Omega$ as the unit square $\Omega:=(-1,1)^2\subset\R^2$, and set the forward problem as 
\begin{equation}\label{eq:forward}
\left\{\begin{aligned}
		-\div \left(\g C : \cE(\g u) \right) = \bm{f} &\quad \textrm{in }\Omega,\\
		\g u = \g 0 &\quad \textrm{on } \Gamma_1,\\
		\g u = \bm{g} &\quad \textrm{on } \Gamma_2,\\
		\g C : \cE(\g u) \cdot \bm{\nu} = \g 0 &\quad \textrm{on } \partial\Omega\bs\Gamma,
\end{aligned}\right.
\end{equation}
where $\Gamma_1$ and $\Gamma_2$ are opposite edges of the square and $\bm{g}$ is a Dirichlet boundary condition on $\Gamma_1$. 

The numerical approximation of \eqref{eq:forward} is computed using the classical $\P^1$ finite element method over an unstructured triangular mesh of resolution $h_{\text{Forward}}:=2.5.10^{-3}$. Using various positioning of $\Gamma_1$ and $\Gamma_2$ as well as various boundary conditions $\g g$, we compute different numerical solutions of \eqref{eq:forward}.

\subsection{Discretization of the inverse problem}\label{sec:discretization}

The discretization of the Reverse Weak Formulation \eqref{eq:RWF2} follows the idea from \cite{bretin2023stability}. A pair of interpolation spaces is chosen $(\cM_h,\cV_h)$ where $\cM_h$ approaches the parameters space $\cM:=L^2(\Omega,\R^n)$ while $\cV_h$ approaches the test function space $\cV:=H^1_0(\Omega,\R^N)$. The discrete version of the problem \eqref{eq:RWF2} reads: 
\begin{equation}\label{eq:RWFdis}\begin{aligned}
\text{Find }\g\mu^h\in \cM_h\quad \text{ s.t. } \quad
\inner{A^h(\g \mu^h) ,\g v^h}{\cV_h',\cV_h} = \inner{\g f^h,\g v^h}{\cV_h',\cV_h},\qquad \forall \g v^h\in \cV_h,
\end{aligned}\end{equation}
where the operator $A^h:\cM_h\to \cV_h' $ is defined by
$$\inner{A^h(\g \mu^h) ,\g v^h}{\cV_h',\cV_h} :=\int_\Omega (\g T_{\g u^h}\cdot\g \mu^h)\cddot (\nabla \g v^h)^ T, \quad \forall (\g \mu^h,\g v^h) \in (\cM_h,\cV_h),$$
and where ${\g u^h}$ and $\g f^h$ are discrete approximations of the data $\g u$ and $\g f$ that may contain both measurement noise and discretization errors.

Call $\{\g e_1,\dots,\g e_{p}\}$ the canonical basis of $\mathcal{V}_h$, $\{\bm{\e},\dots,\bm{\e}_{q}\}$ the canonical basis of $\mathcal{M}_h$ and define the finite element matrices 
\begin{align}\label{eq:matrixA}
\A_{ij} &:= \inner{A^h(\bm{\e}_j) ,\g e_i}{\cV_h',\cV_h} = \int_\Omega (\g T_{\g u^h}\cdot  \bm{\e}_j) \cddot (\nabla \bm{e}_i)^T,\qquad 1\leq i\leq p,\quad 1\leq j\leq q,\\
g_i&:=\inner{\g f^h ,\g e_i}{\cV_h',\cV_h} = \int_\Omega \g f^h\cdot\g e_i,\qquad 1\leq i\leq p. 
\end{align}
We write both $\g\mu^h$ and $\g v^h$ in the bases of $\cM_h$ and $\cV_h$ respectively: 
 \begin{align}
\g\mu^h :=\sum_{i=1}^{q} m_i \bm{\e}_i, \quad\textrm{and}\quad \g v^h:=\sum_{i=1}^{p}  w_i\g e_i,
 \end{align}
 the problem \eqref{eq:RWFdis} simply reads
 \begin{equation}\label{eq:RWFdis2}
 \begin{aligned}
 \g w^T\A \g m &= \g w^T\g g,\quad \forall \g w\in\R^p,\\
 \eq\quad  \A \g m &= \g g. 
 \end{aligned}
 \end{equation}

It turns out that the choice of the finite element pair of spaces $(\cM_h,\cV_h)$ is crucial to preserve the stability of the inversion. The analysis of this kind of discretization has been made in \cite{bretin2023stability} in the case of a single parameter recovery. In this reference, the authors present several pairs that preserve stability, and they suggest the use of a particular pair of finite element spaces called the honeycomb pair $(\P^0_\text{hexa},\P^1_0)$ made of piecewise-constant maps on a hexagonal tilling to approach the parameter space $L^2(\Omega)$ and piecewise linear maps on the canonical sub-triangulation of the hexagonal tilling to approach $H^1_0(\Omega)$. This specific pair shows the best numerical stability while allowing consideration of discontinuous parameter maps. See Subsection 5.2 in \cite{bretin2023stability} for more details.

\subsection{The space-dependent least-squares solution}\label{sec:least-square}

Unlike the finite element method for solving direct PDE problems, the finite element matrix $\A$ defined in \eqref{eq:matrixA} is rectangular and is over-determinate in general. Hence, the linear system \eqref{eq:RWFdis2} may have no solution. A least squares approach is required. 

We point out that since the discrete operator $A^h$ is defined in $\cL(\cM_h,\cV_h')$, in order to approach $A_{\g u}\in\cL(\cM,\cV')$, the norms used for the least squares approach must be these of spaces $\cM_h$ and $\cV_h'$.

\paragraph{Case $\g f^h\neq \g 0$} In this case, the equation \eqref{eq:RWFdis} is replaced by 
\begin{equation}\label{eq:RWFLS}
\begin{aligned}
 (\g\mu^h)^* &:=\argmin_{\g\mu^h \in \cM_h} \max_{\g v^h\in \cV_h} \frac{\inner{A^h(\g \mu^h)-\g f^h ,\g v^h}{\cV_h',\cV_h}}{\norm{\g v}{\cV_h}}. 
\end{aligned}
\end{equation}
Define now the positive definite matrices
\begin{equation}\nonumber
(\S_{\cM})_{ij}=\inner{\bm{\e}_i,\bm{\e}_j}{\cM} \qquad\text{and}\qquad (\S_\cV)_{ij}=\inner{\g e_i ,\g e_j}{\cV}, 
\end{equation}
and their corresponding square-root symmetric matrices $\B_{\cM}$ and $\B_{\cV}$, (i.e. such that $\B_{\cM}^2=\S_{\cM}$ and $\B_{\cV}^2=\S_{\cV}$), then
\begin{equation}\nonumber
\norms{\g\mu^h}{\cM_h} = \norm{\B_{\cM}\g m}{2}\quad\tand\quad \norms{\g v^h}{\cV_h} = \norm{\B_{\cV}\g w}{2}.
\end{equation}
In the canonical bases, the problem \eqref{eq:RWFLS} reads
\begin{align*}
\g m^*&:=\argmin_{\g m\in\R^q}\max_{\g w\in\R^p}\frac{\g w^T(\A \g m-\g g)}{\norm{\B_{\cV}\g w}{2}}=\argmin_{\g m\in\R^q}\max_{\g\xi\in\R^p}\frac{\g\xi^T\B_{\cV}^{-1}(\A \g m-\g g)}{\norm{\g\xi}{2}}\\
&=\argmin_{\g m\in\R^q}\norm{\B_{\cV}^{-1}\A \g m-\B_{\cV}^{-1}\g g}{2}.
\end{align*}
As this is now a classical Euclidian least squares inversion of a linear system, we know that the unique solution satisfies
\begin{equation}\nonumber
\begin{aligned}
\A^T\B_{\cV}^{-1}\B_{\cV}^{-1}\A \g m^* &= \A^T\B_{\cV}^{-1}\B_{\cV}^{-1}\g g,\\
\A^T\S_{\cV}^{-1}\A \g m^* &= \A^T\S_{\cV}^{-1}\g g.
\end{aligned}
\end{equation}
Hence, defining the symmetric matrix $\H:=\A^T\S_{\cV}^{-1}\A$, the least square solution of \eqref{eq:RWFdis} is 
\begin{equation}\label{eq:matrixH}
\begin{aligned}
 (\g\mu^h)^* :=\sum_{i=1}^{q} m_i^* \bm{\e}_i,\quad \text{where}\quad  \g m^*\quad \text{solves}\quad  \H\g m^* =  \A^T\S_{\cV}^{-1}\g g. 
\end{aligned}
\end{equation}

\paragraph{Case $\g f^h= \g 0$}

For the homogeneous problem, which is in fact an eigenvalue problem, we add the condition $\norm{\g\mu^h}{\cM_h}=1$ to \eqref{eq:RWFLS}. In the canonical bases, 
\begin{equation}\nonumber
\begin{aligned}
\g m^*&:=\argmin_{\underset{\norm{\B_{\cM}\g m}{2}=1}{\g m\in\R^q}}\max_{\g w\in\R^p}\frac{\g w^T\A \g m}{\norm{\B_{\cV}\g w}{2}}, \quad \mathrm{or}\\
\g m^*&=\B_{\cM}^{-1}\g z^*\qquad\text{with}\qquad \g z^*=\argmin_{\underset{\norm{\g z}{2}=1}{\g z\in\R^q}}\max_{\g\xi\in\R^p}\frac{\g\xi^T\B_{\cV}^{-1}\A \B_{\cM}^{-1}\g z}{\norm{\g\xi}{2}} = \argmin_{\underset{\norm{\g z}{2}=1}{\g z\in\R^q}}\norm{\B_{\cV}^{-1}\A \B_{\cM}^{-1}\g z}{2}. \\
\end{aligned}
\end{equation}
We remark that $\g z^*$ is the eigenvector associated to the smallest eigenvalue of the positive semidefinite (or definite) matrix $\B_{\cM}^{-1}\A^T \B_{\cV}^{-2}\A \B_{\cM}^{-1} = \B_{\cM}^{-1}\H\, \B_{\cM}^{-1}$. Hence, the solution in the least square sense of \eqref{eq:RWFLS} in the case $\g f^h= \g 0$ is found by solving the problem 
\begin{equation}\label{eq:eig}
\begin{aligned}
 (\g\mu^h)^* :=\sum_{i=1}^{q} m_i^* \bm{\e}_i,\quad \g m^*&=\B_{\cM}^{-1}\g z^*\quad \text{where}\  \g z^*\ \text{is the first eigenvector of}\  \B_{\cM}^{-1}\H\, \B_{\cM}^{-1},
\end{aligned}
\end{equation}
up to a multiplicative constant.

\begin{remark} In our case, as we take the interpolation space $\cM_h$ as a $\P^0$ finite element space over a structured mesh, the matrix $\B_{\cM}$ is of the form $\lambda I$. The coefficient $\lambda$ is the square root of the finite element surface in the mesh. Then the above problem is simplified to
\begin{equation}\label{eq:eig2}
\begin{aligned}
 \text{Find}\quad (\g\mu^h)^* :=\sum_{i=1}^{q} m_i^* \bm{\e}_i,\quad\text{where}\quad  \g m^*\ \text{is the first eigenvector of}\  \H. 
\end{aligned}
\end{equation}

\end{remark}

\begin{remark}
While $\A$ and $\S_V$ are sparse matrices, this is not the case for $\H$ in general. Therefore, a direct solving of the eigenvalue problem \eqref{eq:eig2} might be inefficient for a large matrix $\H$. In order to preserve sparsity, we transform in practice the eigenvalue problem $\H\g m=\alpha \g m$ into 
\begin{equation}\nonumber
\begin{aligned}
\A^T\S_{\cV}^{-1}\A\g m=\alpha\, \g m \quad &\imp\quad \A\A^T\S_{\cV}^{-1}\A\g m=\alpha\, \A\g m,\\
&\imp \quad \A\A^T\S_{\cV}^{-1}\g y =\alpha\,\g y,\quad\text{with}\quad  \A\g m = \g y,\\
&\imp \quad \A\A^T\g z =\alpha\, \S_{\cV}\g z, \,\quad\text{with}\quad  \A\g m = \S_{\cV}\g z.
\end{aligned}
\end{equation}
The new problem  $\A\A^T\g z =\alpha \S_{\cV}\g z$ is now a sparse generalized eigenvalue problem that can be solved efficiently from native algorithms (such as \texttt{eigs} in Maltab for instance) even for very large matrices.

\end{remark}

\subsection{Reconstruction of a single parameter from a single static measurement}\label{sec:recon-mu}

In the first numerical test, we assume that the elastic tensor is of the form $\g C:=\mu^\tex \g I$ and $\g f=\bm{0}$. The map $\mu^\tex$ is chosen piecewise constant and we test two different boundary conditions
\begin{align*}
\g g_1=
\left[\begin{matrix}
1\\
-0.5
\end{matrix}\right]
\quad\tand\quad
\g g_2=
\left[\begin{matrix}
0\\
-1
\end{matrix}\right],
\end{align*}
applied on $\Gamma_2$, as illustrated in Figure \ref{fig:mu-exact}. 

We first compute the displacement field $\g u$ and the associated strain tensor $\cE(\g u)$ using the boundary condition $\g g_1$  (see Figure \ref{fig:direct1}). In this case, the tensor $\g T_{\g u}$ is given by $(\g T_{\g u})_{ij} := \cE(\g u)_{ij}$. The condition stated in Proposition \ref{prop:expanded} saying that $\g T_{\g u}$ must be invertible everywhere in the domain of interest $D$ is fulfilled as the maximum condition number of $\cE(\g u)$ is $13.8$. We then compute the numerical solution $\bm{\mu}^h$, up to a multiplicative constant, by solving the eigenvalue problem \eqref{eq:eig2}. We rescale the solution norm to compare it to the exact map $\mu^\tex$. The relative $L^2$-error between the exact solution and the reconstruction is $2.91\%$ (see Figure \ref{fig:mu-recon-case}).

\begin{figure}
\begin{center}
\begin{tikzpicture}[scale=1.7*\graphscale]
		\draw[ultra thick] (-1, -1) -- (-1, 1) [red]; 
		\draw[ultra thick] (-1, 1) -- (1, 1) [blue];    
		\draw[ultra thick] (1, 1) -- (1, -1) [red];         
		\draw[ultra thick] (1, -1) -- (-1, -1) [black!60!green];       
		
		\node[above] at (0.5, 1) {{$\Gamma_2$}};
		
		\node[below] at (-0.5, -1) {{$\Gamma_1$}};
		
		\draw[->, thick] (-1.4, 0) -- (1.4, 0) node[right] {$x_2$};
		\draw[->, thick] (0, -1.4) -- (0, 1.4) node[above] {$x_1$};
		\node[below right] at (1, 0) {1};
		\node[below right] at (0, 1) {1};
		\node[below right] at (-1, 0) {-1};
		\node[below right] at (0, -1) {-1};
		\node at (1, 0) {\textbullet};
		\node at (-1, 0) {\textbullet};
		\node at (0, 1) {\textbullet};
		\node at (0, -1) {\textbullet};
		
		
    		\draw[thick] (-0.6, -0.6) -- (-0.6, 0.6) ; 
			\draw[thick] (-0.6, 0.6) -- (0.6, 0.6)   ;  
			\draw[thick] (0.6, 0.6) -- (0.6, -0.6)   ;         
			\draw[thick] (0.6, -0.6) -- (-0.6, -0.6) ;
		\end{tikzpicture}
		\hspace{2cm}
		\begin{tikzpicture}[scale=1.25*\scale]
    		\begin{axis}[width=\width, height=\height, axis on top, scale only axis, xmin=\xmin, xmax=\xmax, ymin=\ymin, ymax=\ymax, xtick={\xmin,\pointsize*\xmax + (1-\pointsize)*\xmin,...,\xmax},ytick={\ymin,\pointsize*\ymax + (1-\pointsize)*\ymin,...,\ymax}, colormap/jet, colorbar,point meta min=3.5,point meta max=8.5]
    			\addplot graphics [xmin=\xmin,xmax=\xmax,ymin=\ymin,ymax=\ymax]{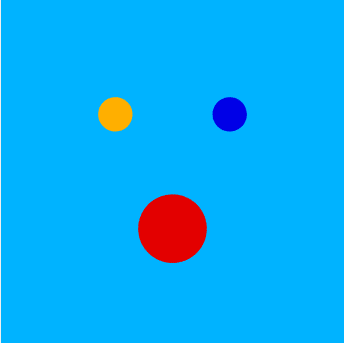};
    			\draw[thick] (-0.6, -0.6) -- (-0.6, 0.6) ; 
				\draw[thick] (-0.6, 0.6) -- (0.6, 0.6)   ;  
				\draw[thick] (0.6, 0.6) -- (0.6, -0.6)   ;         
				\draw[thick] (0.6, -0.6) -- (-0.6, -0.6) ;
    		\end{axis}
		\end{tikzpicture}
\caption{Left: the domain $\Omega$, the boundary parts $\Gamma_1$ and $\Gamma_2$. Right: the exact map $\mu$. On both, the black square $D = (-0.6,0.6)^2$ is the subdomain of interest where the inversion will be performed.}
\label{fig:mu-exact}
\end{center}
\end{figure}
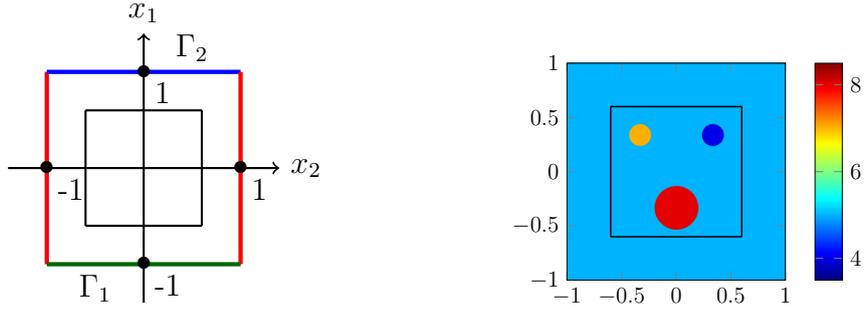

\begin{figure}
\begin{center}
\begin{tikzpicture}[scale=\scale]
    \begin{axis}[width=\width, height=\height, axis on top, scale only axis, xmin=\xmin, xmax=\xmax, ymin=\ymin, ymax=\ymax, xtick={\xmin,\pointsize*\xmax + (1-\pointsize)*\xmin,...,\xmax},ytick={\ymin,\pointsize*\ymax + (1-\pointsize)*\ymin,...,\ymax}, colormap/jet, colorbar,point meta min=0,point meta max=1,title={$\g u_{1}$}]
    \addplot graphics [xmin=\xmin,xmax=\xmax,ymin=\ymin,ymax=\ymax]{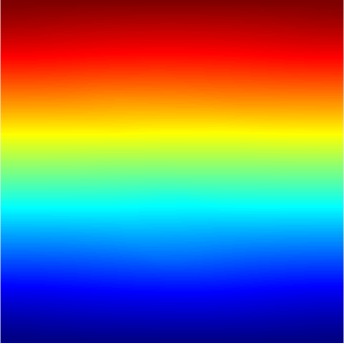};
    \end{axis}
\end{tikzpicture}
\begin{tikzpicture}[scale=\scale]
    \begin{axis}[width=\width, height=\height, axis on top, scale only axis, xmin=\xmin, xmax=\xmax, ymin=\ymin, ymax=\ymax, xtick={\xmin,\pointsize*\xmax + (1-\pointsize)*\xmin,...,\xmax},ytick={\ymin,\pointsize*\ymax + (1-\pointsize)*\ymin,...,\ymax}, colormap/jet, colorbar,point meta min=-0.60135,point meta max=0.09985,title={$\g u_{2}$}]
    \addplot graphics [xmin=\xmin,xmax=\xmax,ymin=\ymin,ymax=\ymax]{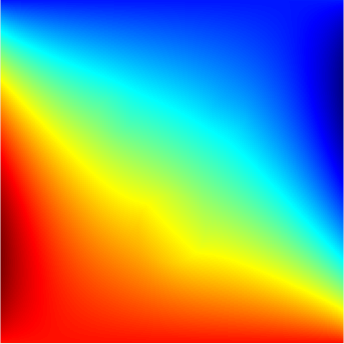};
    \end{axis}
\end{tikzpicture}\\
\bigskip

\begin{tikzpicture}[scale=\scale]
    \begin{axis}[width=\width, height=\height, axis on top, scale only axis, xmin=\xmin, xmax=\xmax, ymin=\ymin, ymax=\ymax, xtick={\xmin,\pointsize*\xmax + (1-\pointsize)*\xmin,...,\xmax},ytick={\ymin,\pointsize*\ymax + (1-\pointsize)*\ymin,...,\ymax}, colormap/jet, colorbar,point meta min=-0.072681,point meta max=0.079901,title={$\cE(\g u)_{11}$}]
    \addplot graphics [xmin=\xmin,xmax=\xmax,ymin=\ymin,ymax=\ymax]{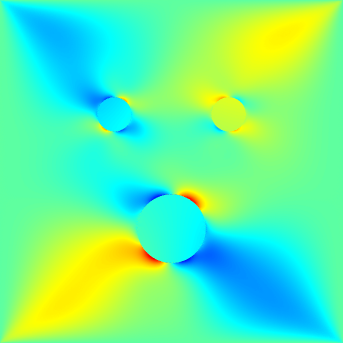};
    \end{axis}
\end{tikzpicture}
        \begin{tikzpicture}[scale=\scale]
    \begin{axis}[width=\width, height=\height, axis on top, scale only axis, xmin=\xmin, xmax=\xmax, ymin=\ymin, ymax=\ymax, xtick={\xmin,\pointsize*\xmax + (1-\pointsize)*\xmin,...,\xmax},ytick={\ymin,\pointsize*\ymax + (1-\pointsize)*\ymin,...,\ymax}, colormap/jet, colorbar,point meta min=0.00021189,point meta max=0.25641,title={$\cE(\g u)_{21}=\cE(\g u)_{12}$}]
    \addplot graphics [xmin=\xmin,xmax=\xmax,ymin=\ymin,ymax=\ymax]{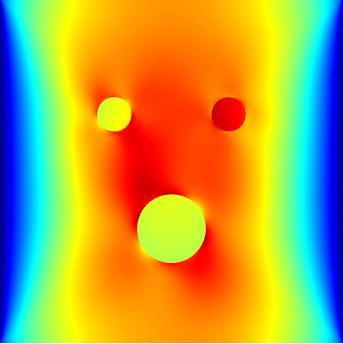};
    \end{axis}
\end{tikzpicture}
 \begin{tikzpicture}[scale=\scale]
    \begin{axis}[width=\width, height=\height, axis on top, scale only axis, xmin=\xmin, xmax=\xmax, ymin=\ymin, ymax=\ymax, xtick={\xmin,\pointsize*\xmax + (1-\pointsize)*\xmin,...,\xmax},ytick={\ymin,\pointsize*\ymax + (1-\pointsize)*\ymin,...,\ymax}, colormap/jet, colorbar,point meta min=-1.182,point meta max=0.66523,title={$\cE(\g u)_{22}$}]
    \addplot graphics [xmin=\xmin,xmax=\xmax,ymin=\ymin,ymax=\ymax]{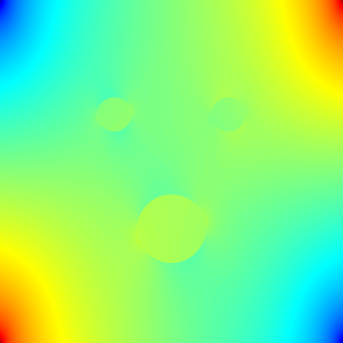};
    \end{axis}
\end{tikzpicture}
\end{center}
\caption{Numerical solution of the forward problem with the boundary condition $\g g_1$. Top: the computed displacement field. Bottom: the corresponding strain $\cE(\g u)$. The strain is everywhere invertible as its maximum condition number is $13.8$. }
\label{fig:direct1}
\end{figure}


\begin{figure}
\def\xmin{-0.6}\def\xmax{0.6}\def\ymin{-0.6}\def\ymax{0.6}
\begin{center}
\begin{tikzpicture}[scale=\scale]
    \begin{axis}[width=\width, height=\height, axis on top, scale only axis, xmin=\xmin, xmax=\xmax, ymin=\ymin, ymax=\ymax, xtick={\xmin,\pointsize*\xmax + (1-\pointsize)*\xmin,...,\xmax},ytick={\ymin,\pointsize*\ymax + (1-\pointsize)*\ymin,...,\ymax}, colormap/jet, colorbar,point meta min=3.5,point meta max=8.5,title={Exact map}]
    \addplot graphics [xmin=\xmin,xmax=\xmax,ymin=\ymin,ymax=\ymax]{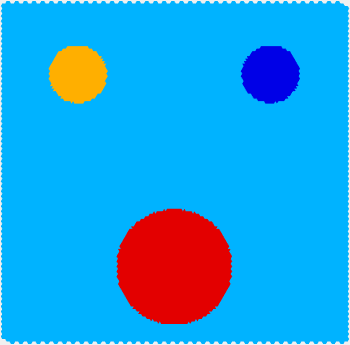};
    \end{axis}
\end{tikzpicture}
\hspace{1cm}
\begin{tikzpicture}[scale=\scale]
    \begin{axis}[width=\width, height=\height, axis on top, scale only axis, xmin=\xmin, xmax=\xmax, ymin=\ymin, ymax=\ymax, xtick={\xmin,\pointsize*\xmax + (1-\pointsize)*\xmin,...,\xmax},ytick={\ymin,\pointsize*\ymax + (1-\pointsize)*\ymin,...,\ymax}, colormap/jet, colorbar,point meta min=3.5,point meta max=8.5,title={Reconstructed map}]
    \addplot graphics [xmin=\xmin,xmax=\xmax,ymin=\ymin,ymax=\ymax]{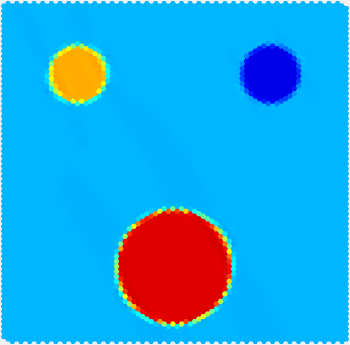};
    \end{axis}
\end{tikzpicture}
\caption{The exact map $\mu^\tex$ and the reconstructed map $\mu$ from the data computed with the boundary condition $\g g_1$. The mesh resolution is $h=0.01$ and the relative $L^2$-error is $2.91\%$.}
\label{fig:mu-recon-case}
\end{center}
\end{figure}

We then compute the displacement and strain using the boundary condition $\g g_2$, as shown in Figure \ref{fig:direct2}. We can observe that the strain tensor is now degenerated in the domain of interest with a maximum condition number of $9.32.10^7$. Therefore, the hypothesis of Proposition \ref{prop:expanded} is not fulfilled and we expect an unstable inversion, as observed in Figure \ref{fig:mu-non-recon-case}, where the reconstruction failed with a relative $L^2$-error of $18.3\%$. More precisely, the instability occurs in the horizontal direction, which corresponds to the degenerate direction of the strain $\cE(\g u)$.

\begin{figure}
\begin{center}
	\def\xmin{-1}\def\xmax{1}\def\ymin{-1}\def\ymax{1}\def\pointsize{0.25}\def\width{\mainwidth}\def\height{\mainheight}
        \begin{tikzpicture}[scale=\scale]
    \begin{axis}[width=\width, height=\height, axis on top, scale only axis, xmin=\xmin, xmax=\xmax, ymin=\ymin, ymax=\ymax, xtick={\xmin,\pointsize*\xmax + (1-\pointsize)*\xmin,...,\xmax},ytick={\ymin,\pointsize*\ymax + (1-\pointsize)*\ymin,...,\ymax}, colormap/jet, colorbar,point meta min=-0.004262,point meta max=0.0047156,title={$\g u_{1}$}]
    \addplot graphics [xmin=\xmin,xmax=\xmax,ymin=\ymin,ymax=\ymax]{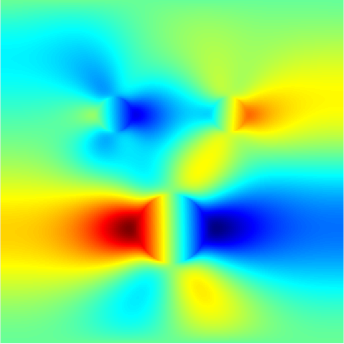};
    \end{axis}
\end{tikzpicture}
\begin{tikzpicture}[scale=\scale]
    \begin{axis}[width=\width, height=\height, axis on top, scale only axis, xmin=\xmin, xmax=\xmax, ymin=\ymin, ymax=\ymax, xtick={\xmin,\pointsize*\xmax + (1-\pointsize)*\xmin,...,\xmax},ytick={\ymin,\pointsize*\ymax + (1-\pointsize)*\ymin,...,\ymax}, colormap/jet, colorbar,point meta min=-1,point meta max=0,title={$\g u_{2}$}]
    \addplot graphics [xmin=\xmin,xmax=\xmax,ymin=\ymin,ymax=\ymax]{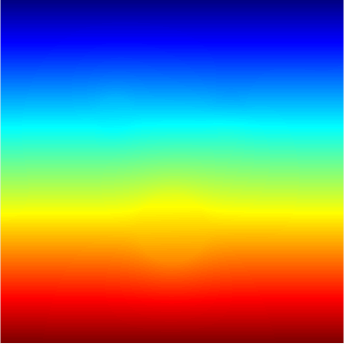};
    \end{axis}
\end{tikzpicture}\\
\bigskip

\begin{tikzpicture}[scale=\scale]
    \begin{axis}[width=\width, height=\height, axis on top, scale only axis, xmin=\xmin, xmax=\xmax, ymin=\ymin, ymax=\ymax, xtick={\xmin,\pointsize*\xmax + (1-\pointsize)*\xmin,...,\xmax},ytick={\ymin,\pointsize*\ymax + (1-\pointsize)*\ymin,...,\ymax}, colormap/jet, colorbar,point meta min=-0.037974,point meta max=0.032518,title={$\cE(\g u)_{11}$}]
    \addplot graphics [xmin=\xmin,xmax=\xmax,ymin=\ymin,ymax=\ymax]{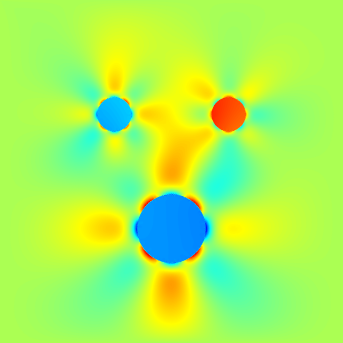};
    \end{axis}
\end{tikzpicture}
\begin{tikzpicture}[scale=\scale]
    \begin{axis}[width=\width, height=\height, axis on top, scale only axis, xmin=\xmin, xmax=\xmax, ymin=\ymin, ymax=\ymax, xtick={\xmin,\pointsize*\xmax + (1-\pointsize)*\xmin,...,\xmax},ytick={\ymin,\pointsize*\ymax + (1-\pointsize)*\ymin,...,\ymax}, colormap/jet, colorbar,point meta min=-0.085949,point meta max=0.087775,title={$\cE(\g u)_{21}=\cE(\g u)_{12}$}]
    \addplot graphics [xmin=\xmin,xmax=\xmax,ymin=\ymin,ymax=\ymax]{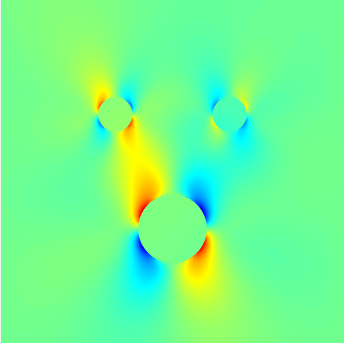};
    \end{axis}
\end{tikzpicture}
\begin{tikzpicture}[scale=\scale]
    \begin{axis}[width=\width, height=\height, axis on top, scale only axis, xmin=\xmin, xmax=\xmax, ymin=\ymin, ymax=\ymax, xtick={\xmin,\pointsize*\xmax + (1-\pointsize)*\xmin,...,\xmax},ytick={\ymin,\pointsize*\ymax + (1-\pointsize)*\ymin,...,\ymax}, colormap/jet, colorbar,point meta min=-0.62445,point meta max=-0.32843,title={$\cE(\g u)_{22}$}]
    \addplot graphics [xmin=\xmin,xmax=\xmax,ymin=\ymin,ymax=\ymax]{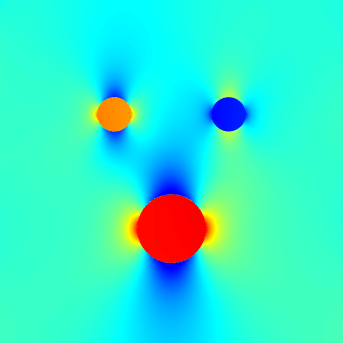};
    \end{axis}
\end{tikzpicture}
\caption{Solution of the forward problem with the boundary condition $\g g_2$. Top: the computed displacement field. Bottom: the corresponding strain $\cE(\g u)$. The strain is numerically degenerated as its maximum condition number of $\cE(\g u)$ is $9.32.10^7$. }
	\label{fig:direct2}
\end{center}
\end{figure}

\begin{figure}
\begin{center}
\def\xmin{-0.6}\def\xmax{0.6}\def\ymin{-0.6}\def\ymax{0.6}\def\pointsize{0.25}\def\width{\mainwidth}\def\height{\mainheight}
	\begin{tikzpicture}[scale=1.2*\scale]
    \begin{axis}[width=\width, height=\height, axis on top, scale only axis, xmin=\xmin, xmax=\xmax, ymin=\ymin, ymax=\ymax, xtick={\xmin,\pointsize*\xmax + (1-\pointsize)*\xmin,...,\xmax},ytick={\ymin,\pointsize*\ymax + (1-\pointsize)*\ymin,...,\ymax}, colormap/jet, colorbar,point meta min=3.5,point meta max=8.5,title={Exact map}]
    \addplot graphics [xmin=\xmin,xmax=\xmax,ymin=\ymin,ymax=\ymax]{mu_exact_small_mu_1.png};
    \end{axis}
\end{tikzpicture}
 \begin{tikzpicture}[scale=1.2*\scale]
    \begin{axis}[width=\width, height=\height, axis on top, scale only axis, xmin=\xmin, xmax=\xmax, ymin=\ymin, ymax=\ymax, xtick={\xmin,\pointsize*\xmax + (1-\pointsize)*\xmin,...,\xmax},ytick={\ymin,\pointsize*\ymax + (1-\pointsize)*\ymin,...,\ymax}, colormap/jet, colorbar,point meta min=3.5,point meta max=8.5,title={Reconstructed map}]
    \addplot graphics [xmin=\xmin,xmax=\xmax,ymin=\ymin,ymax=\ymax]{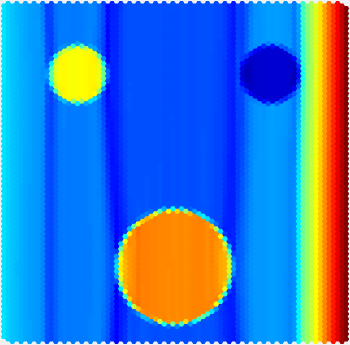};
    \end{axis}
\end{tikzpicture}
\caption{The exact map $\mu^\tex$ and the reconstructed map $\mu$ from the data field computed with the boundary condition $\g g_2$. The mesh resolution is $h=0.01$ and the relative $L^2$-error is $18.3\%$.}
\label{fig:mu-non-recon-case}
\end{center}

\end{figure}

The stability of these two cases can be illustrated through the computation of the first eigenvalues of the matrix $\H$ defined in \eqref{eq:matrixH}. Indeed, as the numerical solution $\g\mu^h$ is computed as the first eigenvector of this matrix, the stability of this recovery is related to the spectral gap between the two first eigenvalues:  $\alpha_2(\mathbb{H}) - \alpha_1(\mathbb{H})$. In Figure \ref{fig:mu-spectral}, we represent the ten first eigenvalues of $\H$ for the two cases $\g g_1$ and $\g g_2$. As expected, the spectral gap is high in the first case ($\approx 0.11$) and low in the second case ($\approx 1.2.10^{-3}$).

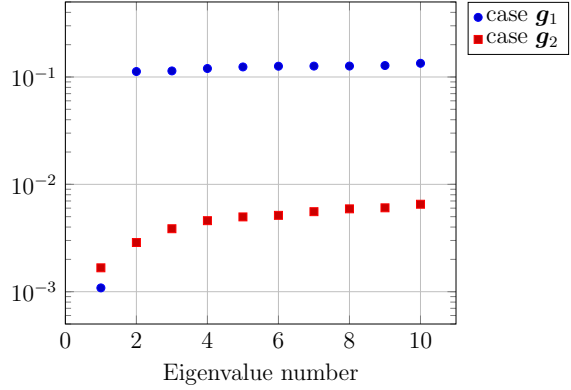
\begin{figure} \begin{center}\begin{tikzpicture}[scale=\graphscale]\begin{semilogyaxis}
   [xmin=0, xmax=11, ymin=0.0005, ymax=0.5, grid = major, 
   xlabel = Eigenvalue number,
   legend entries ={case $\g g_1$,case $\g g_2$}, legend pos=outer north east,]

\addplot+[only marks] table{
1    0.0010849
2    0.11247
3    0.11391
4    0.12
5    0.12409
6    0.12602
7    0.12636
8    0.12645
9    0.12786
10    0.13419
};
\addplot+[only marks] table{
1    0.0016686
2    0.0028706
3    0.0038592
4    0.0045873
5    0.004973
6    0.0051441
7    0.0055633
8    0.0059116
9    0.0060437
10    0.0065281
};
\end{semilogyaxis}\end{tikzpicture}\end{center}
\vspace{-0.5cm}
\caption{The ten smallest eigenvalues of the matrix $\H$ under the boundary conditions $\g g_1$ (blue disks) and $\g g_2$ (red squares). The spectral gap is $\alpha_2-\alpha_1\approx 0.11$ in the first case and $\alpha_2-\alpha_1\approx 1.2.10^{-3}$ in the second case.}
\label{fig:mu-spectral}
\end{figure}

\subsection{Reconstruction of the Lamé parameters from static measurements}\label{sec:recon-lame}

We consider here an elastic tensor $\g C$ of the form $\g C := \mu^\tex \bm{I} + \lambda^\tex I \otimes I $. The maps $\mu^\tex$ and $\lambda^\tex$ are defined as in Figure \ref{fig:lambda-mu-exact}, and we use the boundary conditions as in Figure \ref{fig:lambda-mu-BC}.

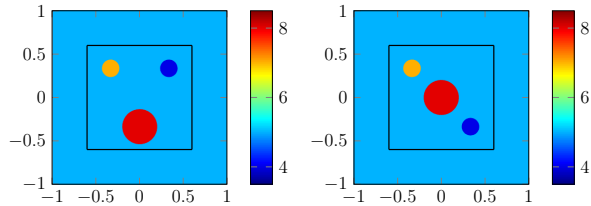
\begin{figure}
\begin{center}
        \def\xmin{-1}\def\xmax{1}\def\ymin{-1}\def\ymax{1}
        \begin{tikzpicture}[scale=\scale]
    		\begin{axis}[width=\width, height=\height, axis on top, scale only axis, xmin=\xmin, xmax=\xmax, ymin=\ymin, ymax=\ymax, xtick={\xmin,\pointsize*\xmax + (1-\pointsize)*\xmin,...,\xmax},ytick={\ymin,\pointsize*\ymax + (1-\pointsize)*\ymin,...,\ymax}, colormap/jet, colorbar,point meta min=3.5,point meta max=8.5]
    			\addplot graphics [xmin=\xmin,xmax=\xmax,ymin=\ymin,ymax=\ymax]{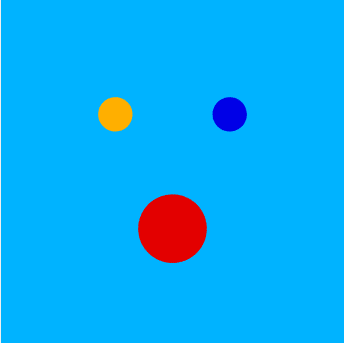};
    			\draw[thick] (-0.6, -0.6) -- (-0.6, 0.6) ; 
				\draw[thick] (-0.6, 0.6) -- (0.6, 0.6)   ;  
				\draw[thick] (0.6, 0.6) -- (0.6, -0.6)   ;         
				\draw[thick] (0.6, -0.6) -- (-0.6, -0.6) ;
    		\end{axis}
		\end{tikzpicture}
             \begin{tikzpicture}[scale=\scale]
    		\begin{axis}[width=\width, height=\height, axis on top, scale only axis, xmin=\xmin, xmax=\xmax, ymin=\ymin, ymax=\ymax, xtick={\xmin,\pointsize*\xmax + (1-\pointsize)*\xmin,...,\xmax},ytick={\ymin,\pointsize*\ymax + (1-\pointsize)*\ymin,...,\ymax}, colormap/jet, colorbar,point meta min=3.5,point meta max=8.5]
    			\addplot graphics [xmin=\xmin,xmax=\xmax,ymin=\ymin,ymax=\ymax]{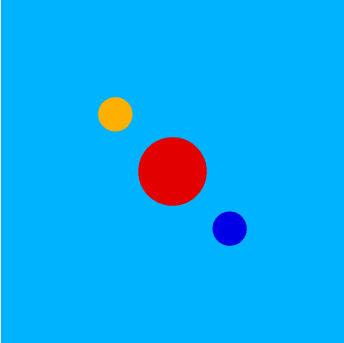};
    			\draw[thick] (-0.6, -0.6) -- (-0.6, 0.6) ; 
				\draw[thick] (-0.6, 0.6) -- (0.6, 0.6)   ;  
				\draw[thick] (0.6, 0.6) -- (0.6, -0.6)   ;         
				\draw[thick] (0.6, -0.6) -- (-0.6, -0.6) ;
    		\end{axis}
		\end{tikzpicture}
	\caption{The maps $\mu^\tex$ (left) and $\lambda^\tex$ (right) and the domain of interest $D$ in black. }
	 \label{fig:lambda-mu-exact}
\end{center}
\end{figure}

\begin{figure}
	\centering

		\begin{tikzpicture}[scale=1.7*\graphscale]
		\draw[ultra thick] (-1, -1) -- (-1, 1) [red]; 
		\draw[ultra thick] (-1, 1) -- (1, 1) [blue];    
		\draw[ultra thick] (1, 1) -- (1, -1) [red];         
		\draw[ultra thick] (1, -1) -- (-1, -1) [black!60!green];       
		
		\node[above] at (-1.3, 1) {{$\Gamma_2$}};
		
		\node[below] at (-0.5, -1) {{$\Gamma_1$}};
		
		
		
		\draw[->, thick] (-1.6, 0) -- (1.6, 0) node[right] {$x_1$};
		\draw[->, thick] (0, -1.6) -- (0, 1.6) node[above] {$x_2$};
		\node[below right] at (1, 0) {1};
		\node[below right] at (0, 1) {1};
		\node[below right] at (-1, 0) {-1};
		\node[below right] at (0, -1) {-1};
		\node at (1, 0) {\textbullet};
		\node at (-1, 0) {\textbullet};
		\node at (0, 1) {\textbullet};
		\node at (0, -1) {\textbullet};
		
		\node[right] at (0.3, 1.70) {$\bm{g} = \begin{pmatrix}1\\-0.5\end{pmatrix} $}; %
		\draw[->, thick] (0.4, 1.5) -- (1, 1.1) ;
		\draw[->, thick] (0, 1.5) -- (0.6, 1.1) ;
		\draw[->, thick] (-0.4, 1.5) -- (0.2, 1.1) ;
		\draw[->, thick] (-0.8, 1.5) -- (-0.2, 1.1) ;
		\draw[->, thick] (-1.2, 1.5) -- (-0.6, 1.1) ;

		
    		\draw[thick] (-0.6, -0.6) -- (-0.6, 0.6) ; 
			\draw[thick] (-0.6, 0.6) -- (0.6, 0.6)   ;  
			\draw[thick] (0.6, 0.6) -- (0.6, -0.6)   ;         
			\draw[thick] (0.6, -0.6) -- (-0.6, -0.6) ;
		\end{tikzpicture}
		\begin{tikzpicture}[scale=1.7*\graphscale]
		\draw[ultra thick] (-1, -1) -- (-1, 1) [black!60!green]; 
		\draw[ultra thick] (-1, 1) -- (1, 1) [red];    
		\draw[ultra thick] (1, 1) -- (1, -1) [blue];         
		\draw[ultra thick] (1, -1) -- (-1, -1) [red];       
		
		
		
		\node[right] at (1, 1.3) {{$\Gamma_2$}};
		
		\node[left] at (-1, 0.5) {{$\Gamma_1$}};
		
		\draw[->, thick] (-1.6, 0) -- (1.6, 0) node[right] {$x_1$};
		\draw[->, thick] (0, -1.6) -- (0, 1.6) node[above] {$x_2$};
		\node[below right] at (1, 0) {1};
		\node[below right] at (0, 1) {1};
		\node[below right] at (-1, 0) {-1};
		\node[below right] at (0, -1) {-1};
		\node at (1, 0) {\textbullet};
		\node at (-1, 0) {\textbullet};
		\node at (0, 1) {\textbullet};
		\node at (0, -1) {\textbullet};
		
		\node[right] at (1.65, 0.45) {$\bm{g}= \begin{pmatrix}-0.5\\ 1 \end{pmatrix}$}; 
		\draw[->, thick] (1.5, -1.2) -- (1.1, -0.6) ;
		\draw[->, thick] (1.5, -0.8) -- (1.1, -0.2) ;
		\draw[->, thick] (1.5, -0.4) -- (1.1, 0.2) ;
		\draw[->, thick] (1.5, 0.0) -- (1.1, 0.6) ;
		\draw[->, thick] (1.5, 0.4) -- (1.1, 1) ;
		
    		\draw[thick] (-0.6, -0.6) -- (-0.6, 0.6) ; 
			\draw[thick] (-0.6, 0.6) -- (0.6, 0.6)   ;  
			\draw[thick] (0.6, 0.6) -- (0.6, -0.6)   ;         
			\draw[thick] (0.6, -0.6) -- (-0.6, -0.6) ;
		\end{tikzpicture}
	
	\caption{The domain $\Omega$ and the two settings of boundary conditions used for the reconstruction of the Lamé parameters.}
	 \label{fig:lambda-mu-BC}
\end{figure}
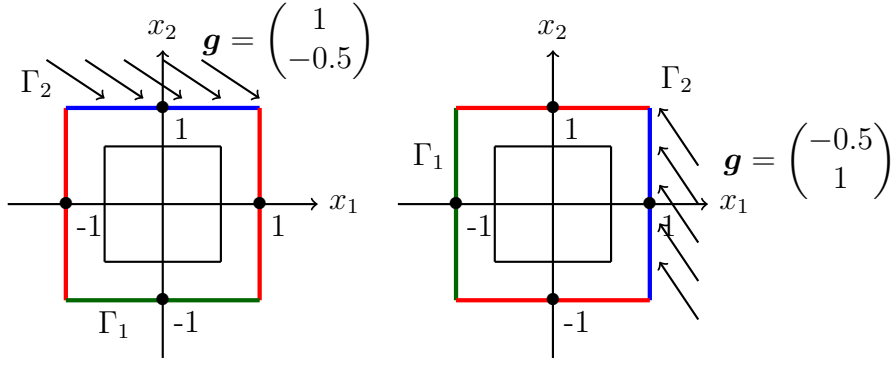

We plot the reconstruction of the Lamé parameters in Figure \ref{fig:lambda-mu-recon} using respectively $m=1$ or $m=2$ displacements fields. As expected, the inversion fails in the first case as the tensor $\g T_{\g u}$ cannot be left-invertible, while the inversion is stable in the second case. Additional data does not increase significantly the stability, as shown in Figure \ref{fig:lambda-mu-spectral} and Table \ref{tab:lambda-mu-error}.

\begin{figure}
\begin{center}
	\def\xmin{-0.6}\def\xmax{0.6}\def\ymin{-0.6}\def\ymax{0.6}\def\width{\mainwidth}\def\height{\mainheight}
	\begin{tikzpicture}[scale=\scale]
    \begin{axis}[width=\width, height=\height, axis on top, scale only axis, xmin=\xmin, xmax=\xmax, ymin=\ymin, ymax=\ymax, xtick={\xmin,\pointsize*\xmax + (1-\pointsize)*\xmin,...,\xmax},ytick={\ymin,\pointsize*\ymax + (1-\pointsize)*\ymin,...,\ymax}, colormap/jet, colorbar,point meta min=3.5,point meta max=8.5,ylabel={$\mu$}, ylabel style={rotate=270},title={Exact map}]
    \addplot graphics [xmin=\xmin,xmax=\xmax,ymin=\ymin,ymax=\ymax]{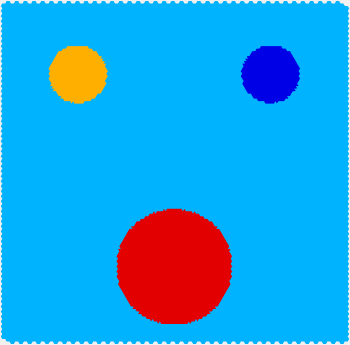};
    \end{axis}
\end{tikzpicture}
        \begin{tikzpicture}[scale=\scale]
    \begin{axis}[width=\width, height=\height, axis on top, scale only axis, xmin=\xmin, xmax=\xmax, ymin=\ymin, ymax=\ymax, xtick={\xmin,\pointsize*\xmax + (1-\pointsize)*\xmin,...,\xmax},ytick={\ymin,\pointsize*\ymax + (1-\pointsize)*\ymin,...,\ymax}, colormap/jet, colorbar,point meta min=3.5,point meta max=8.5]
    \addplot graphics [xmin=\xmin,xmax=\xmax,ymin=\ymin,ymax=\ymax]{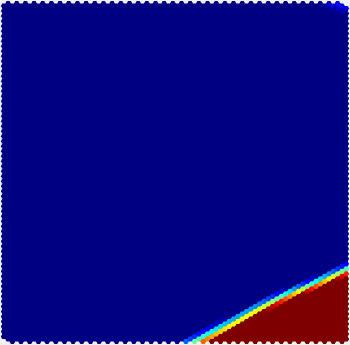};
    \end{axis}
\end{tikzpicture}
\begin{tikzpicture}[scale=\scale]
    \begin{axis}[width=\width, height=\height, axis on top, scale only axis, xmin=\xmin, xmax=\xmax, ymin=\ymin, ymax=\ymax, xtick={\xmin,\pointsize*\xmax + (1-\pointsize)*\xmin,...,\xmax},ytick={\ymin,\pointsize*\ymax + (1-\pointsize)*\ymin,...,\ymax}, colormap/jet, colorbar,point meta min=3.5,point meta max=8.5,title={Recontruction with $2$ data}]
    \addplot graphics [xmin=\xmin,xmax=\xmax,ymin=\ymin,ymax=\ymax]{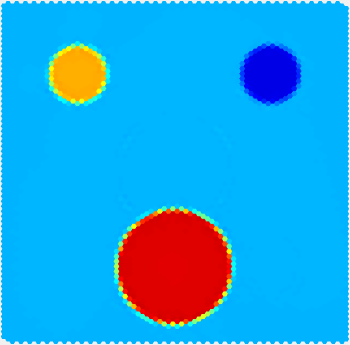};
    \end{axis}
\end{tikzpicture}

\begin{tikzpicture}[scale=\scale]
    \begin{axis}[width=\width, height=\height, axis on top, scale only axis, xmin=\xmin, xmax=\xmax, ymin=\ymin, ymax=\ymax, xtick={\xmin,\pointsize*\xmax + (1-\pointsize)*\xmin,...,\xmax},ytick={\ymin,\pointsize*\ymax + (1-\pointsize)*\ymin,...,\ymax}, colormap/jet, colorbar,point meta min=3.5,point meta max=8.5,ylabel={$\lambda$}, ylabel style={rotate=270},title={Exact map}]
    \addplot graphics [xmin=\xmin,xmax=\xmax,ymin=\ymin,ymax=\ymax]{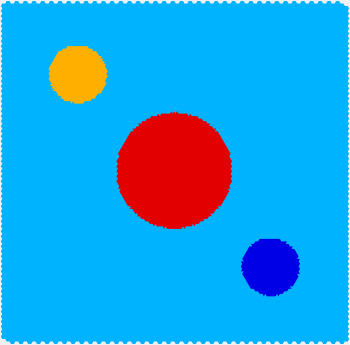};
    \end{axis}
\end{tikzpicture}
	        \begin{tikzpicture}[scale=\scale]
    \begin{axis}[width=\width, height=\height, axis on top, scale only axis, xmin=\xmin, xmax=\xmax, ymin=\ymin, ymax=\ymax, xtick={\xmin,\pointsize*\xmax + (1-\pointsize)*\xmin,...,\xmax},ytick={\ymin,\pointsize*\ymax + (1-\pointsize)*\ymin,...,\ymax}, colormap/jet, colorbar,point meta min=3.5,point meta max=8.5]
    \addplot graphics [xmin=\xmin,xmax=\xmax,ymin=\ymin,ymax=\ymax]{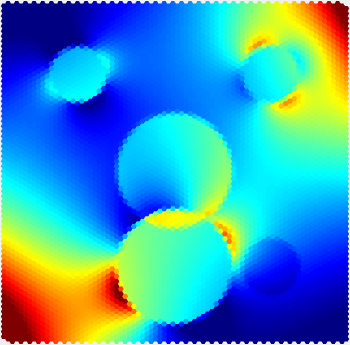};
    \end{axis}
\end{tikzpicture}
	        \begin{tikzpicture}[scale=\scale]
    \begin{axis}[width=\width, height=\height, axis on top, scale only axis, xmin=\xmin, xmax=\xmax, ymin=\ymin, ymax=\ymax, xtick={\xmin,\pointsize*\xmax + (1-\pointsize)*\xmin,...,\xmax},ytick={\ymin,\pointsize*\ymax + (1-\pointsize)*\ymin,...,\ymax}, colormap/jet, colorbar,point meta min=3.5,point meta max=8.5,title={Reconstruction from $2$ data}]
    \addplot graphics [xmin=\xmin,xmax=\xmax,ymin=\ymin,ymax=\ymax]{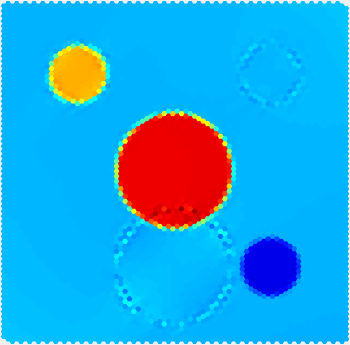};
    \end{axis}
\end{tikzpicture}
	\caption{Reconstruction of Lamé parameters from static measurements. From left to right: the exact maps $\mu^\tex$ and $\lambda^\tex$, reconstructions with $m=1$ and $m = 2$ data field. The mesh resolution is $h=0.01$.}
	\label{fig:lambda-mu-recon}
\end{center}	
\end{figure}

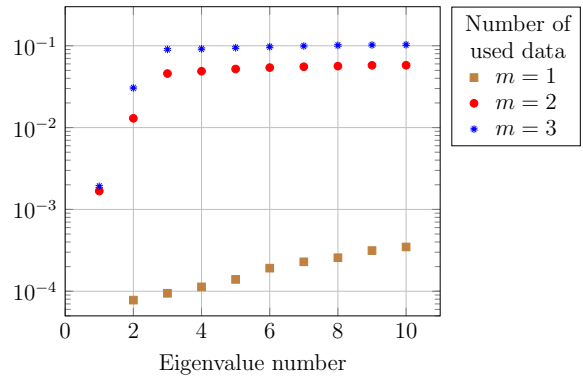
\begin{figure}
\begin{center}
\begin{tikzpicture}[scale=0.96*\graphscale]\begin{semilogyaxis}
   [xmin=0, xmax=11, ymin=0.00005, ymax=0.3, grid = major, 
   xlabel = Eigenvalue number,
   legend entries = {\hspace{-.3cm}Number of,\hspace{-.3cm}used data,$m=1$,$m=2$,$m=3$,}, legend pos=outer north east,]
   \addlegendimage{empty legend}
   \addlegendimage{empty legend}
\addplot[only marks,mark=square*,color=brown] table{
1    0
2    7.7926e-05
3    9.4437e-05
4    0.00011311
5    0.00013956
6    0.00019158
7    0.00022877
8    0.00025675
9    0.00031396
10    0.00034767
};\addplot[only marks,mark=*,color=red] table{
1    0.001676
2    0.01298
3    0.045751
4    0.048826
5    0.051991
6    0.054065
7    0.055418
8    0.056382
9    0.057655
10    0.057771
};\addplot[only marks,mark=10-pointed star,color=blue] table{
1    0.0019121
2    0.030492
3    0.09022
4    0.091645
5    0.094651
6    0.097144
7    0.099352
8    0.10119
9    0.10227
10    0.1029
};
\end{semilogyaxis}\end{tikzpicture}\end{center}
\vspace{-0.5cm}
	\caption{The ten smallest eigenvalues of the matrix $\H$ $(\alpha_1,\dots,\alpha_{10})$, computed from, one (brown squares), two (red dots) and three (blue stars) displacement fields. The spectral gaps is $\alpha_2-\alpha_1 \approx 7.8.10^{-5}$ in the first case, $\alpha_2-\alpha_1 \approx 1.1.10^{-2}$ in the second case and $\alpha_2-\alpha_1 \approx 2.9.10^{-2}$ in third case.}
	\label{fig:lambda-mu-spectral}
\end{figure}

\begin{figure} 
\centering
	\begin{tabular}{cccc}
\toprule
Number of used data & \multicolumn{2}{c}{Relative $L^2$-error (\%)}  \\ 
 & $\mu$ & $\lambda$   \\
\midrule
1 & 108.79 & 26.11  \\
2 & 2.89 & 2.98   \\
3 & 2.88 & 2.89  \\ 
\bottomrule
\end{tabular}

	\caption{The relative $L^2$-error between the exact maps and the reconstructed maps for the reconstruction of Lamé parameters from static measurements.}
	\label{tab:lambda-mu-error}
\end{figure}

\subsection{Reconstruction of the Lamé parameters from time-harmonic measurements}\label{sec:recon-lame-harmo}

We now propose to reconstruct the Lamé coefficient in the case of time harmonic measures. 
The direct model is now the time-harmonic wave equation:
\begin{equation}\label{eq:helmholtz}
\left\{\begin{aligned}
		-\div \left(\g C : \cE(\g u) \right) -	\omega^2\g u= \bm{0} &\quad \textrm{in }\Omega,\\
		\g u = \g 0 &\quad \textrm{on } \Gamma_1,\\
		\g u = \g g &\quad \textrm{on } \Gamma_2,\\
		\g C : \cE(\g u) \cdot \bm{\nu} = \g 0 &\quad \textrm{on } \partial\Omega\bs\Gamma.
\end{aligned}\right.
\end{equation} 
We chose the frequency $\omega$ in the set $\{10, 20, 30, 40, 50\}$. We present in Figure \ref{fig:lambda-mu-oscillations-recon} the reconstruction of the two Lamé parameters, using either one or two distinct frequencies. These reconstructions are performed by solving the linear system \eqref{eq:matrixH}. In Table \ref{tab:lambda-mu-oscillations-error}, we present the relative errors of the reconstruction using one to five numbers of frequencies. As expected, we see that at least two independent displacement fields are required to perform a stable inversion.

\begin{figure}
\begin{center}
	\def\xmin{-0.6}\def\xmax{0.6}\def\ymin{-0.6}\def\ymax{0.6}\def\width{\mainwidth}\def\height{\mainheight}
	\begin{tikzpicture}[scale=\scale]
    \begin{axis}[width=\width, height=\height, axis on top, scale only axis, xmin=\xmin, xmax=\xmax, ymin=\ymin, ymax=\ymax, xtick={\xmin,\pointsize*\xmax + (1-\pointsize)*\xmin,...,\xmax},ytick={\ymin,\pointsize*\ymax + (1-\pointsize)*\ymin,...,\ymax}, colormap/jet, colorbar,point meta min=3.5,point meta max=8.5,ylabel={$\mu$}, ylabel style={rotate=270},title={Exact map}]
    \addplot graphics [xmin=\xmin,xmax=\xmax,ymin=\ymin,ymax=\ymax]{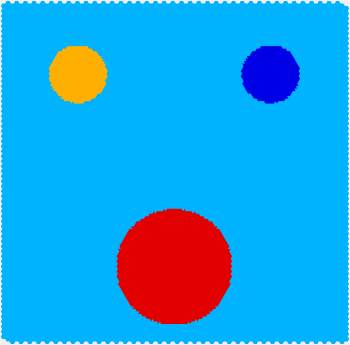};
    \end{axis}
\end{tikzpicture}
\begin{tikzpicture}[scale=\scale]
    \begin{axis}[width=\width, height=\height, axis on top, scale only axis, xmin=\xmin, xmax=\xmax, ymin=\ymin, ymax=\ymax, xtick={\xmin,\pointsize*\xmax + (1-\pointsize)*\xmin,...,\xmax},ytick={\ymin,\pointsize*\ymax + (1-\pointsize)*\ymin,...,\ymax}, colormap/jet, colorbar,point meta min=3.5,point meta max=8.5]
    \addplot graphics [xmin=\xmin,xmax=\xmax,ymin=\ymin,ymax=\ymax]{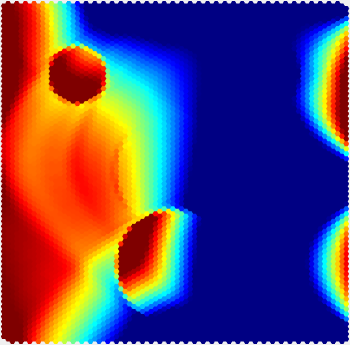};
    \end{axis}
\end{tikzpicture}
\begin{tikzpicture}[scale=\scale]
    \begin{axis}[width=\width, height=\height, axis on top, scale only axis, xmin=\xmin, xmax=\xmax, ymin=\ymin, ymax=\ymax, xtick={\xmin,\pointsize*\xmax + (1-\pointsize)*\xmin,...,\xmax},ytick={\ymin,\pointsize*\ymax + (1-\pointsize)*\ymin,...,\ymax}, colormap/jet, colorbar,point meta min=3.5,point meta max=8.5]
    \addplot graphics [xmin=\xmin,xmax=\xmax,ymin=\ymin,ymax=\ymax]{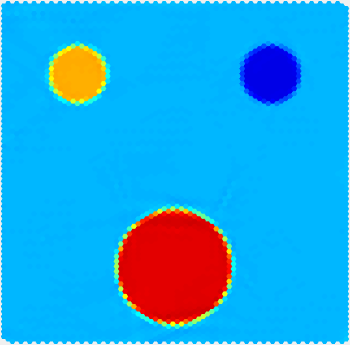};
    \end{axis}
\end{tikzpicture}
		
\begin{tikzpicture}[scale=\scale]
    \begin{axis}[width=\width, height=\height, axis on top, scale only axis, xmin=\xmin, xmax=\xmax, ymin=\ymin, ymax=\ymax, xtick={\xmin,\pointsize*\xmax + (1-\pointsize)*\xmin,...,\xmax},ytick={\ymin,\pointsize*\ymax + (1-\pointsize)*\ymin,...,\ymax}, colormap/jet, colorbar,point meta min=3.5,point meta max=8.5,ylabel={$\lambda$}, ylabel style={rotate=270}]
    \addplot graphics [xmin=\xmin,xmax=\xmax,ymin=\ymin,ymax=\ymax]{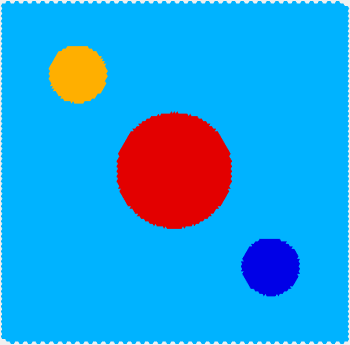};
    \end{axis}
\end{tikzpicture}
\begin{tikzpicture}[scale=\scale]
    \begin{axis}[width=\width, height=\height, axis on top, scale only axis, xmin=\xmin, xmax=\xmax, ymin=\ymin, ymax=\ymax, xtick={\xmin,\pointsize*\xmax + (1-\pointsize)*\xmin,...,\xmax},ytick={\ymin,\pointsize*\ymax + (1-\pointsize)*\ymin,...,\ymax}, colormap/jet, colorbar,point meta min=3.5,point meta max=8.5]
    \addplot graphics [xmin=\xmin,xmax=\xmax,ymin=\ymin,ymax=\ymax]{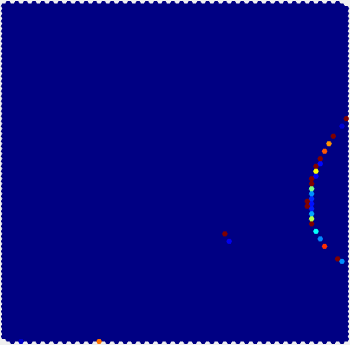};
    \end{axis}
\end{tikzpicture}
        \begin{tikzpicture}[scale=\scale]
    \begin{axis}[width=\width, height=\height, axis on top, scale only axis, xmin=\xmin, xmax=\xmax, ymin=\ymin, ymax=\ymax, xtick={\xmin,\pointsize*\xmax + (1-\pointsize)*\xmin,...,\xmax},ytick={\ymin,\pointsize*\ymax + (1-\pointsize)*\ymin,...,\ymax}, colormap/jet, colorbar,point meta min=3.5,point meta max=8.5]
    \addplot graphics [xmin=\xmin,xmax=\xmax,ymin=\ymin,ymax=\ymax]{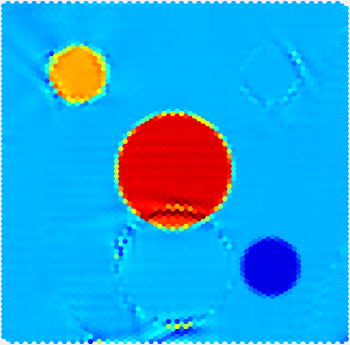};
    \end{axis}
\end{tikzpicture}

	\caption{Reconstruction of Lamé parameters from time-harmonic measurements. From left to right: the exact map $\mu^\tex$ and $\lambda^\tex$, their reconstructions using $1$ frequency and their reconstructions using $2$ frequencies. The mesh resolution  $h=0.01$ for the inversion.}
	\label{fig:lambda-mu-oscillations-recon}
	\end{center}
\end{figure}

\begin{figure}
\centering
	\begin{tabular}{cccc}
\toprule
Number of & \multicolumn{2}{c}{Relative $L^2$-error (\%)} \\ 
\cmidrule(lr){2-3}
used frequencies & $\mu$ & $\lambda$ \\
\midrule
1 & 52.47 & 136.66  \\
2 & 2.92 & 3.42  \\
3 & 2.91 & 2.98 \\ 
4 & 2.90 & 2.92 \\
5 & 2.91 & 3.05 \\
\bottomrule
\end{tabular}

	\caption{The relative $L^2$-error between the exact maps and the reconstructed maps for the reconstruction of Lamé parameters from time-harmonic measurements.}
	 \label{tab:lambda-mu-oscillations-error}
\end{figure}


\subsection{Reconstruction of the fully anisotropic elastic tensor from static measurements}\label{sec:recon-aniso}
We consider a fully anisotropic tensor model in dimension two:
\begin{equation}\nonumber
\g C:=\sum_{k=1}^6 \mu_k\g C_k,
\end{equation}
where the constant tensors $\bm{C}_1,...,\bm{C}_6$ are chosen as an orthogonal basis of the symmetric elastic tensors that reads, using Voigt notation, as
\begin{align*}
	& \g C_1 = \begin{pmatrix}
		\sqrt{2} & 0 & 0 \\
		0 & 0 & 0 \\
		0 & 0 & 0
	\end{pmatrix}, \quad \g C_2 = \begin{pmatrix}
		0 & 0 & 0 \\
		0 & 0 & 0 \\
		0 & 0 & \sqrt{2}
	\end{pmatrix}, \quad \g C_3 = \begin{pmatrix}
		0 & 0 & 0 \\
		0 & \sqrt{2} & 0 \\
		0 & 0 & 0
	\end{pmatrix},\\
	 & \g C_4 = \begin{pmatrix}
		0 & 0 & 1 \\
		0 & 0 & 0 \\
		1 & 0 & 0
	\end{pmatrix}, \quad \g C_5 = \begin{pmatrix}
		0 & 1 & 0 \\
		1 & 0 & 0 \\
		0 & 0 & 0
	\end{pmatrix}, \quad \g C_6 = \begin{pmatrix}
		0 & 0 & 0 \\
		0 & 0 & 1 \\
		0 & 1 & 0
	\end{pmatrix}.
\end{align*}
\par In Figure \ref{fig:aniso-recon}, we plot the reconstruction of the six parameter maps $\mu_1,...,\mu_6$, directly obtained by solving the eigenvalue problem \eqref{eq:eig2}, using $m=3$, $m=4$ and $m=6$ data fields respectively. We observe that from the use of $m \ge 4$ data fields, the reconstruction of the six parameter maps is stable, while it fails for $m < 4$. This illustrates the ability of our approach to recover fully anisotropic tensors from a direct discretization of the Reverse Weak Formulation \eqref{eq:RWF} or \eqref{eq:RWFcompact} from a rather low number of data fields. To our knowledge, this is the first method that allows such reconstruction from elasto-static measurements.

\begin{remark}
In practice, we do not check the validity of the hypotheses of Proposition \ref{prop:expanded} nor Corollary \ref{cor:stability-static} before the discretization. Instead, we compute the matrix $\mathbb{H}$ of the discretized system and check the stability of the inversion, looking at the first eigenvalues of this matrix.
\end{remark}

\begin{remark}
Surprisingly, the stable reconstruction is possible for $m = 4$, while we expected that it would only be possible for $m\ge 6$. As mentioned in Remark \ref{rq:number_parameters}, the theoretical stability analysis requires the same number of data than the number of maps to be recover.
\end{remark}
In Figure \ref{fig:aniso-spectral}, we plot the first eigenvalue of the matrix $\H$ for different number of used data fields and in Table \ref{tab:aniso-error}, we present the relative $L^2$-errors for the whole tensor recovery.

\begin{figure}
\begin{center}
\def\scale{\textscale*0.11}
	\def\xmin{-0.6}\def\xmax{0.6}\def\ymin{-0.6}\def\ymax{0.6}\def\pointsize{0.5}\def\width{0.53*\mainwidth}\def\height{0.53*\mainheight}

        \begin{tikzpicture}[scale=\scale]
    \begin{axis}[width=\width, height=\height, axis on top, scale only axis, xmin=\xmin, xmax=\xmax, ymin=\ymin, ymax=\ymax, xtick={\xmin,\pointsize*\xmax + (1-\pointsize)*\xmin,...,\xmax},ytick={\ymin,\pointsize*\ymax + (1-\pointsize)*\ymin,...,\ymax}, colormap/jet, colorbar,point meta min=12.5,point meta max=19.5,colorbar style={xshift=-7pt,width=9pt},title={$\mu_1$}]
    \addplot graphics [xmin=\xmin,xmax=\xmax,ymin=\ymin,ymax=\ymax]{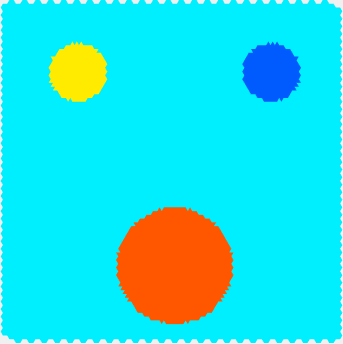};
    \end{axis}
\end{tikzpicture}
        \begin{tikzpicture}[scale=\scale]
    \begin{axis}[width=\width, height=\height, axis on top, scale only axis, xmin=\xmin, xmax=\xmax, ymin=\ymin, ymax=\ymax, xtick={\xmin,\pointsize*\xmax + (1-\pointsize)*\xmin,...,\xmax},ytick={\ymin,\pointsize*\ymax + (1-\pointsize)*\ymin,...,\ymax}, colormap/jet, colorbar,point meta min=12.5,point meta max=19.5,colorbar style={xshift=-7pt,width=9pt},title={$\mu_2$}]
    \addplot graphics [xmin=\xmin,xmax=\xmax,ymin=\ymin,ymax=\ymax]{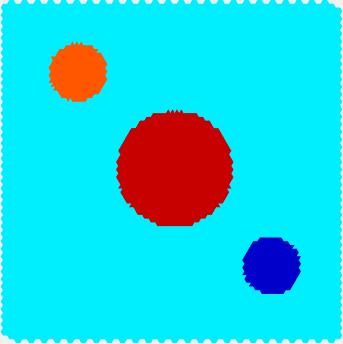};
    \end{axis}
\end{tikzpicture}
        \begin{tikzpicture}[scale=\scale]
    \begin{axis}[width=\width, height=\height, axis on top, scale only axis, xmin=\xmin, xmax=\xmax, ymin=\ymin, ymax=\ymax, xtick={\xmin,\pointsize*\xmax + (1-\pointsize)*\xmin,...,\xmax},ytick={\ymin,\pointsize*\ymax + (1-\pointsize)*\ymin,...,\ymax}, colormap/jet, colorbar,point meta min=12.5,point meta max=19.5,colorbar style={xshift=-7pt,width=9pt},title={$\mu_3$}]
    \addplot graphics [xmin=\xmin,xmax=\xmax,ymin=\ymin,ymax=\ymax]{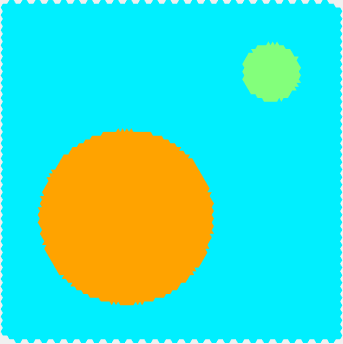};
    \end{axis}
\end{tikzpicture}
        \begin{tikzpicture}[scale=\scale]
    \begin{axis}[width=\width, height=\height, axis on top, scale only axis, xmin=\xmin, xmax=\xmax, ymin=\ymin, ymax=\ymax, xtick={\xmin,\pointsize*\xmax + (1-\pointsize)*\xmin,...,\xmax},ytick={\ymin,\pointsize*\ymax + (1-\pointsize)*\ymin,...,\ymax}, colormap/jet, colorbar,point meta min=0,point meta max=4.5,colorbar style={xshift=-7pt,width=9pt},title={$\mu_4$}]
    \addplot graphics [xmin=\xmin,xmax=\xmax,ymin=\ymin,ymax=\ymax]{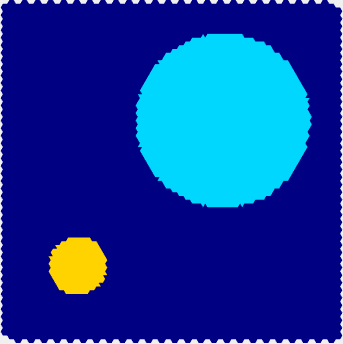};
    \end{axis}
\end{tikzpicture}
        \begin{tikzpicture}[scale=\scale]
    \begin{axis}[width=\width, height=\height, axis on top, scale only axis, xmin=\xmin, xmax=\xmax, ymin=\ymin, ymax=\ymax, xtick={\xmin,\pointsize*\xmax + (1-\pointsize)*\xmin,...,\xmax},ytick={\ymin,\pointsize*\ymax + (1-\pointsize)*\ymin,...,\ymax}, colormap/jet, colorbar,point meta min=0,point meta max=4.5,colorbar style={xshift=-7pt,width=9pt},title={$\mu_5$}]
    \addplot graphics [xmin=\xmin,xmax=\xmax,ymin=\ymin,ymax=\ymax]{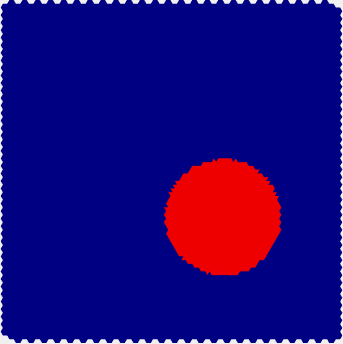};
    \end{axis}
\end{tikzpicture}
        \begin{tikzpicture}[scale=\scale]
    \begin{axis}[width=\width, height=\height, axis on top, scale only axis, xmin=\xmin, xmax=\xmax, ymin=\ymin, ymax=\ymax, xtick={\xmin,\pointsize*\xmax + (1-\pointsize)*\xmin,...,\xmax},ytick={\ymin,\pointsize*\ymax + (1-\pointsize)*\ymin,...,\ymax}, colormap/jet, colorbar,point meta min=0,point meta max=4.5,colorbar style={xshift=-7pt,width=9pt},title={$\mu_6$}]
    \addplot graphics [xmin=\xmin,xmax=\xmax,ymin=\ymin,ymax=\ymax]{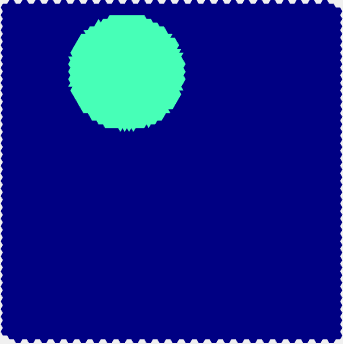};
    \end{axis}
\end{tikzpicture}

 \begin{tikzpicture}[scale=\scale]
    \begin{axis}[width=\width, height=\height, axis on top, scale only axis, xmin=\xmin, xmax=\xmax, ymin=\ymin, ymax=\ymax, xtick={\xmin,\pointsize*\xmax + (1-\pointsize)*\xmin,...,\xmax},ytick={\ymin,\pointsize*\ymax + (1-\pointsize)*\ymin,...,\ymax}, colormap/jet, colorbar,point meta min=12.5,point meta max=19.5,colorbar style={xshift=-7pt,width=9pt},title={$\mu_1$}]
    \addplot graphics [xmin=\xmin,xmax=\xmax,ymin=\ymin,ymax=\ymax]{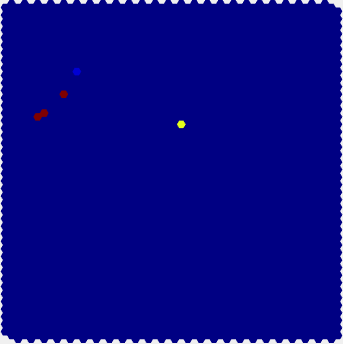};
    \end{axis}
\end{tikzpicture}
        \begin{tikzpicture}[scale=\scale]
    \begin{axis}[width=\width, height=\height, axis on top, scale only axis, xmin=\xmin, xmax=\xmax, ymin=\ymin, ymax=\ymax, xtick={\xmin,\pointsize*\xmax + (1-\pointsize)*\xmin,...,\xmax},ytick={\ymin,\pointsize*\ymax + (1-\pointsize)*\ymin,...,\ymax}, colormap/jet, colorbar,point meta min=12.5,point meta max=19.5,colorbar style={xshift=-7pt,width=9pt},title={$\mu_2$}]
    \addplot graphics [xmin=\xmin,xmax=\xmax,ymin=\ymin,ymax=\ymax]{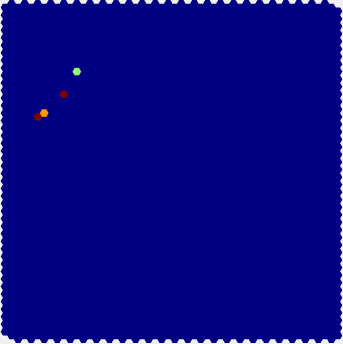};
    \end{axis}
\end{tikzpicture}
        \begin{tikzpicture}[scale=\scale]
    \begin{axis}[width=\width, height=\height, axis on top, scale only axis, xmin=\xmin, xmax=\xmax, ymin=\ymin, ymax=\ymax, xtick={\xmin,\pointsize*\xmax + (1-\pointsize)*\xmin,...,\xmax},ytick={\ymin,\pointsize*\ymax + (1-\pointsize)*\ymin,...,\ymax}, colormap/jet, colorbar,point meta min=12.5,point meta max=19.5,colorbar style={xshift=-7pt,width=9pt},title={$\mu_3$}]
    \addplot graphics [xmin=\xmin,xmax=\xmax,ymin=\ymin,ymax=\ymax]{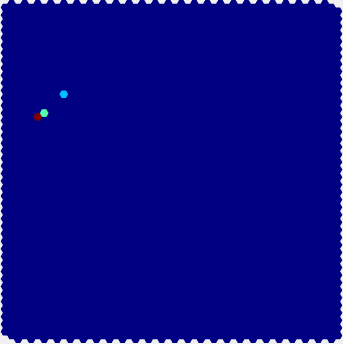};
    \end{axis}
\end{tikzpicture}
        \begin{tikzpicture}[scale=\scale]
    \begin{axis}[width=\width, height=\height, axis on top, scale only axis, xmin=\xmin, xmax=\xmax, ymin=\ymin, ymax=\ymax, xtick={\xmin,\pointsize*\xmax + (1-\pointsize)*\xmin,...,\xmax},ytick={\ymin,\pointsize*\ymax + (1-\pointsize)*\ymin,...,\ymax}, colormap/jet, colorbar,point meta min=0,point meta max=4.5,colorbar style={xshift=-7pt,width=9pt},title={$\mu_4$}]
    \addplot graphics [xmin=\xmin,xmax=\xmax,ymin=\ymin,ymax=\ymax]{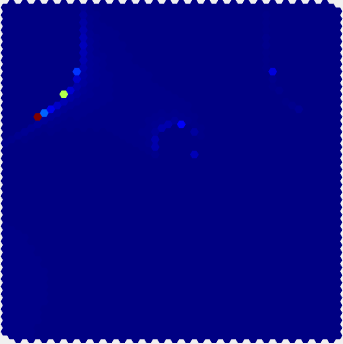};
    \end{axis}
\end{tikzpicture}
        \begin{tikzpicture}[scale=\scale]
    \begin{axis}[width=\width, height=\height, axis on top, scale only axis, xmin=\xmin, xmax=\xmax, ymin=\ymin, ymax=\ymax, xtick={\xmin,\pointsize*\xmax + (1-\pointsize)*\xmin,...,\xmax},ytick={\ymin,\pointsize*\ymax + (1-\pointsize)*\ymin,...,\ymax}, colormap/jet, colorbar,point meta min=0,point meta max=4.5,colorbar style={xshift=-7pt,width=9pt},title={$\mu_5$}]
    \addplot graphics [xmin=\xmin,xmax=\xmax,ymin=\ymin,ymax=\ymax]{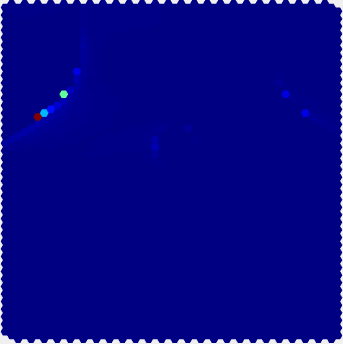};
    \end{axis}
\end{tikzpicture}
        \begin{tikzpicture}[scale=\scale]
    \begin{axis}[width=\width, height=\height, axis on top, scale only axis, xmin=\xmin, xmax=\xmax, ymin=\ymin, ymax=\ymax, xtick={\xmin,\pointsize*\xmax + (1-\pointsize)*\xmin,...,\xmax},ytick={\ymin,\pointsize*\ymax + (1-\pointsize)*\ymin,...,\ymax}, colormap/jet, colorbar,point meta min=0,point meta max=4.5,colorbar style={xshift=-7pt,width=9pt},title={$\mu_6$}]
    \addplot graphics [xmin=\xmin,xmax=\xmax,ymin=\ymin,ymax=\ymax]{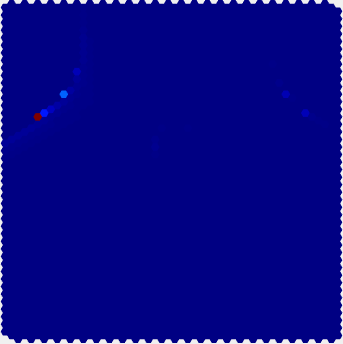};
    \end{axis}
\end{tikzpicture}

        \begin{tikzpicture}[scale=\scale]
    \begin{axis}[width=\width, height=\height, axis on top, scale only axis, xmin=\xmin, xmax=\xmax, ymin=\ymin, ymax=\ymax, xtick={\xmin,\pointsize*\xmax + (1-\pointsize)*\xmin,...,\xmax},ytick={\ymin,\pointsize*\ymax + (1-\pointsize)*\ymin,...,\ymax}, colormap/jet, colorbar,point meta min=12.5,point meta max=19.5,colorbar style={xshift=-7pt,width=9pt},title={$\mu_1$}]
    \addplot graphics [xmin=\xmin,xmax=\xmax,ymin=\ymin,ymax=\ymax]{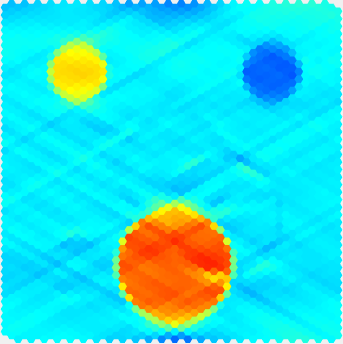};
    \end{axis}
\end{tikzpicture}
        \begin{tikzpicture}[scale=\scale]
    \begin{axis}[width=\width, height=\height, axis on top, scale only axis, xmin=\xmin, xmax=\xmax, ymin=\ymin, ymax=\ymax, xtick={\xmin,\pointsize*\xmax + (1-\pointsize)*\xmin,...,\xmax},ytick={\ymin,\pointsize*\ymax + (1-\pointsize)*\ymin,...,\ymax}, colormap/jet, colorbar,point meta min=12.5,point meta max=19.5,colorbar style={xshift=-7pt,width=9pt},title={$\mu_2$}]
    \addplot graphics [xmin=\xmin,xmax=\xmax,ymin=\ymin,ymax=\ymax]{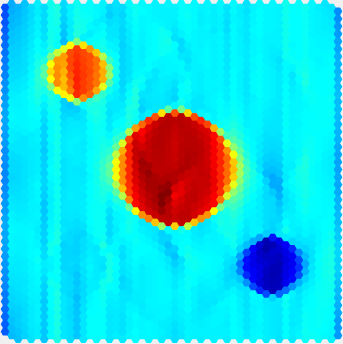};
    \end{axis}
\end{tikzpicture}
        \begin{tikzpicture}[scale=\scale]
    \begin{axis}[width=\width, height=\height, axis on top, scale only axis, xmin=\xmin, xmax=\xmax, ymin=\ymin, ymax=\ymax, xtick={\xmin,\pointsize*\xmax + (1-\pointsize)*\xmin,...,\xmax},ytick={\ymin,\pointsize*\ymax + (1-\pointsize)*\ymin,...,\ymax}, colormap/jet, colorbar,point meta min=12.5,point meta max=19.5,colorbar style={xshift=-7pt,width=9pt},title={$\mu_3$}]
    \addplot graphics [xmin=\xmin,xmax=\xmax,ymin=\ymin,ymax=\ymax]{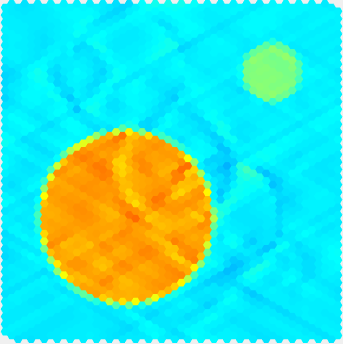};
    \end{axis}
\end{tikzpicture}
        \begin{tikzpicture}[scale=\scale]
    \begin{axis}[width=\width, height=\height, axis on top, scale only axis, xmin=\xmin, xmax=\xmax, ymin=\ymin, ymax=\ymax, xtick={\xmin,\pointsize*\xmax + (1-\pointsize)*\xmin,...,\xmax},ytick={\ymin,\pointsize*\ymax + (1-\pointsize)*\ymin,...,\ymax}, colormap/jet, colorbar,point meta min=0,point meta max=4.5,colorbar style={xshift=-7pt,width=9pt},title={$\mu_4$}]
    \addplot graphics [xmin=\xmin,xmax=\xmax,ymin=\ymin,ymax=\ymax]{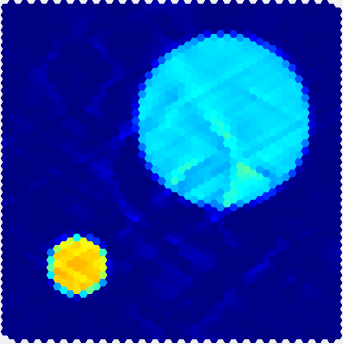};
    \end{axis}
\end{tikzpicture}
        \begin{tikzpicture}[scale=\scale]
    \begin{axis}[width=\width, height=\height, axis on top, scale only axis, xmin=\xmin, xmax=\xmax, ymin=\ymin, ymax=\ymax, xtick={\xmin,\pointsize*\xmax + (1-\pointsize)*\xmin,...,\xmax},ytick={\ymin,\pointsize*\ymax + (1-\pointsize)*\ymin,...,\ymax}, colormap/jet, colorbar,point meta min=0,point meta max=4.5,colorbar style={xshift=-7pt,width=9pt},title={$\mu_5$}]
    \addplot graphics [xmin=\xmin,xmax=\xmax,ymin=\ymin,ymax=\ymax]{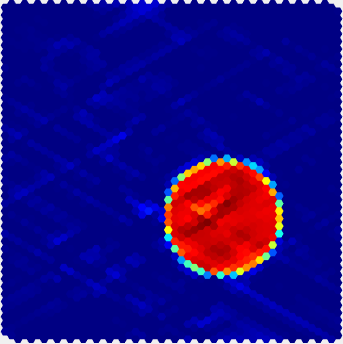};
    \end{axis}
\end{tikzpicture}
        \begin{tikzpicture}[scale=\scale]
    \begin{axis}[width=\width, height=\height, axis on top, scale only axis, xmin=\xmin, xmax=\xmax, ymin=\ymin, ymax=\ymax, xtick={\xmin,\pointsize*\xmax + (1-\pointsize)*\xmin,...,\xmax},ytick={\ymin,\pointsize*\ymax + (1-\pointsize)*\ymin,...,\ymax}, colormap/jet, colorbar,point meta min=0,point meta max=4.5,colorbar style={xshift=-7pt,width=9pt},title={$\mu_6$}]
    \addplot graphics [xmin=\xmin,xmax=\xmax,ymin=\ymin,ymax=\ymax]{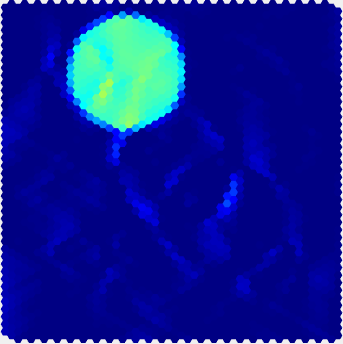};
    \end{axis}
\end{tikzpicture}

        \begin{tikzpicture}[scale=\scale]
    \begin{axis}[width=\width, height=\height, axis on top, scale only axis, xmin=\xmin, xmax=\xmax, ymin=\ymin, ymax=\ymax, xtick={\xmin,\pointsize*\xmax + (1-\pointsize)*\xmin,...,\xmax},ytick={\ymin,\pointsize*\ymax + (1-\pointsize)*\ymin,...,\ymax}, colormap/jet, colorbar,point meta min=12.5,point meta max=19.5,colorbar style={xshift=-7pt,width=9pt},title={$\mu_1$}]
    \addplot graphics [xmin=\xmin,xmax=\xmax,ymin=\ymin,ymax=\ymax]{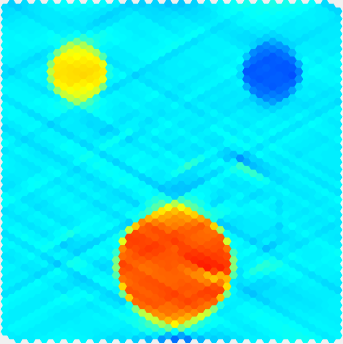};
    \end{axis}
\end{tikzpicture}
        \begin{tikzpicture}[scale=\scale]
    \begin{axis}[width=\width, height=\height, axis on top, scale only axis, xmin=\xmin, xmax=\xmax, ymin=\ymin, ymax=\ymax, xtick={\xmin,\pointsize*\xmax + (1-\pointsize)*\xmin,...,\xmax},ytick={\ymin,\pointsize*\ymax + (1-\pointsize)*\ymin,...,\ymax}, colormap/jet, colorbar,point meta min=12.5,point meta max=19.5,colorbar style={xshift=-7pt,width=9pt},title={$\mu_2$}]
    \addplot graphics [xmin=\xmin,xmax=\xmax,ymin=\ymin,ymax=\ymax]{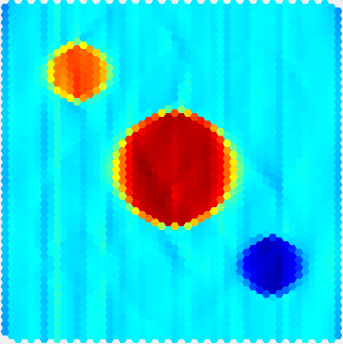};
    \end{axis}
\end{tikzpicture}
        \begin{tikzpicture}[scale=\scale]
    \begin{axis}[width=\width, height=\height, axis on top, scale only axis, xmin=\xmin, xmax=\xmax, ymin=\ymin, ymax=\ymax, xtick={\xmin,\pointsize*\xmax + (1-\pointsize)*\xmin,...,\xmax},ytick={\ymin,\pointsize*\ymax + (1-\pointsize)*\ymin,...,\ymax}, colormap/jet, colorbar,point meta min=12.5,point meta max=19.5,colorbar style={xshift=-7pt,width=9pt},title={$\mu_3$}]
    \addplot graphics [xmin=\xmin,xmax=\xmax,ymin=\ymin,ymax=\ymax]{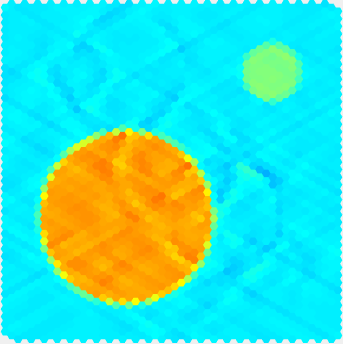};
    \end{axis}
\end{tikzpicture}
        \begin{tikzpicture}[scale=\scale]
    \begin{axis}[width=\width, height=\height, axis on top, scale only axis, xmin=\xmin, xmax=\xmax, ymin=\ymin, ymax=\ymax, xtick={\xmin,\pointsize*\xmax + (1-\pointsize)*\xmin,...,\xmax},ytick={\ymin,\pointsize*\ymax + (1-\pointsize)*\ymin,...,\ymax}, colormap/jet, colorbar,point meta min=0,point meta max=4.5,colorbar style={xshift=-7pt,width=9pt},title={$\mu_4$}]
    \addplot graphics [xmin=\xmin,xmax=\xmax,ymin=\ymin,ymax=\ymax]{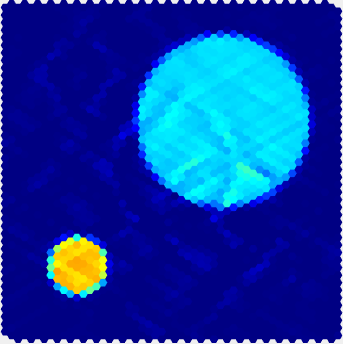};
    \end{axis}
\end{tikzpicture}
        \begin{tikzpicture}[scale=\scale]
    \begin{axis}[width=\width, height=\height, axis on top, scale only axis, xmin=\xmin, xmax=\xmax, ymin=\ymin, ymax=\ymax, xtick={\xmin,\pointsize*\xmax + (1-\pointsize)*\xmin,...,\xmax},ytick={\ymin,\pointsize*\ymax + (1-\pointsize)*\ymin,...,\ymax}, colormap/jet, colorbar,point meta min=0,point meta max=4.5,colorbar style={xshift=-7pt,width=9pt},title={$\mu_5$}]
    \addplot graphics [xmin=\xmin,xmax=\xmax,ymin=\ymin,ymax=\ymax]{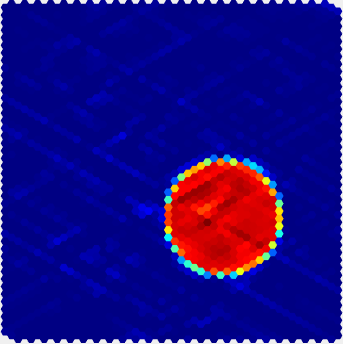};
    \end{axis}
\end{tikzpicture}
        \begin{tikzpicture}[scale=\scale]
    \begin{axis}[width=\width, height=\height, axis on top, scale only axis, xmin=\xmin, xmax=\xmax, ymin=\ymin, ymax=\ymax, xtick={\xmin,\pointsize*\xmax + (1-\pointsize)*\xmin,...,\xmax},ytick={\ymin,\pointsize*\ymax + (1-\pointsize)*\ymin,...,\ymax}, colormap/jet, colorbar,point meta min=0,point meta max=4.5,colorbar style={xshift=-7pt,width=9pt},title={$\mu_6$}]
    \addplot graphics [xmin=\xmin,xmax=\xmax,ymin=\ymin,ymax=\ymax]{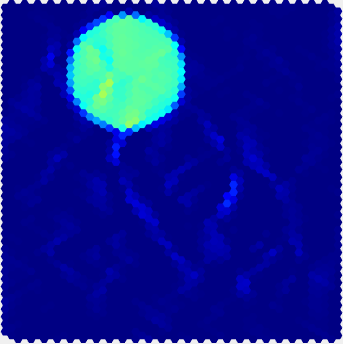};
    \end{axis}
\end{tikzpicture}

	\caption{Reconstruction of a fully anisotropic tensor from static measurements. From top to bottom: the exact parameters maps, and their reconstructions using respectively $m=3$, $4$ and $6$ data fields.The mesh resolution  $h=0.015$ for the inversion.}
	\label{fig:aniso-recon}
\end{center}
\end{figure}

\begin{figure}\begin{center}\begin{tikzpicture}[scale=\graphscale]\begin{semilogyaxis}
   [xmin=0, xmax=11, ymin=0.000005, ymax=0.006, grid = major, 
   xlabel = Eigenvalue number,
   legend entries = {\hspace{-.3cm}Number of,\hspace{-.3cm} used data, $m=3$, $m=4$, $m=5$, $m=6$,}, legend pos=outer north east, 
   ]
   \addlegendimage{empty legend}
   \addlegendimage{empty legend}
\addplot+[only marks] table{
1    6.0085e-06
2    2.979e-05
3    7.0848e-05
4    0.00010038
5    0.0001196
6    0.00025388
7    0.00026867
8    0.00028923
9    0.00031597
10    0.00033965
};\addplot+[only marks] table{
1    0.0005084
2    0.0015578
3    0.0022136
4    0.0022781
5    0.0024828
6    0.0025386
7    0.0026844
8    0.003033
9    0.003409
10    0.0035802
};\addplot+[only marks] table{
1    0.00064372
2    0.0027864
3    0.0035877
4    0.0038523
5    0.0039112
6    0.0039291
7    0.0042529
8    0.0042634
9    0.004274
10    0.0043444
};\addplot+[only marks] table{
1    0.00077064
2    0.0036467
3    0.0040266
4    0.0041242
5    0.0041436
6    0.0042933
7    0.0045211
8    0.0047947
9    0.0048607
10    0.0049189
};\end{semilogyaxis}\end{tikzpicture}\end{center}
\vspace{-0.5cm}
	\caption{The ten smallest eigenvalues of the matrix $\H$ $(\alpha_1,\dots,\alpha_{10})$ when having access to three (blue discs), four (red squares), five (brown discs) and six (black stars) independent displacements. The spectral gap is $\alpha_2-\alpha_1 \approx 2.4.10^{-5}$ in the first case, $\alpha_2-\alpha_1 \approx 1.1.10^{-3}$ in the second case, $\alpha_2-\alpha_1 \approx 2.1.10^{-3}$ in the third case and $\alpha_2-\alpha_1 \approx 2.9.10^{-3}$ in the fourth case.}
	\label{fig:aniso-spectral}
\end{figure}
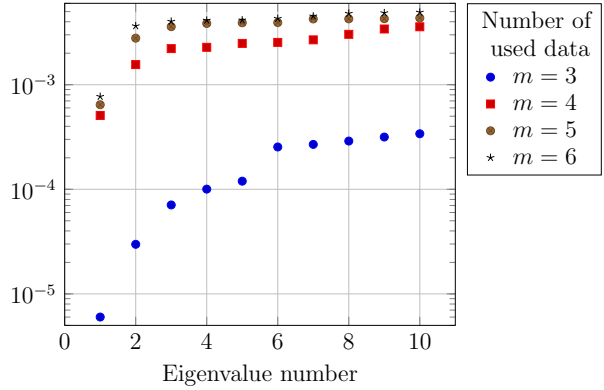

\begin{figure}
\centering
	\begin{tabular}{cc}
\toprule
Number of used data & Relative $L^2$-error (\%)   \\ 
\midrule
		 1 & 138.28  \\
		 2 & 140.31  \\
		 3 & 140.20  \\
		 4 & 2.29     \\
		 5 & 2.17      \\
		 6 & 2.14       \\
		 7 & 2.11     \\
		 8 & 2.11     \\
\bottomrule
\end{tabular}

	\caption{The relative $L^2$-error between the exact maps in the cas of the full anisotropic tensor recovery from various numbers of accessible data. The Discretization parameter is $h=0.015$ for the recovery.}
	\label{tab:aniso-error}
\end{figure}

\section{Concluding remarks}
In this work, we show that the Reverse Weak Formulation \eqref{eq:RWF} allows to see the multi-parameter inverse problem as a single linear inversion or a single eigenvalue problem. Under rather non-restrictive hypotheses, the involved linear operator has closed range and has a "small" null-space. This suggests that a suitable discretization of this operator produces stable numerical reconstructions of the parameter maps. Indeed, as observed in Section \ref{sec:num}, this works indeed for either isotropic or anisotropic tensors from both static and dynamic data. As the stability is preserved in the discrete problem, we do not use any regularization when performing the inversion. This approach provides a powerful alternative to classical methods for solving this kind of inverse problems that consist of minimizing a discrepancy functional, which could be very slow.

To characterize the null space, it has been necessary to introduce a very interesting extension of the notion of conservative tensor fields to third-order tensor field. This idea could have consequences on the theory of systems of first order PDEs.

\section*{Acknowledgement}
The authors acknowledge support from the French National Research Agency (ANR) under
grants ANR-22-CE40-0005 (project REWARD). Part of this work was also supported by the LABEX MILYON (ANR-10-LABX-0070) of Université de Lyon, within the program "Investissements d'Avenir" (ANR-11-IDEX- 0007) operated by the French National Research Agency (ANR).

\appendix
\makeatletter
\renewcommand{\thelemma}{\@Alph\c@section.\arabic{lemma}}
\renewcommand{\thetheorem}{\@Alph\c@section.\arabic{theorem}}
\renewcommand{\theproposition}{\@Alph\c@section.\arabic{proposition}}
\renewcommand{\thecorollary}{\@Alph\c@section.\arabic{corollary}}
\renewcommand{\theremark}{\@Alph\c@section.\arabic{remark}}
\renewcommand{\thedefinition}{\@Alph\c@section.\arabic{definition}}
\makeatother

\section{Technical results}

\subsection{Some products in Sobolev spaces}\label{sec:appendixA}

\begin{lemma}\label{prop:products} Let $\Omega$ be a open, bounded subset of $\R^d$. Let $u\in W^{1,p}(\Omega)$ with $p>d$, $\ph\in H^1_0(\Omega)$ and  $v\in H^{-1}(\Omega)$. Then,

\begin{enumerate}

\item The product $u\, \ph$ belongs to $H^1_0(\Omega)$ and we have $\norm{u\, \ph}{H^1_0(\Omega)}\leq c \norm{u}{W^{1,p}(\Omega)}\norm{\ph}{ H^1_0(\Omega)}$ where $c>0$ depends only on $p,d$ and $\Omega$.

\item  The product $u\, v$ belongs to$ H^{-1}(\Omega)$ and we have $\norm{u\, v}{H^{-1}(\Omega)} \leq c\norm{u}{W^{1,p}(\Omega)}\norm{v}{H^{-1}(\Omega)}$ where $c>0$ depends only on $p,d$ and $\Omega$. 

\end{enumerate}

\end{lemma}

\begin{proof} 1. From Sobolev embeddings, we have the bounded injections $W^{1,p}\hookrightarrow L^\infty$ and $H^1_0\hookrightarrow L^q$ for any $q$ such that $1/q>1/2-1/d$. We remark that $u\, \ph\in L^2$ and we write, 
\begin{equation}\nonumber\begin{aligned}
\norm{\nabla(u\, \ph)}{L^2} \leq \norm{u\nabla\ph}{L^2} + \norm{\ph\nabla u}{L^2}\leq \norm{u}{L^\infty}\norm{\nabla\ph}{L^2}  + \norm{\ph}{L^q}\norm{\nabla u}{L^p}
\end{aligned}\end{equation}
where $1/q=1/2-1/p>1/2-1/d$. then 
\begin{equation}\nonumber
\norm{\nabla(u\, \ph)}{L^2}\leq c \norm{u}{W^{1,p}}\norm{\ph}{ H^1_0},
\end{equation}
where $c>0$ depends only on $p,d$ and $\Omega$. Then $u\, \ph\in H^1$ and finally $u\, \ph\in H^1_0$ thanks to the continuity of the trace operator on $H^1$. 

2.  From the previous point, the distribution $u\, v$ is well-defined in $H^{-1}$ by
\begin{equation}\nonumber
\inner{u\, v,\ph}{H^{-1},H^1_0}:=\inner{v,u\, \ph}{H^{-1},H^1_0},\qquad \forall \ph\in H^1_0. 
\end{equation}
Moreover, 
\begin{equation}\nonumber\begin{aligned}
\inner{u\, v,\ph}{H^{-1},H^1_0}:=\inner{v,u\, \ph}{H^{-1},H^1_0}\leq \norm{v}{H^{-1}}\norm{u\, \ph}{H^1_0}\leq c\norm{v}{H^{-1}}\norm{u}{W^{1,p}}\norm{\ph}{H^1_0}
\end{aligned}\end{equation}
and then $\norm{uv}{H^{-1}} \leq c\norm{u}{W^{1,p}}\norm{v}{H^{-1}}$. 
\end{proof}

\subsection{Poincaré-type inequalities in $W^{1,p}(\Omega)$ with $p>d$.}

\begin{proposition}\label{prop:poincare1} Let $\Omega$ be a open, bounded, connected, Lipschitz domain of $\R^d$. Fix $p>d$, $n\in\N^*$ and let $ x_0 \in \overline{\Omega}$. There exists a constant $c_1>0$ such that
\begin{equation}\nonumber
\forall \g u\in W^{1,p}(\Omega,\R^n), \; \qquad \norm{\g u - \g u(x_0)}{L^p(\Omega)}\leq c_1\norm{\nabla \g u}{L^p(\Omega)}. 
\end{equation}
Moreover, there exists a constant $c_2>0$ such that
\begin{equation}\nonumber
\forall \g u\in W^{1,p}(\Omega,\R^n), \; \qquad \norm{\g u - \g u(x_0)}{L^\infty(\Omega)}\leq c_2\norm{\nabla \g u}{L^p(\Omega)}. 
\end{equation}

\end{proposition}
\begin{proof} By contradiction, we assume that there exits a sequence $(\g u_k)_k$ such that 
\begin{align*}
\norm{\g u_k- \g u_k (x_0)}{L^p(\Omega)}=1,
\end{align*}
and $\norm{\nabla \g u_k}{L^p(\Omega)}\to 0$. Denoting $\bm{v}_k := \g u_k - \g u_k (x_0)$, we subsequently have $\norm{\bm{v}_k}{L^p(\Omega)}=1$, $\bm{v}(x_0) = \bm{0}$ and $\norm{\nabla \bm{v}_k}{L^p(\Omega)}\to 0$. By classical compactness arguments, this sequence converge (up to a subsequence) weakly to $\bm{v}^*$ in $W^{1,p}(\Omega)$ and strongly in $L^p(\Omega)$. This limit satisfies $\nabla\bm{v}^*= \bm{0}$ and as $\{\bm{w}\in W^{1,p} \; | \;\bm{w}(x_0)=0 \}$ is closed in $W^{1,p}(\Omega)$, we also get that $\bm{v}^*(0)=0$. Finally $\bm{v}^*=\g 0$ which contradicts the hypothesis $\norm{\bm{v}_k}{L^p(\Omega)}=1$.  \medskip

The second inequality derives from the Sobolev embedding $W^{1,p}(\Omega,\R^d)\hookrightarrow L^\infty(\Omega,\R^d)$ as $p>d$, which reads: 
\begin{align*}
	\norm{\g u - \g u(0)}{L^\infty(\Omega)} \leq c \norm{\g u - \g u(0)}{W^{1,p}(\Omega)}\leq c(1+c_1^p)^{\frac 1p}\norm{\nabla\g u }{L^p(\Omega)}. 
\end{align*}
\end{proof}

\begin{corollary}\label{cor:poincare1} Let $B$ be the unit ball of $\R^d$ and fix $p>d$ and $n\in\N^*$. There exists a constant $c>0$ such that
\begin{equation}\nonumber
\forall \g u\in W^{1,p}(B,\R^n),\qquad \norm{\g u - \g u(0)}{L^\infty(B)}\leq c\norm{\nabla \g\mu}{L^p(B)},
\end{equation}
and for any $x_0 \in \mathbb{R}^d$, any $\e >0$, we have
\begin{equation}\nonumber
\forall \g\mu\in W^{1,p}(B(x_0,\e),\R^n),\qquad \norm{\g u - \g u(x_0)}{L^\infty(B(x_0,\e))}\leq c\, \e^{1-d/p}\norm{\nabla \g\mu}{L^p(B(x_0,\e))}. 
\end{equation}
\end{corollary}
\begin{proof} The first inequality is a consequence of the previous Lemma. Now, for any function $\g u\in W^{1,p}(B(x_0,\e),\R^n)$ we denote $\g u_\e(x):=\g u(\e x + x_0)$ and write
\begin{equation}\nonumber
\begin{aligned}
\norm{\g u - \g u(x_0)}{L^\infty(B(x_0,\e))}^p &= \norm{\g u_\e - \g u_\e(0)}{L^\infty(B)}^p \leq c_2^p\norm{\nabla\g u_\e }{L^p(B}^p = c_2^p\int_B\e^p|\nabla \g u|^p(\e x + x_0)\td x\\
&\leq c_2^p \e^{p-d}\int_B |\nabla \g u|^p(\e x + x_0)\e^d\td x = c_2^p \e^{p-d}\int_{B(0,\e)}|\nabla \g u|^p(x)\td x. 
\end{aligned}
\end{equation}
\end{proof}

\begin{definition}
	We define an open cone in $\R^d$ of vertex $x_0\in\R^d$ the set 
	\begin{align*}
		C(x_0,v) := \set{x_0 + x}{|x| < x\cdot v}
	\end{align*}
	for $v\in\R^d$ such that $0<|v|<1$.  It is known (see \cite{adams2003sobolev,hofmann2007geometric}) that a Lipschitz domain  $\Omega \subset \mathbb{R}^d$ satisfies the interior cone condition, saying that for any $x_0\in\partial\Omega$, there exists an open cone $C(x_0,v)$ and $\e >0$ such that 
	\begin{equation}\nonumber
	C(x_0,v)\cap B(x_0,\e)\subset \Omega
	\end{equation}
	for $\e>0$ small enough. 
\end{definition}

\begin{corollary}\label{cor:poincare2} Let $\Omega$ be a open, bounded, connected, Lipschitz domain of $\mathbb{R}^d$, fix $p>d$ and $n\in\N^*$. Let $x_0 \in\partial\Omega$. As $\Omega$ satisfies the cone condition, there exists an open cone $C(x_0,v)$ and $\e_0 >0$ such that for all $\e\leq \e_0$ we have $C(x_0,v)\cap B(x_0,\e) \subset \Omega$. Moreover, there exists a constant $c>0$ such that 
\begin{equation}\nonumber
\forall \g u\in W^{1,p}(\Omega,\R^n), \qquad \norm{\g u - \g u(x_0)}{L^{\infty}(C(x_0,v) \cap B(x_0,\e))} \leq c\, \e^{1-d/p}\norm{\nabla \g u}{L^p(C(x_0,v) \cap B(x_0,\e))}.
\end{equation}
\end{corollary}
\begin{proof}
	We rescale the domain to $C(0,v) \cap B(0,1)$ which is an open bounded and connected Lipschitz domain and apply \ref{prop:poincare1}. For any $\g u\in W^{1,p}(\Omega)$ we denote $\g u_\e(x):=\g u(\e\, x + x_0)$ and write
\begin{align*}
&\norm{\g u - \g u(x_0)}{L^\infty(C(x_0,v) \cap B(x_0,\e))}^p = \norm{\g u_\e - \g u_\e(0)}{L^\infty(C(0,v)\cap B(0,1))}^p \\
&\qquad\qquad  \leq c_3^p\norm{\nabla \g u_\e }{L^p(C(0,v) \cap B(0,1))}^p = c^p\int_{C(0,v) \cap B(0,1)}\e^p|\nabla \g u|^p(\e x + x_0)\td x\\
&\qquad\qquad \leq c^p \e^{p-d}\int_{C(0,v) \cap B(0,1)}|\nabla \g u|^p(\e x + x_0)\e^d\td x = c^p \e^{p-d}\int_{C(x_0,v) \cap B(x_0,\e)}|\nabla \g u|^p(x)\td x. 
\end{align*}
\end{proof}

\subsection{Null space of a sequence of operators limit} \label{app:approx}

\begin{lemma}\label{prop:lemmaLp-proof}Let $H$ be an Hilbert space and $F$ any normed vector space. Consider a sequence of operator $(A_\ell)_{\ell\in\N}$ of $\cL(H,F)$ that converges to $A\in\cL(H,F)$ for the operator norm. If
\begin{itemize}

\item[-] $A$ has closed range,

\item[-] $\exists k\in\N,\quad \forall\ell\in\N,\quad \dim N(A_\ell) \geq k$,

\end{itemize}
then  $\dim N(A)\geq k$. 
\end{lemma}

\begin{proof} If $\dim N(A) = \infty$, then the proof is complete. Assume now $\dim N(A)<+\infty$. For every $\ell \in \N$, we consider an orthonormal family $(v^1_\ell,...,v^k_\ell)$ of $N(A_\ell)$.
From the convergence of $A_\ell$ in $\cL(H,F)$, we have:
\begin{align*} 
\forall i \in \{1,...,k \}, \quad \| Av^i_\ell \|_F = \| (A_\ell - A) v^i_\ell \|_F \le \|A_\ell - A \|_{H,F} \underset{\ell \to \infty}{\to} 0.
\end{align*}
For each $1\leq i\leq k$ and $\ell\in\N$ we decompose $v^i_\ell$ as $v^i_\ell{} = v^i_\ell{}^\perp + v^i_\ell{}^N$ with $v^i_\ell{}^N\in N(A)$ and $v^i_\ell{}^\perp\in N(A)^\perp$. The fact that $A$ has closed range implies that all the orthogonal parts converge to zero:
\begin{align}\nonumber
\|v^i_\ell{}^\perp \|_H \le C \|A v^i_\ell{}^\perp \|_F \le C \|A v^i_\ell \|_F  \underset{\ell \to \infty}{\to} 0,
\end{align}
where the constant $C$ is the closed range constant of the operator $A$. We now recall that each sequence $(v^i_\ell{}^N)_\ell$ is bounded in the finite dimensional space $N(A)$. By compactness, up to an subsequence, these sequences converge to some $v^i\in N(A)$. Then, up to a subsequence, we have that 
\begin{equation}\nonumber
\forall i\in\{1,\dots,k\},\quad v^i_\ell  \overset{H}{\underset{\ell \to \infty}{\lra}} v^i\in N(A). 
\end{equation}
The limit family $(v^1,\dots,v^k)$ is also orthonormal as $\norm{v^i}{H}=\lim_{\ell\to +\infty}\norm{v^i_\ell}{H}=1$ and for any $i\neq j$, by continuity of the inner product, we have 
\begin{equation}\nonumber
\inner{v^i,v^j}{H} = \lim_{\ell\to +\infty}\inner{v^i_\ell,v^j_\ell}{H} = 0.
\end{equation}
Therefore, we have constructed an orthonormal family of $k$ vectors of $N(A)$, and this implies that $\dim N(A)\geq k$. 
\end{proof}

\begin{remark}
The closed range hypothesis in \ref{prop:lemmaLp} is fundamental. Here is a counterexample where $A$ does not have closed range. Let $H$ be a seperable Hilbert space and $(e_n)_{n\geq 1}$ an Hilbert basis. Wet set
\begin{equation}\nonumber
A:=\sum_{n\geq 1}\frac{1}{n}e_n\otimes e_n\quad\tand\quad A_\ell := A - \frac{1}{\ell} e_\ell\otimes e_\ell,\quad \forall \ell\geq 1.
\end{equation}
We see that $A_\ell\to A$ in $\cL(H)$, $N(A)=\{0\}$ and $N(A_\ell) = \vspan(e_\ell)$ for all $\ell\geq 1$.  
\end{remark}

\bibliography{biblio.bib}

\begin{thebibliography}{10}

\bibitem{adams2003sobolev}
R.~A. Adams and J.~J. Fournier.
\newblock {\em Sobolev spaces}.
\newblock Elsevier, 2003.

\bibitem{ammari2015mathematical}
H.~Ammari, E.~Bretin, J.~Garnier, H.~Kang, H.~Lee, and A.~Wahab.
\newblock {\em Mathematical methods in elasticity imaging}.
\newblock Princeton University Press, 2015.

\bibitem{ammari2021direct}
H.~Ammari, E.~Bretin, P.~Millien, and L.~Seppecher.
\newblock A direct linear inversion for discontinuous elastic parameters
  recovery from internal displacement information only.
\newblock {\em Numerische Mathematik}, 147(1):189--226, 2021.

\bibitem{ammari2015stability}
H.~Ammari, A.~Waters, and H.~Zhang.
\newblock Stability analysis for magnetic resonance elastography.
\newblock {\em Journal of Mathematical Analysis and Applications},
  430(2):919--931, 2015.

\bibitem{babaniyi2017direct}
O.~A. Babaniyi, A.~A. Oberai, and P.~E. Barbone.
\newblock Direct error in constitutive equation formulation for plane stress
  inverse elasticity problem.
\newblock {\em Computer methods in applied mechanics and engineering},
  314:3--18, 2017.

\bibitem{bal2014}
G.~Bal, C.~Bellis, S.~Imperiale, and F.~Monard.
\newblock Reconstruction of constitutive parameters in isotropic linear
  elasticity from noisy full-field measurements.
\newblock {\em Inverse Problems}, 30(12):125004, Oct. 2014.

\bibitem{bal2015reconstruction}
G.~Bal, F.~Monard, and G.~Uhlmann.
\newblock Reconstruction of a fully anisotropic elasticity tensor from
  knowledge of displacement fields.
\newblock {\em SIAM Journal on Applied Mathematics}, 75(5):2214--2231, 2015.

\bibitem{barbone2004}
P.~E. Barbone and N.~H. Gokhale.
\newblock Elastic modulus imaging: on the uniqueness and nonuniqueness of the
  elastography inverse problem in two dimensions.
\newblock {\em Inverse Problems}, 20(1):283, jan 2004.

\bibitem{barbone2007elastic}
P.~E. Barbone and A.~A. Oberai.
\newblock Elastic modulus imaging: some exact solutions of the compressible
  elastography inverse problem.
\newblock {\em Physics in Medicine and Biology}, 52(6):1577, 2007.

\bibitem{barbone2010adjoint}
P.~E. Barbone, C.~E. Rivas, I.~Harari, U.~Albocher, A.~A. Oberai, and Y.~Zhang.
\newblock Adjoint-weighted variational formulation for the direct solution of
  inverse problems of general linear elasticity with full interior data.
\newblock {\em International journal for numerical methods in engineering},
  81(13):1713--1736, 2010.

\bibitem{tanterFink2003}
J.~Bercoff, S.~Chaffai, M.~Tanter, L.~Sandrin, S.~Catheline, M.~Fink, J.-L.
  Gennisson, and M.~Meunier.
\newblock In vivo breast tumor detection using transient elastography.
\newblock {\em Ultrasound In Medicine and Biology}, 29:1387--1396, 10 2003.

\bibitem{bretin2023stability}
E.~Bretin, P.~Millien, and L.~Seppecher.
\newblock Stability for finite element discretization of some inverse parameter
  problems from internal data: Application to elastography.
\newblock {\em SIAM Journal on Imaging Sciences}, 16(1):340--367, 2023.

\bibitem{brusseau20082}
E.~Brusseau, J.~Kybic, J.-F. D{\'e}prez, and O.~Basset.
\newblock 2-d locally regularized tissue strain estimation from radio-frequency
  ultrasound images: Theoretical developments and results on experimental data.
\newblock {\em IEEE Transactions on Medical Imaging}, 27(2):145--160, 2008.

\bibitem{brusseau2000axial}
E.~Brusseau, C.~Perrey, P.~Delachartre, M.~Vogt, D.~Vray, and H.~Ermert.
\newblock Axial strain imaging using a local estimation of the scaling factor
  from rf ultrasound signals.
\newblock {\em Ultrasonic imaging}, 22(2):95--107, 2000.

\bibitem{Brusseau_Seppecher1}
E.~Brusseau, L.~Petrusca, E.~Bretin, P.~Millien, and L.~Seppecher.
\newblock Reconstructing the shear modulus contrast of linear elastic and
  isotropic media in quasi-static ultrasound elastography.
\newblock In {\em 2021 43rd Annual International Conference of the IEEE
  Engineering in Medicine \& Biology Society (EMBC)}, pages 3177--3180. IEEE,
  2021.

\bibitem{catheline1998}
S.~Catheline.
\newblock {\em Interférométrie Speckle ultrasonore : application à la mesure
  d'élasticité}.
\newblock PhD thesis, 1998.
\newblock Thèse de doctorat dirigée par Wu, François Acoustique Physique
  Paris 7 1998.

\bibitem{deprez2011potential}
J.-F. Deprez, E.~Brusseau, J.~Fromageau, G.~Cloutier, and O.~Basset.
\newblock On the potential of ultrasound elastography for pressure ulcer early
  detection.
\newblock {\em Medical physics}, 38(4):1943--1950, 2011.

\bibitem{doyley2012model}
M.~Doyley.
\newblock Model-based elastography: a survey of approaches to the inverse
  elasticity problem.
\newblock {\em Physics in Medicine and Biology}, 57(3):R35, 2012.

\bibitem{ferreira2016calderon}
D.~D.~S. Ferreira, Y.~Kurylev, M.~Lassas, and M.~Salo.
\newblock The calder{\'o}n problem in transversally anisotropic geometries.
\newblock {\em Journal of the European Mathematical Society},
  18(11):2579--2626, 2016.

\bibitem{gennisson2013ultrasound}
J.-L. Gennisson, T.~Deffieux, M.~Fink, and M.~Tanter.
\newblock Ultrasound elastography: Principles and techniques.
\newblock {\em Diagnostic and Interventional Imaging}, 94(5):487--495, 2013.
\newblock Ultrasound elastography.

\bibitem{ghosh2020calderon}
T.~Ghosh, M.~Salo, and G.~Uhlmann.
\newblock The calder{\'o}n problem for the fractional schr{\"o}dinger equation.
\newblock {\em Analysis \& PDE}, 13(2):455--475, 2020.

\bibitem{hofmann2007geometric}
S.~Hofmann, M.~Mitrea, and M.~Taylor.
\newblock Geometric and transformational properties of lipschitz domains,
  semmes-kenig-toro domains, and other classes of finite perimeter domains.
\newblock {\em The Journal of Geometric Analysis}, 17(4):593--647, 2007.

\bibitem{huang1995thermal}
C.-H. Huang and Y.~Jan-Yuan.
\newblock An inverse problem in simultaneously measuring temperature-dependent
  thermal conductivity and heat capacity.
\newblock {\em International Journal of Heat and Mass Transfer},
  38(18):3433--3441, 1995.

\bibitem{krouskop1998elastic}
T.~A. Krouskop, T.~M. Wheeler, F.~Kallel, B.~S. Garra, and T.~Hall.
\newblock Elastic moduli of breast and prostate tissues under compression.
\newblock {\em Ultrasonic imaging}, 20(4):260--274, 1998.

\bibitem{mclaughlin2004unique}
J.~R. McLaughlin and J.-R. Yoon.
\newblock Unique identifiability of elastic parameters from time-dependent
  interior displacement measurement.
\newblock {\em Inverse Problems}, 20(1):25, nov 2003.

\bibitem{mclaughlin2010calculating}
J.~R. McLaughlin, N.~Zhang, and A.~Manduca.
\newblock Calculating tissue shear modulus and pressure by 2d log-elastographic
  methods.
\newblock {\em Inverse Problems}, 26(8):085007, 2010.

\bibitem{munoz2020calderon}
C.~Munoz and G.~Uhlmann.
\newblock The calder{\'o}n problem for quasilinear elliptic equations.
\newblock In {\em Annales de l'Institut Henri Poincar{\'e} C, Analyse non
  lin{\'e}aire}, volume~37, pages 1143--1166. Elsevier, 2020.

\bibitem{nahas2013supersonic}
A.~Nahas, M.~Tanter, T.-M. Nguyen, J.-M. Chassot, M.~Fink, and
  A.~Claude~Boccara.
\newblock From supersonic shear wave imaging to full-field optical coherence
  shear wave elastography.
\newblock {\em Journal of biomedical optics}, 18(12):121514--121514, 2013.

\bibitem{goksel2018}
C.~F. Otesteanu, V.~Vishnevsky, and O.~Goksel.
\newblock Fem-based elasticity reconstruction using ultrasound for imaging
  tissue ablation.
\newblock {\em International Journal of Computer Assisted Radiology and Surgery
  | Issue 6/2018}, 6(1):885--894, 2018.

\bibitem{parker2010imaging}
K.~J. Parker, M.~M. Doyley, and D.~J. Rubens.
\newblock Imaging the elastic properties of tissue: the 20 year perspective.
\newblock {\em Physics in Medicine and Biology}, 56(1):R1, 2010.

\bibitem{pierron2012virtual}
F.~Pierron and M.~Gr{\'e}diac.
\newblock {\em The virtual fields method: extracting constitutive mechanical
  parameters from full-field deformation measurements}.
\newblock Springer Science \& Business Media, 2012.

\bibitem{tanterFink2002}
L.~Sandrin, M.~Tanter, S.~Catheline, and M.~Fink.
\newblock Shear modulus imaging with 2-d transient elastography.
\newblock {\em IEEE Transactions on Ultrasonics, Ferroelectrics, and Frequency
  Control}, 49(4):426--435, 2002.

\bibitem{sarvazyan1995biophysical}
A.~Sarvazyan, A.~Skovoroda, S.~Emelianov, J.~Fowlkes, J.~Pipe, R.~Adler,
  R.~Buxton, and P.~Carson.
\newblock Biophysical bases of elasticity imaging.
\newblock {\em Acoustical imaging}, pages 223--240, 1995.

\bibitem{Brusseau_Seppecher2}
L.~Seppecher, E.~Bretin, P.~Millien, L.~Petrusca, and E.~Brusseau.
\newblock Reconstructing the spatial distribution of the relative shear modulus
  in quasi-static ultrasound elastography: plane stress analysis.
\newblock {\em Ultrasound in Medicine \& Biology}, 49(3):710--722, 2023.

\bibitem{sherina2021challenges}
E.~Sherina, L.~Krainz, S.~Hubmer, W.~Drexler, and O.~Scherzer.
\newblock Challenges for optical flow estimates in elastography.
\newblock In {\em International Conference on Scale Space and Variational
  Methods in Computer Vision}, pages 128--139. Springer, 2021.

\bibitem{uhlmannCalderon}
G.~Uhlmann.
\newblock 30 years of calderón’s problem.
\newblock {\em Séminaire Laurent Schwartz — EDP et applications}, pages
  1--25, 01 2012.

\bibitem{widlak2015}
T.~Widlak and O.~Scherzer.
\newblock Stability in the linearized problem of quantitative elastography.
\newblock {\em Inverse Problems}, 31(3):035005, Feb. 2015.

\end{thebibliography}

\end{document}